\documentclass[11pt]{amsart}

\usepackage{stmaryrd}
\usepackage{amscd,amsxtra,amssymb,mathrsfs, bbm}

\usepackage{multirow}
\usepackage[pdftex]{hyperref}
\usepackage{longtable}
\usepackage{tikz}
\usepackage[utf8]{inputenc}
\usepackage{float}
\usepackage{color}
\usetikzlibrary{calc, positioning, shapes, fit, matrix, decorations}
\usetikzlibrary{decorations.shapes}

\usepackage[cmtip]{xypic}
\usepackage[all]{xy}
\xyoption{arc}

\newtheorem{theorem}{Theorem}[section]
\newtheorem{corollary}[theorem]{Corollary}
\newtheorem{lemma}[theorem]{Lemma}
\newtheorem{proposition}[theorem]{Proposition}

\theoremstyle{definition}
\newtheorem{definition}[theorem]{Definition}
\newtheorem{remark}[theorem]{Remark}
\newtheorem{example}[theorem]{Example}

\numberwithin{equation}{section}

\DeclareMathAlphabet{\mathpzc}{OT1}{pzc}{m}{it}

\DeclareMathOperator{\TF}{\mathsf{TF}}

\DeclareMathOperator{\rk}{rk}

\renewcommand{\ker}{\mathsf{Ker}}

\renewcommand{\dim}{\mathsf{dim}}
\DeclareMathOperator{\tor}{\mathsf{tor}}

\DeclareMathOperator{\Coh}{\mathsf{Coh}}

\DeclareMathOperator{\Tri}{\mathsf{Tri}}
\DeclareMathOperator{\Pic}{\mathsf{Pic}}

\DeclareMathOperator{\CM}{\mathsf{CM}}

\DeclareMathOperator{\Pro}{\mathsf{Pro}}

\DeclareMathOperator{\depth}{\mathsf{depth}}

\DeclareMathOperator{\Hom}{\mathsf{Hom}}

\DeclareMathOperator{\Aut}{\mathsf{Aut}}

\DeclareMathOperator{\End}{\mathsf{End}}
\DeclareMathOperator{\Mat}{\mathsf{Mat}}
\DeclareMathOperator{\Spec}{\mathsf{Spec}}

\input xy
\xyoption{all}

\def\bu{{\scriptscriptsize\bullet}}
\def\astar{{\scriptsize\ast}}

\newcommand{\kk}{\mathbbm{k}}

\newcommand{\llbrace}{(\!(}
\newcommand{\rrbrace}{)\!)}

\newcommand{\FF}{\mathbb{F}}

\newcommand{\LL}{\mathbb{L}}

\def\bu{{\scriptscriptstyle\bullet}}

\newcommand{\kB}{\mathcal{B}}

\newcommand{\kF}{\mathcal{F}}

\newcommand{\kO}{\mathcal{O}}
\newcommand{\kL}{\mathcal{L}}
\newcommand{\kP}{\mathcal{P}}

\newcommand{\lar}{\longrightarrow}

\newcommand{\gB}{\mathfrak{B}}

\newcommand{\gD}{\mathfrak{D}}
\newcommand{\gE}{\mathfrak{E}}

\newcommand{\gP}{\mathfrak{S}}
\newcommand{\gO}{\mathfrak{W}}

\newcommand{\rQ}{Q}

\newcommand{\RR}{\mathbb R}
\newcommand{\CC}{\mathbb C}

\newcommand{\idm}{\mathfrak{m}}
\newcommand{\idn}{\mathfrak{n}}
\newcommand{\idp}{\mathfrak{p}}
\newcommand{\idq}{\mathfrak{q}}

\renewcommand{\AA}{\mathbb{A}}

\newcommand{\ZZ}{\mathbb{Z}}

\newcommand{\DD}{\mathbb{D}}

			\def\bu{\bullet}

\def\Mat{\mathop\mathrm{Mat}}

\def\8{\infty}			
	\def\+{\oplus}		
\def\*{\otimes}

\def\kB{\mathcal B} \def\kO{\mathcal O}
 \def\kP{\mathcal P}

\def\kF{\mathcal F} 
 
 \def\kU{\mathcal U}

\def\kL{\mathcal L}

	\def\NN{\mathbb N}

\def\bu{{\scriptscriptstyle\bullet}}

\def\DMO{\DeclareMathOperator}
\DMO{\ob}{Ob}            \DMO{\mor}{Mor}
\DMO{\Ker}{Ker}
\DMO{\id}{Id}

\title[Cohen--Macaulay modules and Calogero--Moser systems]{Cohen--Macaulay modules over the algebra of planar  quasi--invariants  and Calogero--Moser systems}

\author{Igor Burban}
\address{
Universit\"at Paderborn,
Institut f\"ur Mathematik,
Warburger Strasse 100,
33098 Paderborn,
Germany
}
\email{burban@math.uni-paderborn.de}

\author{Alexander Zheglov}
\address{
Moscow State University,
Faculty of Mechanics and Mathematics, Leninskie gory, GSP-1, Moscow, 119899, Russian Federation
}
\email{azheglov@math.msu.su}

\begin{document}

\begin{abstract} In this paper, we study  properties of the  algebras of planar  quasi--invariants. These algebras  are  Cohen--Macaulay and Gorenstein in codimension one.
Using the technique of matrix problems, we classify all Cohen--Macaulay modules of rank one over them  and determine  their Picard groups.
In terms of this classification,  we describe the spectral modules of the planar rational Calogero--Moser systems. Finally, we elaborate  the theory of the algebraic inverse scattering method, providing explicit computations of certain "isospectral deformations" of the planar rational Calogero--Moser system in the case of the split rational potential.
\end{abstract}

\maketitle

\section{Introduction}
\noindent A prototype for our work is given by the following setting.
For any $m \in \NN$, consider the so--called rational Lam\'e operator
$
L_m := \dfrac{d}{dx^2} - \dfrac{m(m+1)}{x^2}.
$
It is well--known (see e.g. \cite{Chalykh,EtingofStrickland}) that there exists a differential  operator $P_m$ of order $2m+1$ with meromorphic coefficients, such that
$
\bigl[P_m, L_m\bigr] = 0$ and $P_m^2 = L_m^{2m+1}$.

\smallskip
\noindent
Let $A_m := \CC[P_m, L_m] \cong \CC[t^2, t^{2m+1}]$. Then the algebra $A_m$ is Gorenstein and the following results are true (see for instance Theorem \ref{T:Cusp}):
\begin{itemize}
\item There exists an isomorphism of algebraic groups:
\begin{equation}\label{E:PicardCusp}
\Pic(A_m) \cong K_m := \bigl(\CC[\sigma]/(\sigma^m), \circ\bigr),
\end{equation} where $\Pic(A_m)$ is the Picard group of $A_m$ and
$\gamma_1 \circ \gamma_2 := (\gamma_1+ \gamma_2)\cdot (1 + \sigma \gamma_1 \gamma_2)^{-1}$ for any $\gamma_1, \gamma_2 \in K_m$.
\item Let $Q$ be a torsion free $A_m$--module of rank one. Then either $Q$ is projective or there exists $m' < m$ and a projective module $A_{m'}$--module $Q'$ of rank one  such that $Q$ is isomorphic to  $Q'$ viewed as a module over  $A_m  \subset A_{m'}$.
\end{itemize}
One goal of this work is  to generalize the described picture on the two--dimensional case. As an input datum, we take
 any pair  $(\Pi, \underline{\mu})$ (called \emph{weighted line arrangement}), where
\begin{itemize}
\item
$\Pi \subset \CC$ is a \emph{finite} subset satisfying the condition: $\alpha-\beta \notin \pi \ZZ$ for any $\alpha \ne \beta \in \Pi$.
\item $\Pi \stackrel{\underline{\mu}}\lar \NN_0, \; \alpha \mapsto \mu_\alpha:= \underline{\mu}(\alpha)$ is any \emph{multiplicity function}.
\end{itemize}
For any $\alpha \in \Pi$, let us  denote:
$
l_\alpha(z_1, z_2) :=  -\sin(\alpha) z_1  + \cos(\alpha) z_2 \in R:= \CC[z_1, z_2].
$
The main object of our paper is the following $\CC$--algebra  of  $(\Pi, \underline{\mu})$--quasi--invariant polynomials:
\begin{equation}\label{E:quasiinvariants}
A = A\bigl(\Pi, \underline{\mu}\bigr):=
\Bigl\{
f \in R \; \big| \; l_\alpha^{2\mu_\alpha +1} \;\mbox{divides}\;  \bigl(f  - s_\alpha(f)\bigr) \; \mbox{for all}\; \alpha \in \Pi
\Bigr\},
\end{equation}
where $R \stackrel{s_\alpha}\lar R$ is  the involution associated with the reflection
$\CC^2 \lar \CC^2$, which  keeps  the line  $l_\alpha = 0$ invariant.
It is not difficult to show  that $A$ (called \emph{algebra of planar quasi--invariants}) is a finitely generated  $\CC$--algebra  of Krull dimension two.

\smallskip
\noindent
A  motivation to study ring-- and module--theoretic properties of the algebra $A$ comes from the theory
of rational Calogero--Moser systems. Assume that $\bigl(\Pi, \underline{\mu}\bigr) = \bigl(\Lambda_n, m\bigr)$ is a \emph{Coxeter} weighted line arrangement, i.e.
$$\Pi = \Lambda_n:= \bigl\{0, \frac{1}{n} \pi, \dots, \frac{n-1}{n} \pi \bigr\} \subset \RR\; \;  \mbox{\rm for some}\; \;  n \in \NN,
$$ whereas  $\mu_\alpha = m$ for all $\alpha \in \Pi$ and some $m \in \NN$.
For any  vector  $\vec\xi \in \CC^2$ such that $l_\alpha(\vec\xi) \ne 0$ for each $\alpha \in \Lambda_n$,  consider the following \emph{rational Calogero--Moser} operator
\begin{equation}\label{E:CMHamiltonian}
H = H\bigl(\bigl(\Lambda_n, m\bigr); \vec\xi \bigr) :=
\left(\frac{\partial^2}{\partial x_1^2} + \frac{\partial^2}{\partial x_2^2}\right)
- \sum\limits_{\alpha \in \Pi} \dfrac{m(m+1)}{l_\alpha^2(\vec{x} - \vec\xi)}
\end{equation}
According to  a result of Chalykh and Veselov \cite{ChalykhVeselov90} (who proved a much more general statement for the real Coxeter groups of arbitrary rank),
there exists an injective algebra homomorphism (defining a so--called algebraically integrable quantum system)
\begin{equation}\label{E:IntegrabilityCM}
A\bigl(\Lambda_n, m\bigr) \stackrel{\Xi(\vec\xi)}\lar \gD := \CC\llbracket x_1, x_2\rrbracket[\partial_1, \partial_2].
\end{equation}
mapping the polynomial $z_1^2 + z_2^2 \in A$ to the operator $H$. In other words, $H$ can be included into a large family of pairwise commuting differential operators (\emph{quantum integrability}).

It was proven by Feigin and Veselov \cite{FeiginVeselov}, that in the Coxeter case, the algebra  $A$ is Gorenstein
(hence Cohen--Macaulay). This result was vastly generalized by Etingof and Ginzburg \cite{EtingofGinzburg} on the case of arbitrary real Coxeter groups
and  multiplicity  functions, invariant under the action of the Weyl group. In \cite{FeiginVeselov2, FeiginJohnston}, the authors  proved that the algebra $A\bigl(\Pi, \underline{\mu}\bigr)$ is Gorenstein in several
non--Coxeter cases.

Let $\gB := \mathrm{Im}\bigl(\Xi(\vec\xi)\bigr) \subset \gD$.  The $\gB$--module
$F:= \gD/(x_1, x_2)\gD \cong \CC[\partial_1, \partial_2]$ (called \emph{spectral module} of $\gB$) is the key object relating algebraic and  analytic tools in the study of  Calogero--Moser systems.
Combining  \cite[Corollary 3.1]{Chalykh} with \cite[Theorem 3.1]{KurkeZheglov}, one can conclude that
$F$ is a \emph{Cohen--Macaulay} $\gB$--module   \emph{of rank one}.  The  analytic meaning  of   $F$  can be illustrated by the following fact. For any algebra homomorphism $\gB \stackrel{\chi}\lar \CC$, consider the vector space
\begin{equation}\label{E:solutionspace}
\mathsf{Sol}\bigl(\gB, \chi\bigr):= \Bigl\{f\in \CC\llbracket x_1, x_2\rrbracket \big| P\diamond f = \chi(P) f \; \mbox{for all}\; P \in \gB\Bigr\},
\end{equation}
where $\diamond$ denotes the usual action of $\gD$ on $\CC\llbracket x_1, x_2\rrbracket$.
Then there exists  a canonical isomorphism of vector spaces (see Theorem \ref{T:spectralmodule}) $$F\big|_{\chi} := F \otimes_{\gB} \bigl(\gB/\mathrm{Ker}(\chi)\bigr) \cong \mathsf{Sol}\bigl(\gB, \chi\bigr)^\ast,$$
explaining why the $A$--module $F$ is called spectral. The statement that $F$ has rank one can be rephrased  by saying that the vector space $\mathsf{Sol}\bigl(\gB, \chi\bigr)$ is one--dimensional for a ``generic'' character $\chi$, i.e.~the Calogero--Moser system (\ref{E:IntegrabilityCM}) is \emph{algebraically integrable}; see \cite{ChalykhVeselov90, ChalykhStyrkasVeselov, Chalykh}.

In their recent paper  \cite[Section 8]{FeiginJohnston}, Feigin and Johnston raised a question about an \emph{explicit} description of the spectral module $F$ for the two--dimensional rational Calogero--Moser systems. This leads to the problem of classification of  \emph{all} Cohen--Macaulay modules of rank one over an arbitrary  algebra
of planar quasi--invariants $A = A\bigl(\Pi, \underline{\mu}\bigr)$. Another natural problem is to describe the Picard group $\Pic(A)$ of $A$. Obviously, the finitely generated projective $A$--modules form a proper subcategory of the category of Cohen--Macaualy $A$--modules. It turned out that it is in fact easier first to describe
all  Cohen--Macaulay $A$--modules of rank one and then specify those of them which are locally free.

\smallskip
\noindent
Now, let us give a short overview of  main  results, which were  obtained in our work.

\medskip
\noindent
\textsl{1.~Ring--theoretic properties of the algebra of planar quasi--invariants}. For any weighted line arrangement $\bigl(\Pi, \underline{\mu}\bigr)$,  the algebra $A= A\bigl(\Pi, \underline{\mu}\bigr)$ is a finitely generated,   Cohen--Macaulay  and  of Krull dimension two. The main new feature  about the algebra $A$ is the result asserting that  it  is Gorenstein in codimension one; see Theorem \ref{T:CMandGorenstein}.

\smallskip
\noindent
\textsl{2.~The divisor class group of $A$}.~Let $\CM_1^{\mathrm{lf}}(A)$ be the abelian group of Cohen--Macaulay $A$--modules of rank one, which are \emph{locally free in codimension one} \cite{SurvOnCM, BDNonIsol} (it is an analogue of the divisor class group of a normal domain \cite[Section 7.3]{BrunsHerzog}). Then there exists  an isomorphism of abelian groups
    $$
    \CM_1^{\mathrm{lf}}(A) \lar K\bigl(\Pi, \underline{\mu}\bigr):= \prod\limits_{\alpha \in \Pi} \bigl(\CC(\rho)[\sigma]/(\sigma^{\mu_\alpha}), \circ\bigr),
    $$
    where the group law $\circ$ on $\CC(\rho)[\sigma]/(\sigma^{\mu_\alpha})$ is the same as in the case of the one--dimensional cuspidal curves (\ref{E:PicardCusp}); see
    Theorem \ref{T:listCMlf}.

\smallskip
\noindent
\textsl{3.~Cohen--Macaulay $A$--modules of rank one}. Let $M \in \CM_1(A)$ be a Cohen--Macaulay $A$--module of rank one, which is not locally free in the codimension one.  Then there exists a multiplicity function
$\Pi \stackrel{\underline{\mu}'}\lar \NN_0$ satisfying  $\mu_\alpha \ge  \mu'_\alpha$ for any $\alpha \in \Pi$  and $M' \in
\CM_1^{\mathrm{lf}}(A')$ for $A'= A(\Pi, \underline{\mu}')$ such that $M$ is isomorphic to  $M'$, where $M'$ is viewed as a module over $A \subset A'$; see Corollary \ref{C:remarkClassifCM}.

\smallskip
\noindent
\textsl{4.~Description of a canonical  module of $A$}. In the Coxeter case $\Pi = \bigl\{0, \frac{1}{n} \pi, \dots, \frac{n-1}{n} \pi \bigr\}$, we get an explicit description  of a  canonical  module
of $A$ for an arbitrary multiplicity function $\underline{\mu}$; see Theorem \ref{T:canonicalmodule} and Lemma \ref{L:mixedweight}.

\smallskip
\noindent
\textsl{5.~The Picard group of $A$}.~For any $\alpha \in \CC$ and $m \in \NN$,
we construct  a certain homomorphism of abelian groups
$
\bigl(\CC\llbracket z_1, z_2\rrbracket, +\bigr) \xrightarrow{\Upsilon_{(\alpha, m)}} \bigl(\CC\llbracket \rho\rrbracket[\sigma]/(\sigma^m), \circ\bigr)
$
(see Lemma \ref{L:prepCompPicard}),
defining    a  homomorphism
$$
\bigl(\CC\llbracket z_1, z_2\rrbracket, +\bigr)
 \xrightarrow{\Upsilon} \prod\limits_{\alpha \in \Pi} \bigl(\CC\llbracket \rho\rrbracket[\sigma]/(\sigma^{\mu_\alpha}),  \circ\bigr), \quad h \mapsto \bigl(\Upsilon_{(\alpha, \mu_\alpha)}(h)\bigr)_{\alpha \in \Pi}.
$$
 In these terms we have: $\Pic(A) \cong \mathrm{Im}(\Upsilon) \cap K^\diamond\bigl(\Pi, \underline{\mu}\bigr)$, where $K^\diamond\bigl(\Pi, \underline{\mu}\bigr) := \prod\limits_{\alpha \in \Pi} \CC[\rho][\sigma]/(\sigma^{\mu_\alpha})$; see Theorem \ref{T:descriptionPicard}.

\smallskip
\noindent
More explicitly, let
$
\Gamma\bigl(\Pi, \underline{\mu}\bigr) := \bigl\{h \in \CC\llbracket z_1, z_2\rrbracket  \big|  \Upsilon(h) \in K^\diamond\bigl(\Pi, \underline{\mu}\bigr)\bigr\}$. Then for any such $h \in \Gamma\bigl(\Pi, \underline{\mu}\bigr)$, we have a projective $A$--module $P(h)$ of rank one, defined as
$$
P(h) := \bigl\{f \in  R \,\big|\, \exp(h) f \;\mbox{is} \; (\Pi, \underline{\mu})-\mbox{quasi--invariant}\bigr\}.
$$
Conversely,  for any $P \in \Pic(A)$, there exists $h \in \Gamma\bigl(\Pi, \underline{\mu}\bigr)$ such that $P \cong P(h)$. Moreover,
\begin{itemize}
\item $P(h_1) \cong P(h_2)$ if and only if $\Upsilon(h_1) = \Upsilon(h_2)$.
\item The multiplication map $P(h_1) \otimes_A P(h_2) \lar P(h_1 + h_2), \; f_1 \otimes f_2 \mapsto f_1 f_2$ is an isomorphism of $A$--modules.
\end{itemize}

\smallskip
\noindent
\textsl{6.~The spectral module of a Calogero--Moser system of rank two}.~In the cases when there exists  an embedding $A \stackrel{\Xi(\vec\xi)}\lar \gD$ (these are the so--called \emph{Baker--Akhieser} weighted line arrangements $(\Pi, \underline{\mu})$,  e.g.~a Coxeter one; see Definition \ref{D:BakerAkhieserWL} and Example \ref{E:CoxeterLA}), we  have an $(\gB-A)$--equivariant isomorphism
$F \cong P(-u)$ (where $u(z_1, z_2)  = \xi_1 z_1 + \xi_2 z_2$), i.e.~an isomorphism of vector spaces $F \stackrel{\cong}\lar P(-u)$ such that  the  diagram
$$
\xymatrix{
F \ar[rr]^-\cong \ar[d]_{-\,\circ \Xi(\vec\xi)(f)} & & P(-u) \ar[d]^-{-\,\cdot f} \\
F \ar[rr]^-\cong  & & P(-u)
}
$$
is commutative for any $f \in A$, where $\circ$ (respectively $\cdot$) denotes the action of $\gB$ (respectively $A$) on $F$ (respectively $P(-u)$).
This answers  a  question of Feigin and Johnston \cite{FeiginJohnston} and proves that
the spectral module $F$ of the rational Calogero--Moser system $\gB$ is projective.

The key role  in the proof of this result is plaid by the following formula for  the Hilbert function $\ZZ \xrightarrow{H_{P(-u)}} \NN_0$ of the filtered module $P(-u) \subset R$:
$$
H_{P(-u)}(k):= \dim_\CC\Bigl\{w \in P(-u) \; \big|\; \deg(w) \le k\Bigr\} = \dfrac{(k-\mu +1)(k-\mu +2)}{2}\quad \mbox{\rm for}\; k \in \ZZ,
$$
where $\mu:= \sum\limits_{\alpha \in \Pi}\mu_\alpha$; see  Theorem \ref{T:DescriptSpectModule}.

\smallskip
\noindent
\textsl{7.~Elements of the higher--dimensional Sato theory}.
 One motivation of our work was  to find an appropriate generalization of   Wilson's description  \cite{WilsonCrelle} of  \emph{bispectral} commutative subalgebras of rank one
of \emph{ordinary} differential operators on the case of \emph{partial} differential operators (the main point is that bispectrality can be characterized by the property of the spectral curve being rational with bijective normalization; see  \cite[Theorem 1]{WilsonCrelle}).
Following the main idea of  the works  \cite{Zheglov, KOZ, KurkeZheglov},
it seems to be natural to replace the algebra $\gD$ by a bigger algebra.
Namely, we introduce an  algebra of formal power series of differential operators having  the following special form:
$$
\gP := \left\{\sum\limits_{k_1, k_2 \ge 0} a_{k_1, k_2}(x_1, x_2) \partial_1^{k_1} \partial_2^{k_2} \, \Big| \, \exists d \in \ZZ: \, k_1 + k_2 - \upsilon\bigl(a_{k_1, k_2}(x_1, x_2)\bigr) \le d \;\, \forall  k_1, k_2 \ge 0\right\},
$$
where $\upsilon\bigl(a(x_1, x_2)\bigr)$ is the valuation of the power series $a(x_1, x_2) \in \CC\llbracket x_1, x_2\rrbracket$; see Definition
 \ref{D:AlgebraPi}. It turns out that the algebra
$\gP$  acts on $\CC\llbracket x_1, x_2\rrbracket$ (this action extends the natural action of the algebra $\gD$ on $\CC\llbracket x_1, x_2\rrbracket$) and contains some natural operators (e.g.~operators of a change of variables, delta--functions and integration operators) which do not belong to $\gD$; see Example \ref{Ex:DeltaFunctionSeries} and Example \ref{Ex:Integration}.

\smallskip
\noindent
Inspired by the theory of Sato Grassmannian \cite{Sato, Mulase, WilsonCrelle}, we introduce  the following sets:
$$
\mathsf{Gr}_\mu(R) := \Bigl\{W \subseteq R \, \Big| \, \dim_\CC\Bigl(w \in W \; \big|\; \deg(w) \le k + \mu\Bigr)  = \binom{k+2}{2} \; \mbox{\rm for any} \;  k \in \NN_0 \Bigr\}
$$
where  $\mu \in \NN_0$. It turns out that any \emph{Schur pair} $(W, A)$, where $W \in \mathsf{Gr}_\mu(R)$ and $A \subseteq R$ is a $\CC$--subalgebra such that
$W \cdot A = W,$ determines an injective algebra homomorphism $A \lar \gP$, which is unique up to an appropriate inner automorphism of $\gP$; see Theorem \ref{T:SatoSchur}. If $A = A\bigl(\Pi, \underline{\mu}\bigr)$ is the algebra of quasi--invariants of a Baker--Akhieser line arrangement  and
$W = P(-\xi_1 z_1 - \xi_2 z_2)$, then we recover in this way the algebra embedding (\ref{E:IntegrabilityCM}).

\smallskip
\noindent
\textsl{8.~Deformations of Calogero--Moser systems arising from Cohen--Macaulay modules}.
In Section \ref{S:InverseScattering}, we illustrate the developed ``algebraic inverse scattering method'' by constructing an ``isospectral deformation'' of the simplest dihedral Calogero--Moser system associated with the operator
$$
H =
\left(\frac{\partial^2}{\partial x_1^2} + \frac{\partial^2}{\partial x_2^2}\right)
- 2 \left(
\frac{1}{(x_1 - \xi_1)^2} + \frac{1}{(x_2 - \xi_2)^2}\right).
$$

\medskip
\noindent
We tried to keep the exposition of this paper  self--contained.
Still, our work is based on the following two external ingredients. Firstly, we essentially use the theory of multivariate Baker--Akhieser functions
of (generalized) Calogero--Moser systems \cite{ChalykhVeselov90, ChalykhStyrkasVeselov,  Berest, ChalykhFeiginVeselov2, Chalykh}.
  The second  ingredient is the ``matrix--problem method'' of \cite{SurvOnCM, BDNonIsol} to study Cohen--Macaulay modules over singular surfaces with non--isolated singularities.

\smallskip
\noindent
\textit{Acknowledgement}. The work of the first--named author was partially  supported by the CRC/TRR 191 project ``Symplectic Structures in Geometry, Algebra and Dynamics'' of German Research Council (DFG). The research of the second--named author was partially supported  by the DAAD program  ``Bilateral Exchange of Academics 2016'' and the RSF grant 16--11--10069.
We are also grateful to Misha Feigin  for  fruitful discussions.

\smallskip
\noindent
\textit{List of notations}.
For reader's convenience  we make an account of   the most important  notations used in this paper.

\smallskip
\noindent
1.~A weighted line arrangement  $(\Pi, \underline{\mu})$ is an input datum defining the corresponding algebra of quasi--invariant polynomials  $A = A(\Pi, \underline{\mu}) \subset R = \CC[z_1, z_2] = \CC\bigl[\rho\cos(\varphi), \rho \sin(\varphi)\bigr]$. It consists of a finite set $\Pi \subset \CC$ and a function $\Pi \stackrel{\underline{\mu}}\lar \NN_0, \; \alpha \mapsto \mu_\alpha = \underline{\mu}(\alpha)$. We put $\mu = \sum_{\alpha \in \Pi} \mu_\alpha$.
For a Baker--Akhieser weighted line arrangement, the function $\underline{\mu}$ has to satisfy  very strong constraints; see Section \ref{S:SpectralModule}.
In the Coxeter case, $\Pi = \Lambda_n = \bigl\{0, \frac{1}{n} \pi, \dots,  \frac{n-1}{n} \pi\bigr\}$. Finally, $X$ always denotes the affine surface, determined by the algebra $A$ and $\kP$ is the set of prime ideals in $A$ of height one.

\smallskip
\noindent
2.~Next, $I = \mathsf{Ann}_A(R/A)$ is the conductor ideal of the algebra extension $A \subseteq R$. For any $m \in \NN$, we denote  $K_m = \CC(\rho)[\sigma]/(\sigma^m)$ and $L_m = \CC(\rho)[\varepsilon]/(\varepsilon^{2m})$. For any $\alpha \in \Pi$, we denote  $K_\alpha = K_{\mu_\alpha}, L_\alpha = L_{\mu_\alpha}$ and $l_\alpha(z_1, z_2) = - \sin(\alpha)z_1 + \cos(\alpha) z_2 \in R$.
For each $f \in \CC\bigl[\rho\cos(\varphi), \rho \sin(\varphi)\bigr]$, $\alpha \in \Pi$ and $k \in \NN_0$ we put
$f^{(k)}_\alpha := \left.\dfrac{\partial^k f}{\partial \varphi^k}\right|_{\varphi = \alpha} \in \CC[\rho]$.

\smallskip
\noindent
3.~Cohen--Macaulay modules in our work are always \emph{maximal} Cohen--Macaulay and $\CM(A)$ denotes the category of Cohen--Macaulay $A$--modules, whereas
$\CM^{\mathsf{lf}}(A)$  stands for  its full subcategory consisting of those modules, which are locally free in codimension one. Next, $\Pic(A)$ is the Picard group of $A$, whereas $\CM_1(A)$ (respectively, $\CM^{\mathsf{lf}}_1(A)$) denotes the set of the isomorphism classes of rank one Cohen--Macaulay $A$--modules (respectively, the subset consisting of  those modules, which are locally free in codimension one). We denote by
$\Tri(A)$  the category of triples (see Definition \ref{D:triples}).

\smallskip
\noindent
Finally, $M(\vec{\nu}, \vec{\gamma})$ (respectively, $B(\vec\gamma)$ and $P(h)$) denote elements of $\CM_1(A)$ (respectively, $\CM^{\mathsf{lf}}_1(A)$ and $\Pic(A)$), expressed through the corresponding classifying parameters; see formula (\ref{E:classificationCM}) (respectively (\ref{E:cosnstraints}) and
(\ref{E:projectiveQuasiInv})).

\smallskip
\noindent
4.~Gothic letters stand for  objects, related with partial differential operators. In particular, $\gD = \CC\llbracket x_1, x_2\rrbracket[\partial_1, \partial_2]$, whereas $\gE = \CC\llbracket x_1, x_2\rrbracket \llbrace \partial_1^{-1}\rrbrace \llbrace \partial_2^{-1}\rrbrace$ is the algebra of pseudo--differential operators.  Next, $\gP$ is another algebra of ``infinite'' partial differential operators (see Section \ref{S:SatoTheory}).
We denote by $\diamond$ the natural left action of $\gD$ or $\gP$ on $\CC\llbracket x_1, x_2\rrbracket$, whereas $\circ$ stands for the natural right action of $\gD$ or $\gP$ on $\CC[\partial_1, \partial_2]$.
 Finally,
$\gB \subset \gD$ is  the   commutative subalgebra of $\gD$  containing the Calogero--Moser operator $H$ (\ref{E:CMHamiltonian})
 and given by the formula (\ref{E:IntegrabilityCM}),  whereas $F$ is the corresponding  spectral module.

\smallskip
\noindent
5.~For any $\mu \in \NN_0$, $\mathsf{Gr}_\mu(R)$ denotes the set of all subspaces $W \subset R$ with Hilbert polynomial $H_W(k + \mu) = \binom{k+2}{2}$; $S \in \gP$ is a  Sato operator of such $W \in \mathsf{Gr}_\mu(R)$ (see Definition \ref{D:SatoGrassmannian}).

\section{Ring--theoretic properties  of the algebra of surface   quasi--invariants}

\subsection{Some results on the Macaulayfication} 
Let $S$ be a local Noetherian ring of Krull dimension $d$ with the residue field $\mathbbm{k}$. Recall that a finitely generated  $S$--module $M$  is  
\emph{maximal Cohen--Macaulay} if $\mathsf{Ext}^i_S(\kk, M) = 0$ for all $0 \le i < d$. For a general Noetherian ring  $A$, a  finitely generated $A$--module $M$  is  Cohen--Macaulay if $M_\idm$ is maximal Cohen--Macaulay for any maximal ideal 
$\idm$ in $A$. It what follows, we shall drop the word ``maximal'' when talking about Cohen--Macaulay modules. A Noetherian ring is Cohen--Macaulay if it is Cohen-Macaulay viewed as a module over itself. 
We refer to the monographs  \cite{BrunsHerzog, LeuschkeWiegand, Serre} for the definition and main properties of Cohen--Macaulay rings and modules.

The
exposition in this subsection closely follows  the survey article \cite{SurvOnCM}, where a special attention  to the study of Cohen--Macaulay modules on singular surfaces  was paid.
From now on, let $A$ be a finitely generated integral $\CC$--algebra of Krull dimension two, $Q = Q(A)$ its field of fractions, $X$  an affine surface, whose coordinate ring is isomorphic to $A$ and  $\kP := \bigl\{\idp \in \Spec(A) \, \big|\,  \mathsf{ht}(\idp) = 1\bigr\}$.

\smallskip
\noindent
A proof of the following proposition can be for instance found in \cite[Appendix 2]{KOZ}.
\begin{proposition}\label{P:MacaulayficationRings} Let $A^\dagger := \bigcap\limits_{\idp \in \kP} A_{\idp} \subset Q.$
Then the following statements are true.
\begin{itemize}
\item $A^\dagger$ is a finitely generated Cohen--Macaulay $\CC$--algebra of Krull dimension two.
\item We have: $\dim_{\CC}\bigl(A^\dagger/A\bigr) < \infty$.
\item Moreover, $A = A^\dagger$ if and only if $A$ is Cohen--Macaulay.
\end{itemize}
The $\CC$--algebra $A^\dagger$ is called \emph{Macaulayfication} of $A$.
\end{proposition}

\smallskip
\noindent
Let $M$ be a Noetherian torsion free $A$--module. Recall that the rank of $M$ is defined to be $\rk(M):= \rk_Q\bigl(Q(M)\bigr)$, where
$Q(M):= Q \otimes_A M$ is the rational envelope of $M$.

\smallskip
\noindent
A proof  of an analogous  result for modules  can be  for instance found in \cite[Section 3]{SurvOnCM}.
\begin{proposition}\label{P:MacaulayficationModules}  Let $M^\dagger := \bigcap\limits_{\idp \in \kP} M_{\idp} \subset Q(M).$
Then the following statements are true.
\begin{itemize}
\item $M^\dagger$ is Noetherian  and Cohen--Macaulay module (both over $A^\dagger$ and $A$).
\item We have: $\dim_{\CC}\bigl(M^\dagger/M\bigr) < \infty$.
\item Moreover, $M = M^\dagger$ if and only if $M$ is Cohen--Macaulay.
\item Assume $A$ is already Cohen--Macaulay. Let $\Omega$ be a \emph{canonical module} of $A$. Then we have: $M^\dagger \cong M^{\vee\vee}$, where
$M^\vee := \Hom_A(M, \Omega)$.
\end{itemize}
Similarly to Proposition \ref{P:MacaulayficationRings}, the $A^\dagger$--module $M^\dagger$ is called \emph{Macaulayfication} of $M$.
\end{proposition}

\smallskip
\noindent
From now on, we shall assume that  the algebra $A$ is Cohen--Macaulay. In what follows, $\CM(A)$ denotes the category of all Cohen--Macaulay $A$--modules.

\begin{lemma} Suppose  that $A$ is Gorenstein in codimension one, i.e.~$A_\idp$ is Gorenstein for any $\idp \in \kP$. Then for any torsion free Noetherian $A$--module $M$ we have:
$
M^\dagger \cong M^{\ast\ast},
$
where $M^\ast = \Hom_A(M, A)$.
\end{lemma}

\begin{proof}
We prove this fact here since the  proof given in \cite[Proposition 3.7]{SurvOnCM} contained an error. First note that both modules $M^{\vee\vee}$ and $M^{\ast\ast}$ are Cohen--Macaulay; see for example \cite[Lemma 3.1]{SurvOnCM}. Let $M \stackrel{i}\lar M^{\vee\vee}$ and $M \stackrel{j}\lar M^{\ast\ast}$ be the canonical morphisms. By the universal property of Macaulayfication \cite[Proposition 3.2]{SurvOnCM}, there exists a unique morphism $M^{\vee\vee} \stackrel{f}\lar M^{\ast\ast}$ such that
the diagram
$$
\xymatrix{
& M \ar[ld]_-{i} \ar[rd]^-{j} & \\
M^{\vee\vee} \ar[rr]^-{f} & & M^{\ast\ast}
}
$$
is commutative. Since both morphisms $i$ and $j$ induce isomorphisms of the corresponding rational envelopes, we may conclude that $f$ is a monomorphism.
Next, $M_\idp$ is automatically Cohen--Macaulay over $A_\idp$ and $\Omega_\idp \cong A_\idp$ for any $\idp \in \kP$.
Therefore, the morphisms of $A_\idp$--modules $i_\idp$ and $j_\idp$ are isomorphisms, hence $f_\idp$ is an isomorphism, too.
As a consequence, the cokernel of $f$ is a finite dimensional $A$--module. Since both $A$--modules $M^{\vee\vee}$ and $M^{\ast\ast}$ have depth two, $f$ is surjective  by the Depth Lemma \cite[Proposition 1.2.9]{BrunsHerzog}.
\end{proof}

\begin{definition}
A torsion free Noetherian  $A$--module $M$ is called \emph{locally free in codimension one} if $M_\idp$ is a free $A_\idp$--module for any $\idp \in \kP$. If $M$ is additionally assumed to be Cohen--Macaulay, this equivalently means that  there exists a finite set $Z \subset X$ and a locally free sheaf $N$ on the surface $X \setminus Z$ such that $M \cong \imath_*(N)$, where
$X \setminus Z \stackrel{\imath}\lar X$ is the canonical open inclusion (the proof of this statement is the same as in \cite[Corollary 3.12]{SurvOnCM}).
In what follows, $\CM^{\mathsf{lf}}(A)$ denotes the category of all Cohen--Macaulay $A$--modules, which are locally free in codimension one.
\end{definition}

\begin{remark}
Let $I \subset A$ be a maximal ideal. Then $I$ is a torsion free Noetherian $A$--module of rank one which is locally free in codimension one. 
It follows from Proposition \ref{P:MacaulayficationModules} that $I^\dagger = A$.
In particular, $I$ is not Cohen--Macaulay. 
\end{remark}

\smallskip
\noindent
Recall that the Picard group of $A$ (denoted $\Pic(A)$) is the set of the isomorphism classes of projective $A$--modules of rank one with the group operation given by the tensor product. 

\begin{proposition}\label{P:CMPicard}
 Assume that the algebra $A$ is Gorenstein in codimension one. Abusing the notation, let $\CM_1^{\mathsf{lf}}(A)$ be the set of the isomorphism classes of Cohen--Macaualy $A$--modules of rank one, which are locally free in codimension one.  Then $\CM_1^{\mathsf{lf}}(A)$ is an abelian group  (which is an analogue of the \emph{divisor class group} of a normal domain \cite[Section 7.3]{BrunsHerzog}) with respect to the operation
\begin{equation}\label{E:tensorproduct}
M_1 \boxtimes_A M_2 := \bigl((M_1 \otimes_A M_2)/\tor)^\dagger,
\end{equation}
where $\tor$ denotes the torsion submodule of $M_1 \otimes_A M_2$.
The neutral element of $\CM_1^{\mathsf{lf}}(A)$ is $A$, whereas  the inverse element of $M$ is $M^\ast = \Hom_A(M, A)$.
\end{proposition}

\begin{proof} For the associativity of $\boxtimes_A$, see \cite[Proposition 3.15]{SurvOnCM}. It is also clear that $A$ is the neutral element of $\CM_1^{\mathsf{lf}}(A)$. Finally, let $M$ be an element of $\CM_1^{\mathsf{lf}}(A)$.
Then the  evaluation  morphism $M \otimes_A \Hom_A(M, A) \stackrel{\mathsf{ev}}\lar A$ induces a morphism of Cohen--Macaulay modules
$
M \boxtimes_A \Hom_A(M, A) \stackrel{\mathsf{ev}'}\lar A
$
such that $\mathsf{ev}'_\idp$ is an isomorphism for any $\idp \in \kP$. By \cite[Lemma 3.6]{SurvOnCM}, $\mathsf{ev}'$ is an isomorphism. Hence, $M^\ast$ is indeed inverse to  $M$ in $\CM_1^{\mathsf{lf}}(A)$.
\end{proof}

\begin{remark}
Let $Z_k \subset X$ be a finite subset, $N_k$ be a locally free  sheaf on $U_k := X \setminus Z_k$ and $M_k := \imath_{k\ast}(N_k) \in \CM_1^{\mathsf{lf}}(A)$ for $k = 1,2$. Then we have:
$
M_1 \boxtimes_A M_2 \cong \imath_\ast\bigl(N_1\big|_{U} \otimes N_2\big|_{U}\bigr),
$
where $U := X \setminus \{Z_1 \cup Z_2\}$ and $U \stackrel{\imath}\lar X$ is the canonical embedding. The proof of this result follows from \cite[Proposition 3.10]{SurvOnCM}. Note that $\Pic(A)$ is a subgroup of $\CM_1^{\mathsf{lf}}(A)$.
\end{remark}

\subsection{Category of triples}
Now we introduce a  certain categorical construction \cite{SurvOnCM,BDNonIsol}, playing the key role in our paper.
Let $A$ be a reduced \emph{Cohen--Macaulay}
$\CC$--algebra  of Krull dimension two, either finitely generated or complete. Let $R$ be the normalization of $A$. Note that $R$ is automatically
Cohen--Macaulay \cite[Theorem IV.D.11]{Serre} and
the algebra extension $A \subseteq R$ is finite.

 We denote by $I = \mathsf{Ann}_A(R/A) = \bigl\{r \in A \,\big|\, rR \subseteq A \bigr\} \cong \Hom_A(R, A)$ the conductor ideal of the algebra extension $A \subseteq R$. Observe that  $I$ is an ideal both in $A$ and $R$. Moreover, $I$ is Cohen--Macaulay (both  over $A$ and $R$); see e.g.~\cite[Lemma 3.1]{SurvOnCM}. Both $\CC$--algebras
$\bar{A} = A/I$ and $\bar{R} = R/I$ are Cohen--Macaulay of Krull dimension one (but not necessarily reduced).
Let
$\rQ(\bar{A})$ and $\rQ(\bar{R})$ be the corresponding total rings
of fractions. Then the algebra extension  $\bar{A} \subseteq \bar{R}$ induces  a canonical embedding
$\rQ(\bar{A}) \subseteq  \rQ(\bar{R})$.

Let $\mathsf{Ass}_R(I) = \left\{\idq_1, \dots, \idq_n\right\}$ be the set of the associated prime ideals of $I$ in $R$. Since the algebra $Q(\bar{R})$ is Artinian,
the Chinese Remainder theorem implies that
$$
Q(\bar{R}) \cong \bar{R}_{\bar{\idq}_1} \times \dots \times \bar{R}_{\bar{\idq}_n},
$$
where $\bar{\idq}_1, \dots, \bar{\idq}_n$ are the images of $\idq_1, \dots, \idq_n$ in $\bar{S}$. Of course, a similar result holds for  the algebra $Q(\bar{A})$, too.

\smallskip
\noindent
The proof of the following result can be found in \cite[Lemma 3.1]{BDNonIsol}.

\begin{lemma}
For any $M \in \CM(A)$, the following statements are true:

\smallskip
\noindent
1.~the canonical morphism of $\rQ(\bar{R})$--modules
$$
\theta_M:
Q(\bar{R}) \otimes_{Q(\bar{A})} \bigl(Q(\bar{A})
\otimes_{A} M\bigr)
\lar Q(\bar{R}) \otimes_{R} \bigl(R  \otimes_{A} M\bigr)
 \lar
 Q(\bar{R}) \otimes_{R} \bigl(R  \boxtimes_{A} M\bigr)
$$
is an epimorphism.

\smallskip
\noindent
2.~the adjoint morphism of $Q(\bar{A})$--modules
$$\tilde\theta_M: \;  Q(\bar{A}) \otimes_{A} M \lar Q(\bar{R}) \otimes_{Q(\bar{A})} \bigl(Q(\bar{A}) \otimes_A M\bigr)
\stackrel{\theta_M}\lar Q(\bar{R}) \otimes_{R} \bigl(R  \boxtimes_{A} M\bigr)
$$
is a monomorphism.
\end{lemma}

\begin{definition}\label{D:triples}
Consider the following diagram of categories and functors:
$$
\CM(R) \xrightarrow{\; \; \; Q(\bar{R}) \otimes_{R} \,-\,\;\;  }  Q(\bar{R})-\mathsf{mod} \xleftarrow{Q(\bar{R}) \otimes_{Q(\bar{A})} \,-\,} Q(\bar{A})-\mathsf{mod}
$$
According to \cite[Section II.6]{MacLane},  the corresponding \emph{comma category} $\mathsf{Comma}(A)$ is defined as follows. Its objects are   \emph{triples}
 $(N, V, \theta)$, where $N$
is a Cohen---Macaulay $R$--module, $V$ is a Noetherian
$Q(\bar{A})$--module and
$Q(\bar{R}) \otimes_{Q(\bar{A})} V \stackrel{\theta}\lar Q(\bar{R}) \otimes_{R} N$
is a morphism of $Q(\bar{R})$--modules (called \emph{gluing map}).

A morphism between two objects  $(N, V, \theta)$
and $(N', V', \theta')$ of the category $\mathsf{Comma}(B)$  is given by a pair $(f, g)$, where
$N \stackrel{f}\lar N'$ is a morphism of $R$--modules and
$V \stackrel{g}\lar V'$ is a morphism of $Q(\bar{A})$--modules such that
the following diagram
\begin{equation}\label{E:morphismtriples}
\begin{array}{c}
\xymatrix
{
\rQ(\bar{R}) \otimes_{\rQ(\bar{A})} V \ar[rr]^-\theta \ar[d]_{1 \otimes g} & &
\rQ(\bar{R}) \otimes_R N \ar[d]^{1 \otimes f}\\
\rQ(\bar{R}) \otimes_{\rQ(\bar{A})} V' \ar[rr]^-{\theta'} & &
\rQ(\bar{R}) \otimes_{R} N'
}
\end{array}
\end{equation}
is commutative.

The \emph{category of triples} $\Tri(A)$ is the full subcategory of $\mathsf{Comma}(A)$ consisting of those triples $(N, V, \theta)$, for which
the gluing map $\theta$ satisfies the following two conditions:
\begin{itemize}
\item $\theta$ is an epimorphism;
\item the adjoint morphism
of $Q(\bar{A})$--modules $\tilde\theta$ defined as the composition
$$V \lar
Q(\bar{R}) \otimes_{Q(\bar{A})} V \stackrel{\theta}\lar Q(\bar{R}) \otimes_{R} N
$$
is a monomorphism. \qed
\end{itemize}
\end{definition}

\noindent
In the above terms, we have a commutative diagram of $\CC$--algebras
\begin{equation}\label{E:keydiag}
\begin{array}{c}
\xymatrix{
A  \ar@{_{(}->}[d] \ar[rr] & & Q(\bar{A}) \ar@{^{(}->}[d] \\
R \ar[rr]   & & Q(\bar{R}).
}
\end{array}
\end{equation}
The main idea is to realize the category $\CM(A)$ as a ``categorical gluing'' of the categories $\CM(R)$ and $Q(\bar{A})-\mathsf{mod}$ along the category
$Q(\bar{R})-\mathsf{mod}$. This is implemented
by  the following result \cite[Theorem 3.5]{BDNonIsol}.

\begin{theorem}\label{T:BurbanDrozd} The
functor
$
\CM(A) \stackrel{\mathbb{F}}\lar  \Tri(A),
$
mapping a Cohen-Macaulay $A$--module $M$ to the triple
$\bigl(R \boxtimes_A M, Q(\bar{A}) \otimes_A M,
\theta_M\bigr)$, is an equivalence of categories.
The functor $\FF$ establishes  an equivalence of categories between
$\CM^{\mathsf{lf}}(A)$ and  the full subcategory $\Tri^{\mathsf{lf}}(A)$ of $\Tri(A)$ consisting
of those triples $(N, V, \theta)$ for which the $Q(\bar{A})$--module $V$ is free  and the morphism
$\theta$ is an isomorphism.

\smallskip
\noindent
Moreover, for any objects $M_1$ and $M_2$ of $\CM^{\mathsf{lf}}(A)$ we have:
$$
\FF(M_1 \boxtimes_A M_2) \cong \FF(M_1) \otimes \FF(M_2),
$$
where
$
(N_1, V_1, \theta_1) \otimes (N_2, V_2, \theta_2):=
\bigl(N_1 \boxtimes_R N_2, V_1 \otimes_{Q(\bar{A})} V_2, \theta_1 \otimes \theta_2\bigr)
$
for any two objects of the category $\Tri^{\mathsf{lf}}(A)$.

\smallskip
\noindent
Finally, for any maximal ideal $\idm$ in the algebra  $A$, we have a commutative diagram of categories and functors
\begin{equation}\label{E:LocalizeCompleteTriples}
\begin{array}{c}
\xymatrix{
\CM(A) \ar[rr]^-{\mathbb{F}} \ar[d]_-{\widehat{(-)}_{\idm}} & & \Tri(A)  \ar[d]^-{\mathbb{L}_\idm} \\
\CM(\widehat{A}_\idm) \ar[rr]^-{\mathbb{F}_\idm}  & & \Tri(\widehat{A}_\idm),
}
\end{array}
\end{equation}
where $\mathbb{L}_\idm\bigl(N, V, \theta\bigr) = \bigl(\widehat{N}_\idm, \widehat{V}_\idm, \widehat{\theta}_\idm\bigr)$.
\end{theorem}

\subsection{Category of triples for the algebra of planar quasi--invariants}\label{SS:CatTriplesQuasiInv}
From now on, we put $R := \CC[z_1, z_2]$. As in Introduction, we fix the following datum $(\Pi, \underline{\mu}\bigr)$, called
\emph{weighted line arrangement}:
\begin{itemize}
\item
$\Pi \subset \CC$ is a finite subset satisfying the condition: $\alpha-\beta \notin \pi \ZZ$ for any $\alpha \ne \beta \in \Pi$.
\item $\Pi \stackrel{\underline{\mu}}\lar \NN_0, \; \alpha \mapsto \mu_\alpha = \underline{\mu}(\alpha)$ is a multiplicity function.
\end{itemize}
The  above restriction on the elements of the set $\Pi$ has the following meaning:  the lines  $V(l_\alpha)$ and $V(l_\beta)$ are distinct for any $\alpha \ne \beta \in \Pi$, where $l_\alpha = - \sin(\alpha)z_1 + \cos(\alpha) z_2 \in R$ and $V(f) \subset \AA^2$ denotes the vanishing set of $f \in R$. Note that any complex line $V(l) \subset \AA^2$ passing through the origin, with the only exception of two  lines $V(z_1 \pm i z_2)$, can be written as $V(l_\alpha)$ for an appropriate $\alpha \in \CC$.

\smallskip
\noindent
For any $\alpha \in \Pi$, consider the reflection
$$\CC^2 \stackrel{r_\alpha}\lar \CC^2, \; \vec{x} \mapsto \vec{z} - 2 \bigl(\vec{z}, \vec{e}_\alpha\bigr) \vec{e}_\alpha,$$
where  $\vec{e}_\alpha := \bigl(-\sin(\alpha), \cos(\alpha)\bigr)$ and $\bigl(\vec{z}, \vec{e}_\alpha\bigr) = - \sin(\alpha) z_1 + \cos(\alpha) z_2$.
Let
$$R \stackrel{s_\alpha}\lar R, \quad f(\vec{z}) \mapsto \bigl(s_\alpha(f)\bigr)(\vec{z}):= f\bigl(r_\alpha(\vec{z})\bigr)
$$ be the involution associated with the reflection $r_\alpha$.

\smallskip
\noindent
Let $A = A\bigl(\Pi, \underline{\mu}\bigr)$ be the  algebra of surface quasi--invariant polynomials (\ref{E:quasiinvariants}) corresponding to the datum
$\bigl(\Pi, \underline{\mu}\bigr)$.
In what follows, we shall denote $$
\delta:= \prod\limits_{\alpha \in \Pi} l_\alpha^{\mu_\alpha} \quad \mbox{\rm and} \quad  \omega:= z_1^2 + z_2^2.
$$
Observe that both polynomials $\omega$ and $\delta^2$  belong to the algebra $A$.

\smallskip
\noindent
Although the following  result is not original (compare with \cite[Lemma 7.3]{BerestEtingofGinzburg} and \cite[Theorem 3.3.2]{Johnston}), we provide  its detailed proof for the sake of completeness and convenience of the reader.
\begin{proposition}\label{P:BasicsQuasiinv}
The algebra  $A$ is a finitely generated homogeneous subalgebra of $R$ of Krull dimension two. Next, $Q(A) = Q(R) = \CC(z_1, z_2)$ and the algebra $R$ is the normalization of $A$ (i.e.~the integral closure of $A$ is $Q(A)$).
Let $X$ be an affine surface, whose coordinate ring is isomorphic to $A$. Then the corresponding normalization map $\AA^2 \stackrel{\nu}\lar  X$ is bijective.
\end{proposition}

\begin{proof} First note the following
 elementary fact: if $l \in R$ is a homogeneous polynomial dividing
another (possibly, non--homogeneous) polynomial $f \in R$, then $l$ divides the  homogeneous component of $f$ of the highest degree. It follows from the definition (\ref{E:quasiinvariants}) that $A$ is a homogeneous subalgebra of $R$.

To prove that the algebra $A$ is finitely generated, put
$A^\circ := \CC[\omega, \delta^2] \subseteq A$. We claim that $R$ is finite viewed as an $A^\circ$--module.
 Indeed, let $J = \langle \omega, \delta^2\rangle_R$.
Since $V\bigl(z_1 \pm i z_2, l_\alpha\bigr) = \{0\}$ for any $\alpha \in \Pi$, we have: $V(J) = \{0\}$. Hence, the ideal $J$ is $(z_1, z_2)$--primary
implying that  $\dim_{\CC}(R/J) < \infty$. Let $h_1, \dots, h_m$ be a basis of $R/J$ consisting of homogeneous polynomials.  We claim that
\begin{equation}\label{E:finitelygenerated}
R = \bigl\langle h_1, \dots, h_m\rangle_A.
\end{equation}
Indeed, let  $f\in R$ be an arbitrary homogeneous polynomial. Then there exist $\lambda_1\dots, \lambda_m \in \CC$ as well as homogeneous polynomials
$g, h \in R$ satisfying  $\deg(g) < \deg(f)$ and $\deg(h) < \deg(f)$, such that
$
f = \bigl(\sum\limits_{j = 1}^m \lambda_j h_j\bigr) + g \omega + h \delta^2.
$
Proceeding inductively with $g, h$, we get the result (\ref{E:finitelygenerated}). Since we have a tower of algebra extensions
$A^\circ \subseteq A \subseteq R$ and $R$ is finitely generated as $A^\circ$--module, the algebra $A$ is finitely generated of Krull dimension two.

Since for any $f \in R$, the polynomial $f \delta^2$ belongs to $A$, we have: $f \in Q(A)$. Hence, $Q(A) = Q(R)$.
The normalization map
$\AA^2 \stackrel{\nu}\lar X$  is automatically surjective. We have to show  that $\nu$ is injective.
Our proof is a slightly modified version of the argument from \cite[Lemma 7.3]{BerestEtingofGinzburg}. By Hilbert's Nullstellensatz, the
injectivity of $\nu$ is equivalent to the statement that for any two maximal ideals $\idm \ne \idn$ of $R$ we have: $A \cap \idm \ne A\cap \idn$. Equivalently, for any pair of points $p \ne q \in \AA^2$, there exists $f \in A$ such that $f(p) = 0$ and $f(q) \ne 0$. Without loss of generality,
we may assume that $q \ne (0, 0)$. Then the following two cases can occur.

\smallskip
\noindent
\underline{Case 1}. $q \notin V(\delta) = \cup_{\alpha \in \Pi} V(l_\alpha)$. Again,  take any $g \in R$ such that $g(p) = 0$ and $g(q) \ne 0$
 and put
$f := \delta g$. Then $f \in A$ and $f(p) = 0$, whereas $f(q) \ne 0$.

\smallskip
\noindent
\underline{Case 2}. $q \in V(\delta) = \cup_{\alpha \in \Pi} V(l_\alpha)$. Since $q \ne (0,0)$, there exists precisely one $\alpha \in \Pi$ such that
$l_\alpha(q) = 0$. Again, take any $g \in R$ such that $g(p) = 0$ and $g(q) \ne 0$. Now we put
$$
f = \bigl(\prod\limits_{\beta \ne \alpha} \bigl(s_\alpha(l_\beta) \cdot l_\beta\bigr)^{2\mu_\beta}\bigr) \cdot \bigl(s_\alpha(g)\cdot g\bigr).
$$
By construction, $l_\beta^{2\mu_\beta} \big| f$ for any $\beta \ne \alpha$, whereas $s_\alpha(f) = f$. Hence, $f \in A$. Obviously,
$f(p) = 0$. On the other hand, $\bigl(s_\alpha(g))(q) = g\bigl(s_\alpha(q)\bigr) = g(q) \ne 0$ and in a similar way,
$\bigl(s_\alpha(l_\beta)\bigr)(q) = l_\beta\bigl(s_\alpha(q)\bigr)\ne 0$ for any $\beta \ne \alpha$. Therefore, $f(q) \ne 0$.
\end{proof}

It is convenient to rewrite the definition of the algebra $A\bigl(\Pi, \underline{\mu}\bigr)$  in terms of  polar coordinates.
We put   $z_1 = \rho \cos(\varphi), z_2 = \rho \sin(\varphi)$ and identify the algebra $R$ with the subalgebra
$\CC\bigl[\rho \cos(\varphi), \rho \sin(\varphi)\bigr]$ of  the algebra $\CC\{\rho, \varphi\}$ of analytic functions on
$\mathbb{C} \times \CC = \CC_\rho \times \CC_\varphi$. For any $f \in \CC\{\rho, \varphi\}$, $\alpha \in \CC$ and $k \in \NN_0$, consider the analytic function
\begin{equation}\label{E:directderivat}
f^{(k)}_\alpha := \left.\frac{\partial^k f}{\partial \varphi^k}\right|_{\varphi = \alpha}: \; \mathbb{C} \lar \CC.
\end{equation}
In these terms, we have an injective algebra homomorphism
$$
\CC\{\rho, \varphi\} \lar \CC\{\rho\}\llbracket \varepsilon \rrbracket, \quad f \mapsto T_\alpha(f) := \sum\limits_{k = 0}^\infty f^{(k)}_\alpha \dfrac{\varepsilon^k}{k!}.
$$
In the polar coordinates, we have: $l_\alpha  = \rho \sin(\varphi - \alpha)$ and
$$\bigl(s_\alpha(f)\bigr)(\rho, \varphi) = f(\rho, 2 \alpha -\varphi) \quad \mbox{\rm for any }\; f \in R.$$

\begin{lemma}\label{L:TaylorFractFields} For $k \in \NN$, let $P_k:= \CC[\rho, \varepsilon]/(\varepsilon^k)$ and
\begin{equation}
T_{(\alpha, k)}: R \lar P_k, \quad f \mapsto \sum\limits_{i = 0}^{k-1} f^{(i)}_{\alpha} \frac{\varepsilon^i}{i!}
\end{equation}
Then the following results are true.
\begin{itemize}
\item The map $T_{(\alpha, k)}$ is an algebra homomorphism and $\ker\bigl(T_{(\alpha, k)}\bigr) = \bigl(l_\alpha^k\bigr)$.
\item The algebra inclusion  $\widetilde{R}_{\alpha, k} := R/(l_\alpha^k) \stackrel{\tau}\lar P_k$  induces an isomorphism $\tilde\tau$ of the corresponding total rings of fractions.
\end{itemize}
\end{lemma}
\begin{proof}
The fact that $T_{(\alpha, k)}$ is an algebra homomorphism, is a basic property of Taylor series. If $f \in R$ is such that $f'_{\alpha} = 0$ then $l_\alpha \mid f$.
The statement that  $\ker(T_{(\alpha, k)}) = \bigl(l_\alpha^k\bigr)$ can be easily proven by induction.
Let $$u := \rho \cos(\varphi - \alpha)\quad  \mbox{\rm and} \quad v := \rho \sin(\varphi - \alpha).$$
 It is clear that $R = \CC[u, v]$. We put:
$$
\left\{
\begin{array}{l}
\bar{u} := \tau(u) = \rho\bigl(1 - \dfrac{\varepsilon^2}{2!} + \dfrac{\varepsilon^4}{4!} - \dots \bigr) \vspace{2mm}\\
\bar{v} := \tau(v) = \rho\bigl(\varepsilon - \dfrac{\varepsilon^3}{3!} + \dfrac{\varepsilon^5}{5!} - \dots \bigr).
\end{array}
\right.
$$
Next, we denote: $\bar{w}_0 = \bar{u}$  and $\bar{w}_i = \bar{u}^{1-i} \bar{v}^i \in Q(P_k)$ for $1 \le i \le k-1$.
It is clear that
\begin{itemize}
\item $\bar{w}_i \in \mathrm{Im}\bigl(Q(\widetilde{R}_{\alpha, k}) \stackrel{\tilde\tau}\lar Q(P_k)\bigr)$ and
\item $\bar{w}_i = \rho \cdot s_i$, where $s_i = \varepsilon^i + \mathrm{h.o.t} \in \CC[\varepsilon]/(\varepsilon^k)$
\end{itemize}
for any $0 \le i \le k-1$. It is now easy to see that $\rho$ and any element $s \in \CC[\varepsilon]/(\varepsilon^k)$ belong to the image of $\tilde\tau$. Hence, $Q(\widetilde{R}_{\alpha, k}) = Q(P_k)$, as claimed.
\end{proof}

\begin{corollary}\label{C:DescriptQuasiInv}
We get  the following description of the algebra of quasi--invariants:
\begin{equation}\label{E:quasiinvpolar}
A = A\bigl(\Pi, \underline{\mu}\bigr) =\bigl\{f \in R \,\big|\, f'_\alpha = f^{'''}_\alpha = \dots = f^{(2\mu_\alpha -1)}_{\alpha} = 0 \;\;  \mbox{\rm for all} \;\; \alpha \in \Pi\bigr\}.
\end{equation}
\end{corollary}

\begin{lemma}\label{L:conductor}
Let $f \in R$ be such that $T_{(\alpha, 2m)}(fg) \in \CC[\rho, \varepsilon^2]/(\varepsilon^{2m})$ for any $g \in R$. Then we have:
$T_{(\alpha, 2m)}(f) = 0$.
\end{lemma}
\begin{proof} Taking $g =1$, we see that
$
T_{(\alpha, 2m)}(f) = \lambda_{j} \varepsilon^{2j} + \lambda_{j+1} \varepsilon^{2j+2} + \dots +\lambda_{m-1} \varepsilon^{2m-2}
$
for some $\lambda_{j}, \dots, \lambda_{m-1} \in \CC[\rho]$. Suppose that $j \le m-1$ and $\lambda_{j} \ne 0$. Take
$g := v^{2(m-j)-1} \in R$. Then we have:
$
T_{(\alpha, 2m)}(gf) = \rho^{2(m-j)-1} \lambda_{j} \cdot \varepsilon^{2m-1} \notin \CC[\rho, \varepsilon^2]/(\varepsilon^{2m}),
$
giving a contradiction.
\end{proof}

\begin{corollary}\label{C:ConductorDescription} Let $I := \mathsf{Ann}_A(R/A) = \bigl\{f \in R \,|\, gf \in A \; \mbox{\rm for any} \; g \in R\bigr\} \cong \Hom_A(R, A)$ be the conductor ideal.  Then we have:
$
I = (\delta^2)_R \cong R.
$
In particular, the  ideal $I$ is Cohen--Macaulay (both over $A$ and over $R$, since $\depth_A(I) = \depth_R(I) = 2$) and we have: $$\mathsf{Ass}_{R}(I) = \bigl\{\idq_\alpha \, |\, \alpha \in \Pi\bigr\}, \quad \mbox{\rm where}\quad  \idq_\alpha := (l_\alpha)_R  \subset R.$$ In other words, the affine variety
$V_R(I) \subset \AA^2$ is a union of $n$ lines defined by the set $\Pi$.
\end{corollary}
\begin{proof} It follows from the definition of the conductor ideal $I$ that
$$
I = \bigl\{f \in R \; \big| \;  T_{(\alpha, 2\mu_\alpha)}(gf) \in \CC[\rho, \varepsilon^2]/(\varepsilon^{2\mu_\alpha}) \; \mbox{for all} \; g \in R
\; \mbox{and} \; \alpha \in \Pi\bigr\}.
$$
Hence, the  statement  immediately follows from Lemma \ref{L:TaylorFractFields} and Lemma \ref{L:conductor}.
\end{proof}

\begin{remark} The algebra $\bar{A}:= A/I$ is \emph{Cohen--Macaulay} of Krull dimension one, since we have a finite
 extension $\bar{A} \subseteq \bar{R}:= R/I$ and the algebra
$\bar{R}$ has these properties.
\end{remark}

\begin{lemma} The map $\mathsf{Ass}_{R}(I) \lar \mathsf{Ass}_{A}(I), \; \idq \mapsto \idp:= A \cap \idq$ is a bijection. Moreover, for any
$\idq \in \mathsf{Ass}_{R}(I)$ we have: $R_\idp = R_\idq$ and the algebra  extension $A_\idp \subseteq R_\idp$ induces an isomorphism of the corresponding residue fields.
\end{lemma}

\begin{proof}
The first statement is a consequence of the fact that the normalization morphism $\AA^2 \stackrel{\nu}\lar X$ is bijective; see Proposition \ref{P:BasicsQuasiinv}. It follows that $\idq R_\idp$ is the unique maximal ideal of $R_\idp$, hence $R_\idp = R_\idq$.

Next, the morphism of affine curves
$
\AA^1 \cong V(\idq) \stackrel{\nu}\lar V(\idp)
$
is bijective, hence it is automatically birational. Therefore, the algebra extension $A/\idp \subseteq R/\idq$ induces an isomorphism of the corresponding fields of fractions and we have:
$
A_\idp/\idp A_\idp \cong R_\idq/\idq R_\idq.
$
\end{proof}

\begin{corollary}\label{C:residuefields}
Let $\bar\idq \in \mathsf{Ass}_{\bar{R}}(0)$ and $\bar\idp = \bar\idq \cap \bar{A}$. Then the algebra extension
$\bar{A}_{\bar\idp} \subseteq \bar{R}_{\bar\idq}$ induces an isomorphism of the corresponding residue fields.
\end{corollary}

\begin{remark} If we view $I$ as an ideal in $A$ then the corresponding  affine variety $V_A(I) \subset X$ is the locus  of those points where the surface $X$ is not normal. If we view $I$ as an ideal in $R$ then we have $\AA^2 \supset V_R(I) = \nu^{-1}\bigl(V_A(I)\bigr)$. Since the map
$\AA^2 \setminus V_R(I) \stackrel{\nu}\lar X \setminus V_A(I)$ is an isomorphism, $V_A(I)$ is precisely the singular locus of $X$. According to Corollary
\ref{C:ConductorDescription}, the curve $V_R(I) \subset \AA^2$ is a line arrangement
 consisting of $n$ lines passing through the origin $(0, 0)$, whose slopes are determined
by the set $\Pi$.  On the other hand, $V_A(I) = \nu\bigl(V_R(I)\bigr)$ is  a union of $n$ rational cuspidal curves (the order of each cusp is determined
by the corresponding value of the multiplicity function $\underline{\mu}$) meeting at the common  point $\nu(0,0)$.
\end{remark}

\begin{proposition} For any $\alpha \in \Pi$, put  $\idq = \idq_\alpha$, $\idp = \idq\cap A$ and consider the following diagram of $\CC$--algebras:
\begin{equation}
\begin{array}{c}
\xymatrix{
A \ar@{->>}[r] \ar@{_{(}->}[d] & \bar{A} \ar@{^{(}->}[d] \ar[r] & \bar{A}_{\bar\idp} \ar@{^{(}->}[d] \ar@/^30pt/[ddd]^-{\bar{T}} \\
R \ar@{->>}[r] \ar@{->>}[rd] \ar@/_10pt/[rdd]_{T}  & \bar{R}  \ar[r] & \bar{R}_{\bar\idq} \\
 & \widetilde{R} \ar@{^{(}->}[r] \ar@{->>}[u] \ar@{_{(}->}[d]_-\tau & \widetilde{R}_{\tilde\idq} \ar[u]_-\cong \ar[d]^{\tilde\tau}\\
 & P \ar@{^{(}->}[r] & L
}
\end{array}
\end{equation}
where
\begin{itemize}
\item $\widetilde{R} = \widetilde{R}_{\alpha, 2 \mu_\alpha} := R/(l_\alpha^{2 \mu_\alpha})$ and $\tilde\idq$ is the image of $\idq$ in $\widetilde{R}$;
\item $P = P_{2\mu_\alpha} := \CC[\rho, \varepsilon]/(\varepsilon^{2\mu_\alpha})$ and $T = T_{(\alpha, 2 \mu_\alpha)}$;
\item $L = L_\alpha := \CC(\rho)[\varepsilon]/(\varepsilon^{2\mu_\alpha})$.
\end{itemize}
Then we have: $\mathsf{Im}(\bar{T}) = K = K_\alpha :=  \CC(\rho)[\varepsilon^2]/(\varepsilon^{2\mu_\alpha})$.
\end{proposition}

\begin{proof} We first show that $\mathsf{Im}(\bar{T}) \subseteq  K$. Indeed, for any $a \in A$ we have:
$\bar{T}(\bar{a}) = T(a) \in K$. Next, observe that if $c \in L$ is invertible and $c \in K$, then $c^{-1} \in K$. Therefore, for any
$\dfrac{\bar{a}}{\bar{b}} \in \bar{A}_{\bar\idp}$ we have:
$
\bar{T}\Bigl(\dfrac{\bar{a}}{\bar{b}}\Bigr) = T(a) \cdot T(b)^{-1} \in K.
$

\smallskip
\noindent
Let $K':= \mathsf{Im}(\bar{T})$, then we have to prove that $K' = K$. Note the following two facts.
\begin{enumerate}
\item Consider the  element $\delta_\alpha:= l_\alpha^2 \cdot \prod\limits_{\beta \ne \alpha} \bigl(s_\alpha(l_\beta) \cdot l_\beta\bigr)^{2\mu_\beta} \in R$. Since $s_\alpha(\delta_\alpha) = \delta_\alpha$ and $l_\beta^{2\mu_\beta} \, \big| \,\delta_\alpha$ for all $\beta \ne \alpha$, we have:
    $\delta_\alpha \in A$. Moreover,
$$
\bar{\delta}_\alpha:= \bar{T}(\delta_\alpha) = \lambda_1 \varepsilon^2 + \dots + \lambda_{m-1}\varepsilon^{2m-2} \in K'
$$
for some $\lambda_1, \dots, \lambda_{m-1} \in \CC(\rho)$ such that $\lambda_1 \ne 0$.
\item By Corollary \ref{C:residuefields}, the algebra extension $\bar{A}_{\bar\idp} \subseteq \bar{R}_{\bar\idq}$ induces an isomorphism of the residue fields. Next, according to Lemma \ref{L:TaylorFractFields}, the morphism $\tilde\tau$ is an isomorphism. Therefore, we have an algebra extension
    $$
    K' \subseteq K = \CC(\rho)[\varepsilon^2]/(\varepsilon^{2\mu_\alpha}) \subset L = \CC(\rho)[\varepsilon]/(\varepsilon^{2\mu_\alpha}),
    $$
which induces an isomorphism of the corresponding residue fields.
\end{enumerate}
The last assumption implies  that for any $g \in \CC(\rho)$,  there exists an element $\gamma_g \in K'$ of the form
$\gamma_g = g + g_1 \varepsilon^2 + \dots + g_{\mu_\alpha-1}\varepsilon^{2\mu_\alpha-2}$.
If $\mu_\alpha = 1$ then we are done. Otherwise, if $\mu_\alpha \ge 2$, consider  the element
$
\bar\delta_\alpha^{\mu_\alpha-1} \cdot \gamma_g = \bigl(\lambda_1^{\mu_\alpha-1} g\bigr) \cdot \varepsilon^{2\mu_\alpha - 2} \in K'.
$
It follows  that for any $h \in  \CC(\rho)$ we have: $h \varepsilon^{2\mu_\alpha - 2} \in K'$. Proceeding inductively, we conclude that
$K' = K$.
\end{proof}

\begin{corollary}
We have the following commutative diagram of $\CC$--algebras:
\begin{equation}\label{E:gluingdiag}
\begin{array}{c}
\xymatrix{
           & & & K_{\alpha_1} \times \dots \times K_{\alpha_n}  \ar@{^{(}->}[ddd]^{\mathsf{diag}}\\
A \ar[rr] \ar@{_{(}->}[d] \ar[rrru] & & Q(\bar{A}) \ar[ru]_-{\cong} \ar@{^{(}->}[d]& \\
R \ar[rr]  \ar@/_10pt/[rrrd]_-{T} & & Q(\bar{R}) \ar[rd]^-{\cong} & \\
            &  & & L_{\alpha_1} \times \dots \times L_{\alpha_n} \\
}
\end{array}
\end{equation}
where $\bigl\{\alpha_1, \dots \alpha_n\bigr\} = \Pi$ and $T(f) = \bigl(T_{(\alpha_1, 2\mu_{\alpha_1})}(f), \dots, T_{(\alpha_n, 2\mu_{\alpha_n})}(f)\bigr)$ for any $f \in R$, whereas
$K_\alpha = \CC(\rho)[\varepsilon^2]/(\varepsilon^{2\mu_\alpha})$ and  $L_\alpha = \CC(\rho)[\varepsilon]/(\varepsilon^{2\mu_\alpha})$ for each $\alpha \in \Pi$.
\end{corollary}

\smallskip
\noindent
Now we are prepared to prove the following statement.
\begin{theorem}\label{T:CMandGorenstein}
The algebra of quasi--invariants $A = A\bigl(\Pi, \underline{\mu}\bigr)$ is Cohen--Macaulay and Gorenstein in codimension one.
More precisely, let $\idm \in \mathsf{Max}(A)$ be the maximal ideal, corresponding to any point of the surface $X \setminus \{p\}$, where $p = \nu(0,0)$
for the normalization map $\AA^2 \stackrel{\nu}\lar X$. Then the local ring $A_\idm$ is Gorenstein.
\end{theorem}

\begin{proof}
Let $A'$ be the Macaulayfication of $A$ (see Proposition \ref{P:MacaulayficationRings}) and $I' := \Hom_{A'}(R, A')$ be the corresponding conductor ideal.
Since $Q(A) = Q(A')$,  we have: $\Hom_{A'}(R, A') = \Hom_{A}(R, A')$. Moreover,  the embedding $A \stackrel{j}\lar A'$ induces a commutative diagram
$$
\xymatrix{
I \ar[r]^-\cong \ar@{_{(}->}[d]_-{j^\ast_{\mid}} & \Hom_A(R, A) \ar@{^{(}->}[d]^{j^\ast} \\
I' \ar[r]^-\cong & \Hom_A(R, A').
}
$$
Since the $\CC$--vector space $A'/A$ is finite dimensional, $j^\ast_\idp$ is an isomorphism for any $\idp \in \kP$. Therefore, the cokernel of $j^\ast_{\mid}$ is finite dimensional, too. On the other hand, both $A$--modules $I$ and $I'$ are Cohen--Macaulay; see Corollary \ref{C:ConductorDescription}
and \cite[Lemma 3.1]{SurvOnCM} respectively. From  \cite[Lemma 3.6]{SurvOnCM} we deduce  that $I = I'$.

Let $\bar{A}' := A'/I$. Then we have an algebra extension $\bar{A} \subseteq \bar{A}'$, where  both algebras $\bar{A}$ and $\bar{A}'$ are Cohen--Macaualy and $\dim_{\CC}\bigl(\bar{A}'/\bar{A}\bigr) < \infty$. This implies that $Q(\bar{A}) \cong Q(\bar{A}')$.
Next, we have the following commutative diagram:
\begin{equation}\label{E:keydiagram}
\begin{array}{c}
\xymatrix{
 A'\ar[r] \ar@{_{(}->}[d]& Q(\bar{A}') \ar[r]^-\cong \ar@{_{(}->}[d] & Q(\bar{A}) \ar@{_{(}->}[d] \ar[r]^-\cong & \bar{A}_{\bar{\idp}_1} \times \dots \times  \bar{A}_{\bar{\idp}_n} \ar@{^{(}->}[d]^-{\mathsf{diag}} \ar[r]^-\cong & K_{\alpha_1} \times \dots \times K_{\alpha_n} \ar@{^{(}->}[d]^-{\mathsf{diag}}\\
R \ar[r] & Q(\bar{R}) \ar[r]^-{=} &  Q(\bar{R}) \ar[r]^-{\cong} & \bar{R}_{\bar{\idq}_1} \times \dots \times \bar{R}_{\bar{\idq}_n}\ar[r]^-\cong &
L_{\alpha_1} \times \dots \times L_{\alpha_n}
}
\end{array}
\end{equation}
where $\idp_k := A \cap \idq_k$ for $1 \le k \le n$; compare with  diagram (\ref{E:gluingdiag}).

Let $U_\circ := \FF(A')$, where
$\CM(A') \stackrel{\FF}\lar \Tri(A')$ is the equivalence of categories from Theorem \ref{T:BurbanDrozd}. Clearly, $U_\circ =  \bigl(R, Q(\bar{A}'), \theta\bigr)$, where
$Q(\bar{A}') \stackrel{\theta}\lar Q(\bar{R})$ is the canonical inclusion. In the terms of  diagram (\ref{E:gluingdiag}) we have:
\begin{itemize}
\item $Q(\bar{A}') = \bigoplus_{\alpha \in \Pi} K_\alpha$,
\item $\theta = \bigl((1), \dots, (1)\bigr)$.
\end{itemize}
Observe that $A' \cong \Hom_{A'}(A', A') \cong \Hom_{\Tri(A')}(U_\circ, U_\circ)$.
Spelling out the definition (\ref{E:morphismtriples}) of morphisms in $\Tri(A')$, we obtain:
\begin{equation*}
\Hom_{\Tri(A')}\bigl(U_\circ, U_\circ)\bigr) =
\left\{f \in R \, \left| \, \forall \alpha \in \Pi
\begin{array}{c}
\xymatrix{
L_\alpha \ar[d]_{T_{(\alpha, 2\mu_\alpha)}(f)} & L_\alpha \ar[d]^-{g_\alpha} \ar[l]_-{1} \\
L_\alpha & L_\alpha  \ar[l]^-{1}
}
\end{array}
\right.
\;
\mbox{\rm for some} \; g_\alpha \in K_\alpha\right\}.
\end{equation*}
It follows that
$
R \supset A' =
\bigl\{
f \in R \, \big| \, T_{(\alpha, 2\mu_\alpha)}(f) \in \CC[\rho][\varepsilon^2]/(\varepsilon^{2\mu_\alpha}) \; \mbox{for all} \; \alpha \in \Pi
\bigr\} = A\bigl(\Pi, \underline{\mu}\bigr).
$
Hence, the algebra  $A$  is indeed Cohen--Macaulay.

It follows that the local ring $A_\idm$ is Cohen--Macaulay for any  $\idm \in \mathsf{Max}(A)$. Moreover, $A_\idm$ is Gorenstein if and only if its completion
$\widehat{A} := \widehat{A}_\idm$ is Gorenstein. Let $q \in X$ be the point corresponding to $\idm$. If $q$ is smooth then $\widehat{A} \cong \CC\llbracket
u, v\rrbracket$ is regular.

Now, assume that the point $q$ is singular and $\alpha \in \Pi$ is such that $q \in V(l_\alpha) \setminus\{p\}$. It follows from the formula $I = \Hom_A(R, A)$
that $\widehat{I} := \widehat{I}_\idm$ is the conductor ideal of the algebra extension $\widehat{A} \subseteq \widehat{R} = \CC\llbracket u, v\rrbracket$.
The diagram (\ref{E:keydiagram}) for the algebra $\widehat{A}$  has   the form
\begin{equation}
\begin{array}{c}
\xymatrix{
\widehat{A} \ar[r] \ar@{_{(}->}[d]&  \widehat{K}_\alpha= \CC\llbrace u\rrbrace[v^2]/(v^{2\mu_\alpha}) \ar@{^{(}->}[d] \\\
\widehat{R}  \ar[r] &  \widehat{L}_\alpha = \CC\llbrace u\rrbrace[v]/(v^{2\mu_\alpha}).
}
\end{array}
\end{equation}
Applying the same trick as in the proof of the Cohen--Macaulayness of $A$, we get
$$
\widehat{A} \cong \Hom_{\widehat{A}}\bigl(\widehat{A}, \widehat{A}\bigr) = \End_{\Tri(\widehat{A})}(U_\bu) \cong \CC\llbracket u, v^2, v^{2\mu_\alpha+1}\rrbracket,
$$
where $U_\bu := \bigl(\widehat{R}, \widehat{K}_\alpha, (1)\bigr)  \in \Tri(\widehat{A})$ is the triple corresponding to the regular module $\widehat{A}$.
Summing up,
$
\widehat{A} \cong \CC\llbracket u, v^2, v^{2\mu_\alpha+1}\rrbracket \cong \CC\llbracket u, z, t\rrbracket/(t^2 - z^{2\mu_\alpha +1})
$
is a hypersurface singularity. Hence, $\widehat{A}$ is Gorenstein. Summing up,  the algebra  $A$ is Gorenstein in codimension one.
\end{proof}

\begin{remark} As a consequence of Theorem \ref{T:CMandGorenstein}, we get the following statement: the algebra of planar quasi--invariants $A(\Pi, \underline{\mu}\bigr)$ is Gorenstein if and only if the its completion
\begin{equation}\label{E:quasiinvariants1}
\widehat{A}_p:= \Bigl\{
f \in \CC\llbracket z_1, z_2\rrbracket \; \big| \; l_\alpha^{2\mu_\alpha +1} \;\mbox{divides}\;  \bigl(f  - s_\alpha(f)\bigr) \; \mbox{for all}\; \alpha \in \Pi
\Bigr\}
\end{equation}
at its ``most singular'' point $p = \nu(0, 0)$
is Gorenstein.  Another proof of the fact that  $A(\Pi, \underline{\mu}\bigr)$ is Cohen--Macaulay can be found in the thesis of Johnston \cite[Theorem 3.3.2]{Johnston}.
\end{remark}

\section{Rank one Cohen--Macaulay modules over the algebra of planar quasi--invariants}
In this section, we classify  all Cohen--Macaualy $A$--modules of rank one, specifying  those of them, which are locally free in codimension one. Next, we give an explicit description of a canonical  module  of $A$. Finally, we describe the Picard group $\Pic(A)$ viewed as a subgroup of
the group $\CM^{\mathsf{lf}}_1(A)$, defined in Proposition \ref{P:CMPicard}.

\subsection{Description of the group $\CM^{\mathsf{lf}}_1(A)$}
For any  $\alpha \in \Pi$ we denote $K_\alpha := \CC(\rho)[\sigma]/(\sigma^{\mu_\alpha})$ and $L_\alpha := \CC(\rho)[\varepsilon]/(\varepsilon^{2\mu_\alpha})$.
It what follows, we shall  view $K_\alpha$ as a $\CC(\rho)$--subalgebra of $L_\alpha$ via the identification $\sigma = \varepsilon^2$.
Note that we have a direct sum decomposition $L_\alpha = K_\alpha \dotplus \varepsilon K_\alpha$.
The proof of the following lemma is a straightforward computation.

\begin{lemma}
For any $\gamma', \gamma'' \in K_\alpha$ we put: $\gamma' \circ \gamma'' := (\gamma' + \gamma'') \cdot (1 + \sigma \gamma' \cdot \gamma'')^{-1}$,
where $+$ and $\cdot$ are the usual addition and multiplication operations in the $\CC(\rho)$--algebra $K_\alpha$. Then we have: $(K_\alpha, \circ)$ is an abelian group. Note that for $\mu_\alpha = 0$,  $K_\alpha$ is simply the additive group of 
the field $\CC(\rho)$.
\end{lemma}

\smallskip
\noindent
For any $\alpha \in \Pi$ and $f \in R$, we define the following two elements $T_{(\alpha, 2\mu_\alpha)}^\pm(f) \in K_\alpha$:
\begin{equation}
\left\{
\begin{array}{ccl}
T_{(\alpha, 2\mu_\alpha)}^+(f) & = & \sum\limits_{j = 0}^{\mu_\alpha-1} f^{(2j)}_\alpha \frac{\displaystyle \sigma^j}{\displaystyle (2j)!} \\
T_{(\alpha, 2\mu_\alpha)}^-(f) & = & \sum\limits_{j = 0}^{\mu_\alpha-1} f^{(2j+1)}_\alpha \frac{\displaystyle \sigma^j}{\displaystyle (2j+1)!}.
\end{array}
\right.
\end{equation}
Note that $T_{(\alpha, 2\mu_\alpha)}(f) = T_{(\alpha, 2\mu_\alpha)}^+(f) + \varepsilon T_{(\alpha, 2\mu_\alpha)}^-(f) \in L_\alpha$ for any $f \in R$. Moreover,  $f \in A$ if and only if $T_{(\alpha, 2\mu_\alpha)}^-(f) = 0$ for all $\alpha \in \Pi$.

\begin{theorem}\label{T:listCMlf} There is an isomorphism of abelian groups
\begin{equation}\label{E:DefKPi}
\CM^{\mathsf{lf}}_1(A) \stackrel{\Theta}\lar K\bigl(\Pi, \underline{\mu}\bigr) := \bigoplus\limits_{\alpha \in \Pi} (K_\alpha, \circ)
\end{equation}
such that for  any element $\vec\gamma = (\gamma_\alpha)_{\alpha \in \Pi} \in K\bigl(\Pi, \underline{\mu}\bigr)$ we have:
\begin{equation}\label{E:cosnstraints}
B(\vec\gamma) := \Theta^{-1}(\vec\gamma) \cong
\bigl\{
f \in R \; \big| \; T_{(\alpha, 2\mu_\alpha)}^-(f) = \gamma_\alpha \cdot T_{(\alpha, 2\mu_\alpha)}^+(f) \;\;   \mbox{\rm for all} \;\;
\alpha \in \Pi
\bigr\}.
\end{equation}
\end{theorem}
\begin{proof} By Theorem \ref{T:BurbanDrozd}, we have an equivalence of categories $\CM^{\mathsf{lf}}(A) \stackrel{\FF}\lar
\Tri^{\mathsf{lf}}(A)$, preserving the monoidal structure on both sides. Let $U$ be an object of $\Tri^{\mathsf{lf}}(A)$ corresponding to a Cohen--Macaulay $A$--module of rank one. Then we have: $U = \bigl(R, \oplus_{\alpha \in \Pi} K_\alpha, (\theta_\alpha)_{\alpha \in \Pi}\bigr)$,
where $\theta_\alpha \in L_\alpha$ are some elements. Since the map $L_\alpha \stackrel{\theta_\alpha \cdot}\lar L_\alpha$ is an   isomorphism for any
$\alpha \in \Pi$, we conclude that all elements $\theta_\alpha$ are in fact invertible. Applying an appropriate automorphism of $K_\alpha$, we can  find
a \emph{uniquely determined} element $\gamma_\alpha \in K_\alpha$ such that
$$
U \cong U(\vec{\gamma}) := \bigl(R, \oplus_{\alpha \in \Pi} K_\alpha, (1 + \varepsilon \gamma_\alpha)_{\alpha \in \Pi}\bigr).
$$
Let $B(\vec{\gamma})$ be the unique (up to an isomorphism) element of the group $\CM^{\mathsf{lf}}_{1}(A)$ such that $\FF\bigl(B(\vec{\gamma})\bigr) \cong U(\vec{\gamma})$. By Theorem \ref{T:BurbanDrozd} we have:
$
\FF\bigl(B(\vec{\gamma}') \boxtimes_A B(\vec{\gamma}'')\bigr) \cong
U(\vec{\gamma}') \otimes U(\vec{\gamma}'') \cong $
$$
\bigl(R, \oplus_{\alpha \in \Pi} K_\alpha, \bigl((1 + \sigma \gamma_\alpha' \gamma_\alpha'') +
\varepsilon (\gamma_\alpha'  + \gamma_\alpha'')\bigr)_{\alpha \in \Pi}\bigr) \cong
\bigl(R, \oplus_{\alpha \in \Pi} K_\alpha, (\gamma_\alpha' \circ \gamma_\alpha'')_{\alpha \in \Pi}\bigr) =
U(\vec{\gamma}' \circ \vec{\gamma}'').
$$
This implies that $\Theta$ is indeed an isomorphism of abelian groups.
To get an explicit description of the module $B(\vec\gamma)$, observe that
$$
B(\vec{\gamma}) \cong \Hom_A\bigl(A, B(\vec{\gamma})\bigr) \cong \Hom_{\Tri(A)}\bigl(U_\circ, U(\vec{\gamma})\bigr),
$$
where $U_\circ := \FF(A) = \bigl(R, \oplus_{\alpha \in \Pi} K_\alpha, \bigl((1), \dots, (1)\bigr)\bigr)$. Writing down  the definition of morphisms
in the category  $\Tri(A)$, we conclude  that
\begin{equation*}
\Hom_{\Tri(A)}\bigl(U_\circ, U(\vec{\gamma})\bigr) =
\left\{f \in R \, \left| \, \forall \alpha \in \Pi
\begin{array}{c}
\xymatrix{
L_\alpha \ar[d]_{T_{(\alpha, 2\mu_\alpha)}(f)} & L_\alpha \ar[d]^-{g_\alpha} \ar[l]_-{1} \\
L_\alpha & L_\alpha  \ar[l]^-{1+ \varepsilon \gamma_\alpha}
}
\end{array}
\right.
\;
\mbox{\rm for some} \; g_\alpha \in K_\alpha\right\}.
\end{equation*}
It is easy to see that the constraints on  $f$ are precisely the ones given by (\ref{E:cosnstraints}).
\end{proof}

\subsection{Classification of all rank one Cohen--Macaulay $A$--modules} We begin with the following  preparatory results.

Let $\kk$ be any field, $m \in \NN$, $K = \kk[\sigma]/(\sigma^m)$ and $L = \kk[\varepsilon]/(\varepsilon^{2m})$. We view $K$ as a $\kk$--subalgebra
of $L$, identifying $\sigma$ with $\varepsilon^2$. For any $0 \le j \le m$ we put: $K_j := K/(\sigma^j)$ (in particular, $K_0 = 0$ and $K_m = K$).

\begin{lemma}\label{L:structureMP}
Let $V$ be a $K$--module and $V\stackrel{\widetilde\theta}\lar L$ an injective map of $K$--modules such that the adjoint map of $L$--modules
$L \otimes_K V \stackrel{\theta}\lar L$ is surjective. Then there exists $0 \le j \le m$ such that $V \cong K \oplus K_j$.
\end{lemma}
\begin{proof} Let $\widetilde{V} := L \otimes_K V$.  Then we have an isomorphism of $K$--modules $\widetilde{V} \cong V \oplus \varepsilon V$.
By Nakayama's Lemma, the map $\widetilde{V} \stackrel{\theta}\lar L$ is surjective if and only if the induced map
$\widetilde{V}/\varepsilon \widetilde{V} \stackrel{\bar\theta}\lar L/\varepsilon L$ is surjective. Note that  $\widetilde{V}/\varepsilon
\widetilde{V} \cong
V/\varepsilon^2 V = V/\sigma V$.

\smallskip
\noindent
Next, any $K$--linear map $K_j \stackrel{\widetilde\psi}\lar L$ is fully determined by the element $a = \widetilde\psi(\bar{1}) \in L$, which has to satisfy the condition
$\varepsilon^{2j} a = 0$, i.e. $a = \varepsilon^{2(m-j)} \tilde{a}$ for some $\tilde{a} \in L$. The induced map
$$
\kk \cong K_j/\sigma K_j \stackrel{\bar\psi}\lar  L/\varepsilon L \cong \kk
$$
sends $1$ to $a(0)$, i.e.~is zero for any $j < m$. Since any  finitely generated $K$--module $V$ splits into a finite direct sum of $K_j$-s, it  follows that $\theta$ can be surjective only if $V$ contains $K$ as a direct summand.

In the next step, we prove that the $K$--linear map $\widetilde\theta$ can be injective only if $V$ has at most two direct summands.
Let $\DD := \Hom_{\kk}(\,-\,,\kk): K-\mathsf{mod} \lar K-\mathsf{mod}$ be the Nakayama functor. Obviously, $\DD$ is an exact contravariant functor  and $\DD(K_j) \cong K_j$ for
 all $0 \le j \le m$. We have: $L \cong K^2$ and
 $$
 0 \lar V \stackrel{\widetilde\theta}\lar  K \oplus K  \quad  \mbox{\rm induces} \quad
 K \oplus K \stackrel{\widetilde\theta^\ast}\lar V^\ast \lar 0.
 $$
 It implies that $V^\ast$ has at most two direct summands, hence $V$ has at most two direct summands, too. Summing up, there exists $0 \le j \le m $ such that $V \cong K \oplus K_j$.
\end{proof}

For  $0 \le j \le m$ put  $V := K \oplus K_j$. Let $V \stackrel{\widetilde\theta}\lar L$ be a $K$--linear map and $L \otimes_K V \cong L\oplus L_j \stackrel{\theta}\lar L$ be its adjoint map,
where $L_j := L/(\varepsilon^{2j})$. Then both morphisms $\theta$ and $\widetilde{\theta}$ can be  presented by a matrix $(a\,|\, b) \in \Mat_{(1 \times 2)}(L)$, where $\varepsilon^{2j}b = 0$.

\begin{definition}
We  call two such maps $\theta, \theta' \in \Hom_L(L \otimes_K V, L)$ \emph{equivalent} if and only if
there exists an automorphism $\varphi \in \Aut_K(V)$ such that $\theta' = \theta \circ \varphi$.
\end{definition}

\smallskip
\noindent
We also assume that $\theta = (a\,|\, b) \in \Hom_L(L \otimes_K V, L)$ is surjective and the corresponding map $\widetilde{\theta} \in \Hom_K(V, L)$ is injective:
$$
\xymatrix{
K \oplus K_j \ar@{^{(}->}[rr]^-{(a\,|\, b)} \ar@{_{(}->}[d] & &  L \ar[d]^{=}\\
L \oplus L_j \ar@{->>}[rr]^-{(a\,|\, b)} & & L.
}
$$
It is clear that both these properties of the matrix $(a\,|\, b)$ are preserved when we replace it by an equivalent matrix. Let us  first treat  the following two ``boundary cases''.

\begin{lemma}\label{L:normalform} The following results are true.
\begin{enumerate}
\item Assume $j = 0$, i.e.~$V = K$. Then $\theta$ is equivalent to $(1 + \varepsilon \gamma)$ for some $\gamma \in K$. Moreover,
$(1 + \varepsilon \gamma) \sim (1 + \varepsilon \gamma')$ if and only if $\gamma = \gamma'$.
\item Assume $j = m$, i.e.~$V = K \oplus K$. Then $\theta$ is equivalent to $(1 \,\mid\, \varepsilon)$.
\end{enumerate}
\end{lemma}

\begin{proof}
In the first case we have: $\theta = (a)$ for some $a \in L$. The surjectivity of $\theta$ is equivalent to the condition  $a(0) \ne 0$, which also insures the injectivity of $\widetilde{\theta}$. Applying an appropriate automorphism of $K$, we get $(a) \sim (1 + \varepsilon \gamma)$, where  $\gamma \in K$ is uniquely determined.

In the second case, first observe that $\rk_K(L) = 2$, hence the surjectivity of $\theta$ is equivalent to its bijectivity (which in its turn, implies the injectivity of $\widetilde{\theta}$).  Since both elements $1, \varepsilon \in L$ belong to the image of the map $\theta$, we can transform
$\theta$ to $(1 \,\mid\, \varepsilon)$.
\end{proof}

\begin{proposition}\label{P:normalform} As above, let  $0 \le j \le m$ and $V  = K \oplus K_j$. Let $V \stackrel{\widetilde\theta}\lar L$ be an injective  $K$--linear map such that its adjoint map  $L \otimes_K V \cong L\oplus L_j \stackrel{\theta}\lar L$ is surjective. Then there exists
an element $\gamma = \alpha_0 + \alpha_1 \varepsilon^2 + \dots + \alpha_{m-j-1} \varepsilon^{2 (m-j-1)} \in L$ such that
\begin{equation}\label{E:MPnormalform}
\theta \sim \vartheta_\gamma:= (1 + \varepsilon \gamma \,\mid \, \varepsilon^{2 (m-j)+1}).
\end{equation}
Moreover, $\vartheta_\gamma \sim \vartheta_{\gamma'}$ if and only if $\gamma = \gamma'$.
\end{proposition}
\begin{proof} Let $\theta = (a \,\mid\, b)\in \Mat_{(1 \times 2)}(L)$, where $\varepsilon^{2j}b = 0$. By definition, the second component of $\theta$ is void if $j = 0$. Since the ``boundary cases''
$j = 0, m$ were already treated in Lemma \ref{L:normalform}, we can without loss of generality assume that $1 \le j \le m-1$.

According to the proof of Lemma \ref{L:structureMP}, the surjectivity of $\theta$ is equivalent to the non--vanishing $a(0) \ne 0$. Moreover, the action of the group $\Aut_K(V)$ leads to the following equivalence relations:
\begin{enumerate}
\item $(a \,\mid\, b) \sim (\lambda a \,\mid\, b) \sim (a \,\mid\, \lambda b)$ for any $\lambda \in K^\ast$;
\item $(a \,\mid\, b) \sim (a\,\mid\, \varepsilon^{2(m-j)}\nu a + b)$ for any $\nu \in K$;
\item $(a \,\mid\, b) \sim (a + \mu b\,\mid\, b)$ for any $\mu \in K$.
\end{enumerate}
Using transformations of the first type, we get:
$
(a \,\mid\, b) \sim  (1 + \varepsilon c \,\mid\, \varepsilon^{2(m-j)} d)
$
 for some $c \in K$ and $d = \beta_0 + \varepsilon \beta_1 + \dots + \varepsilon^{2j-1} \beta_{2j-1} \in L$. Using an appropriate transformation of the second type, we can kill all coefficients $\beta_0, \beta_2, \dots, \beta_{2j-1}$ (i.e.~entries at $1, \varepsilon^2,  \dots, \varepsilon^{2j-2}$ of the element $d$). In other words,
$
\theta \sim \theta' = (1 + \varepsilon c \,\mid\, \varepsilon^{2(m-j)+1} e)
$
for a certain  $e =\xi_0 + \varepsilon^2 \xi_1 + \dots + \varepsilon^{2(j-1)} \xi_{j-1} \in K \subset L$. Now observe that for
$ x= (0, \varepsilon^{2(j-1)}) \in K \oplus K_j$ we have:
$\widetilde{\theta}'(x) = \varepsilon^{2m-1} \xi_0$. Since the map $\widetilde{\theta}'$ is injective, we conclude that $\xi_0 \ne 0$, i.e.~
$e \in K$ is a unit. Hence, we get:
$
\theta \sim \theta' \sim  (1 + \varepsilon c \,\mid\, \varepsilon^{2(m-j)+1}).
$
Finally, using an appropriate transformation of the third type, we can kill all entries at $\varepsilon^{2(m-j)}, \dots, \varepsilon^{2(m-1)}$ of the element $c$ and end up with  a normal form $\theta \sim \vartheta_\gamma$ as in (\ref{E:MPnormalform}).

It is not difficult  to see that the $K$--linear map $\widetilde\theta$ corresponding to the $L$--linear map $\theta = \vartheta_\gamma$ given by the formula (\ref{E:MPnormalform}), is injective for any
$\gamma \in L$ as in the statement of Proposition. Next, consider the following
$K$--modules
$$
\left\{
\begin{array}{ccl}
t(V) & = &  \bigl\{v \in V \,|\, \sigma^j v = 0\bigr\} = \sigma^{m-j} K \oplus K_j \\
t(L) & = &  \bigl\{u \in L \,|\, \varepsilon^{2j} u = 0\bigr\} = \varepsilon^{2(m-j)} L.
\end{array}
\right.
$$
Since the submodule  $t(V)$ is mapped to itself under arbitrary automorphisms of $V$, we obtain  an induced map
$V/t(V) \stackrel{\bar{\vartheta}_\gamma}\lar L/t(L)$ such that
$$
\xymatrix{
V/t(V) \ar[rr]^-{\bar{\vartheta}_\gamma} \ar[d]_-{\cong} & &  L/t(L) \ar[d]^-{\cong}\\
\kk[\sigma]/(\sigma^{m-j}) \ar[rr]^-{1 + \varepsilon \gamma} & & \kk[\varepsilon]/(\varepsilon^{2(m-j)}).
}
$$
Hence, $\vartheta_\gamma \sim \vartheta_{\gamma'}$ if and only if $\gamma = \gamma'$.
\end{proof}

\smallskip
\noindent
Now we are ready to prove the main result of this subsection.

\begin{theorem}\label{T:listCM} For  $\vec{\nu} =\bigl(\nu_\alpha\bigr)_{\alpha \in \Pi} \in \NN_0^n$ such that $0 \le \nu_\alpha \le \mu_\alpha$ and
$\vec{\gamma} =\bigl(\gamma_\alpha\bigr)_{\alpha \in \Pi}$ such that $\gamma_\alpha \in \CC(\rho)[\sigma]/(\sigma^{\mu_\alpha - \nu_\alpha})$ for  any $\alpha \in \Pi$ we put:
\begin{equation}\label{E:classificationCM}
M(\vec{\nu}, \vec\gamma) :=
\bigl\{
f \in R \; \big| \; T_{(\alpha, 2(\mu_\alpha-\nu_\alpha))}^-(f) = \gamma_\alpha \cdot T_{(\alpha, 2(\mu_\alpha-\nu_\alpha))}^+(f) \quad \mbox{\rm for all} \quad
\alpha \in \Pi
\bigr\}.
\end{equation}
Then the following results are true.
\begin{itemize}
\item $M(\vec{\nu}, \vec\gamma)$ is a Cohen--Macaulay $A$--module of rank one.
\item Conversely, any Cohen--Macaulay $A$--module of rank one is isomorphic to some $M(\vec{\nu}, \vec\gamma)$ for appropriate parameters
$\vec{\nu}, \vec\gamma$ as above.
\item $M(\vec{\nu}, \vec\gamma) \cong M(\vec{\nu}', \vec{\gamma}')$ if any only if $\vec{\nu} = \vec{\nu}'$ and $\vec\gamma = \vec{\gamma}'$.
\end{itemize}
\end{theorem}
\begin{proof}
According to Theorem \ref{T:BurbanDrozd}, the isomorphism classes of Cohen--Macaulay $A$--modules stand in bijection with the isomorphism classes of objects of the category of triples $\Tri(A)$. Let $U = (\widetilde{M}, V, \theta)$ be  a rank one object of $\Tri(A)$ (i.e.~an object corresponding to a Cohen--Macaulay $A$--module of rank one) then $\widetilde{M} \cong R$. Moreover, by Lemma \ref{L:structureMP}, there exists a uniquely determined vector $\vec{\nu} \in \NN_0^n$ as above such that
$$
V  \cong \bigoplus\limits_{\alpha \in \Pi} V_\alpha  =  \bigoplus\limits_{\alpha \in \Pi} K_\alpha \oplus \bigl(K_\alpha/(\sigma^{\mu_\alpha - \nu_\alpha})\bigr).
$$
Next, Proposition \ref{P:normalform} implies that there exists  an automorphism of the triple $U$ transforming every component
of the gluing map $\theta$ into the canonical form
$$
\theta_\alpha =  (1 + \varepsilon \gamma_\alpha \,\mid \, \varepsilon^{2 (\mu_\alpha-\nu_\alpha)+1})
$$
for an appropriate  vector $\vec\gamma =\bigl(\gamma_\alpha\bigr)_{\alpha \in \Pi}$ as above. Since $\Aut_R(R) = \CC^\ast$, in order to describe the  isomorphism classes of rank one objects of $\Tri(A)$
it is sufficient to take into account  only the action of  the groups $\Aut_{K_\alpha}(V_\alpha)$ on the matrices $\theta_\alpha$.
 Proposition \ref{P:normalform} then insures  that the vector $\vec{\gamma}$ is in fact uniquely determined.

\smallskip
\noindent
 Summing up, consider the following object of the category $\Tri(A)$:
$$
U(\vec{\nu}, \vec\gamma) := \Bigl(R, \oplus_{\alpha \in \Pi} \bigl(K_\alpha \oplus K_\alpha/(\sigma^{\mu_\alpha - \nu_\alpha})\bigr),
(1 + \varepsilon \gamma_\alpha \,\mid \, \varepsilon^{2 (\mu_\alpha-\nu_\alpha)+1})_{\alpha \in \Pi}\Bigr).
$$
Then the following results are true.
\begin{itemize}
\item Any rank one object of $\Tri(A)$  is isomorphic to some $U(\vec{\nu}, \vec\gamma)$ for appropriate parameters
$\vec{\nu}, \vec\gamma$ as in the statement of the theorem.
\item $U(\vec{\nu}, \vec\gamma) \cong U(\vec{\nu}', \vec{\gamma}')$ if any only if $\vec{\nu} = \vec{\nu}'$ and $\vec\gamma = \vec{\gamma}'$.
\end{itemize}
Let $M(\vec{\nu}, \vec\gamma) := \FF^{-1}\bigl(U(\vec{\nu}, \vec\gamma)\bigr)$, then we have:
$
M(\vec{\nu}, \vec\gamma)  \cong \Hom_{\Tri(A)}\bigl(U_\circ, U(\vec\nu, \vec{\gamma})\bigr),
$
where $U_\circ := \FF(A) = \bigl(R, \oplus_{\alpha \in \Pi} K_\alpha, \bigl((1), \dots, (1)\bigr)\bigr)$.  Following  formula (\ref{E:morphismtriples}) we have: $f \in R$ belongs to $\Hom_{\Tri(A)}\bigl(U_\circ, U(\vec{\gamma})\bigr)$
if and only if for any $\alpha \in \Pi$ there exists a $K_\alpha$--linear map $K_\alpha \stackrel{\left(\begin{smallmatrix} g_\alpha\\h_\alpha \end{smallmatrix}\right)}\lar K_\alpha \oplus \bigl(K_\alpha/(\sigma^{\mu_\alpha -\nu_\alpha)})\bigr)$ such that
\begin{equation}\label{E:constraints}
\begin{array}{c}
\xymatrix{
L_\alpha \ar[d]_{T_{(\alpha, 2\mu_\alpha)}(f)} & & & & L_\alpha \ar[d]^-{\left(\begin{smallmatrix} g_\alpha\\h_\alpha \end{smallmatrix}\right)} \ar[llll]_-{1} \\
L_\alpha & & & & L_\alpha \oplus \bigl(L_\alpha/(\varepsilon^{2(\mu_\alpha-\nu_\alpha)})\bigr) \ar[llll]_-{(1+ \varepsilon \gamma_\alpha\,|\, \varepsilon^{2(\mu_\alpha -\nu_\alpha)+1})}.
}
\end{array}
\end{equation}
 Writing down explicitly the constraints (\ref{E:constraints}), we end up with the description (\ref{E:classificationCM}).
\end{proof}

\begin{corollary}\label{C:remarkClassifCM} Let $(\vec{\nu}, \vec{\gamma})$ be as in Theorem \ref{T:listCM}. Then the following  results are true.
\begin{itemize}
\item The Cohen--Macaulay module $M(\vec\nu, \vec{\gamma})$ is locally free in codimension one if and only if $\vec{\nu} = \vec{0}$. In  notations
of Theorem \ref{T:listCMlf} we have: $M(\vec{0}, \vec{\gamma}) = B(\vec{\gamma})$.
\item Consider the weight function $\Pi \stackrel{\underline{\mu}'}\lar \NN_0$ given by the rule: $\underline{\mu}'(\alpha) = \mu_\alpha - \nu_\alpha$ for any $\alpha \in \Pi$.
Let $A' := A\bigl(\Pi, \underline{\mu}'\bigr)$ be the corresponding algebra of quasi--invariant polynomials. Theorem  \ref{T:listCMlf} implies that $M(\vec\nu, \vec{\gamma})$ is a Cohen--Macaulay $A'$--module of rank one, locally free in codimension one.
\end{itemize}
\end{corollary}

\begin{example}
Consider the special case of a constant multiplicity  function $\Pi \stackrel{\underline{\mu}}\lar \NN_0$ given by the rule: $\mu_\alpha = 1$ for all $\alpha \in \Pi$. Then  the classification of   Cohen--Macaulay $A$--modules of rank one takes the following form.
\begin{itemize}
\item Any object  of $\CM_1^{\mathsf{lf}}(A)$ is isomorphic to some
\begin{equation}\label{E:listCMwhtone}
B(\vec{\gamma}) =
\bigl\{
f \in R \,|\, f'_{\alpha} = \gamma_\alpha f_\alpha \quad \mbox{\rm for all}\quad \alpha \in \Pi
\bigr\},
\end{equation}
where $\vec{\gamma} = (\gamma_\alpha)_{\alpha \in \Pi} \in \CC(\rho)^{\oplus n}$.
\item Moreover, $B(\vec{\gamma}) \cong B(\vec{\gamma}')$ if and only if $\vec{\gamma} = \vec{\gamma}'$ and
$
B(\vec{\gamma}) \boxtimes_A B(\vec{\gamma}') \cong B(\vec{\gamma}+\vec{\gamma}').
$
Note that $A = B(\vec{0})$.
\item We get the full list of objects of $\CM_1(A)$ by the following rule: one  takes any non--empty subset $\Pi^\circ \subseteq \Pi$ and  just omits   the conditions in (\ref{E:listCMwhtone}) for $\alpha \in \Pi^\circ$. In particular, for $\Pi^\circ =  \Pi$ (no conditions on $f$), we get the module $R$.
\end{itemize}
\end{example}

\subsection{Canonical module of the algebra of dihedral quasi--invariants}
It was already pointed  out by Etingof and Ginzburg in \cite[Section 6]{EtingofGinzburg} that the algebra $A= A\bigl(\Pi, \underline{\mu}\bigr)$ is not Gorenstein for a general weighted line arrangement $(\Pi, \underline{\mu})$. However, $A$ is a finitely generated Cohen--Macaulay algebra (see  Theorem \ref{T:CMandGorenstein}), hence it has a canonical module $\Omega$ (which is a Cohen--Macaualy $A$--module of rank one, uniquely determined up to a tensoring with an element of $\Pic(A)$). It is a natural question to describe  $\Omega$ in terms of our classification.

\smallskip
\noindent
We give an explicit description of $\Omega$ in the so--called \emph{Coxeter} (or \emph{dihedral} case), when
\begin{equation}
\Pi = \Lambda_n:= \left\{0, \frac{1}{n} \pi, \dots, \frac{n-1}{n} \pi \right\}
\end{equation}
for some $n \in \NN$.
For $m := \max\bigl\{\mu_\alpha \; \big| \; \alpha \in \Pi\bigr\}$ let  $\Lambda_n \stackrel{\underline{\kappa}}\lar \NN_0,\; \alpha \mapsto m$ be the corresponding constant multiplicity function and $C := A(\Lambda_n, \underline{\kappa}\bigr)$ the corresponding ring of quasi--invariant polynomials. By \cite[Corollary 5.6]{FeiginVeselov} (see also \cite[Theorem 1.2]{EtingofGinzburg}),  the algebra $C$ is  a graded Gorenstein domain.  In particular, $C$ viewed as a module over itself, is a canonical  module of $C$; see for instance \cite[Section I.3.6]{BrunsHerzog}. Therefore,
\begin{equation}\label{E:CanModule}
\Omega :=  \Hom_{C}(A, C)
\end{equation}
is a canonical  module of $A$, see for instance \cite[Section X.9.3]{Bourbaki}.
The following theorem is the main result of this subsection.

\begin{theorem}\label{T:canonicalmodule}
For any $\alpha \in \Lambda_n$ put: $\nu_\alpha := m - \mu_\alpha$. Then we have:
\begin{equation}\label{E:canmodule}
 \Omega \cong \bigl\{f \in C \,\big|\, f^{(0)}_\alpha = f^{(2)}_\alpha = \dots = f^{(2(\nu_\alpha-1))}_{\alpha} = 0 \; \mbox{\rm for all} \; \alpha \in \Lambda_n\bigr\}.
\end{equation}
\end{theorem}

\begin{remark}
It is a pleasant exercise in elementary calculus to verify directly  that the  right hand side of the expression (\ref{E:canmodule}) is an ideal in the algebra $A$ given by (\ref{E:quasiinvpolar}).
\end{remark}
\begin{proof} For any $\alpha \in \Lambda_n$ we put: $\widetilde{K} := \CC(\rho)[\sigma]/(\sigma^m)$ and $\widetilde{L} := \CC(\rho)[\varepsilon]/(\varepsilon^{2m})$.
Then we have the following

\smallskip
\noindent
\underline{Claim}. Let $U_\bu := \FF(A)$, where $\CM(C) \stackrel{\FF}\lar \Tri(C)$ is the  equivalence of categories from Theorem \ref{T:BurbanDrozd}. Then we have: $U_\bu \cong (R, V, \theta)$, where $V = \oplus_{\alpha \in \Lambda_n} V_\alpha$ with
$$
\quad \quad
V_\alpha =
\left\{
\begin{array}{ccl}
\widetilde{K} & \mbox{\rm if} & \mu_\alpha = m \\
\widetilde{K} \oplus \bigl(\widetilde{K}/(\sigma^{\nu_\alpha})\bigr) & \mbox{\rm if} & 0 \le \mu_\alpha \le m-1
\end{array}
\right.
$$
and
$$
\theta_\alpha =
\left\{
\begin{array}{ccl}
1 & \mbox{\rm if} & \mu_\alpha = m \\
(1 \, |\, \varepsilon^{2\mu_\alpha+1}) & \mbox{\rm if} & 0 \le \mu_\alpha \le m-1.
\end{array}
\right.
$$
We prove this claim by computing  the morphism space $\Hom_{\Tri(C)}(U_\circ, U_\bu)$, where
$U_\circ$ is the canonical triple corresponding to the regular module $C$. As in the proof of Theorem \ref{T:listCM}, we get:
$f \in R$ belongs to $ \Hom_{\Tri(C)}(U_\circ, U_\bu)$
if and only if for any $\alpha \in \Lambda_n$ there exists a $\widetilde{K}$--linear map $\widetilde{K} \stackrel{\left(\begin{smallmatrix} g_\alpha\\h_\alpha \end{smallmatrix}\right)}\lar \widetilde{K} \oplus \bigl(\widetilde{K}/(\sigma^{\nu_\alpha})\bigr)$ making the diagram
\begin{equation}\label{E:constraints3}
\begin{array}{c}
\xymatrix{
\widetilde{L} \ar[d]_{T_{(\alpha, 2m)}(f)} & &  \widetilde{L} \ar[d]^-{\left(\begin{smallmatrix} g_\alpha\\h_\alpha \end{smallmatrix}\right)} \ar[ll]_-{1} \\
\widetilde{L} & & \widetilde{L} \oplus \bigl(\widetilde{L}/(\varepsilon^{2\nu_\alpha})\bigr) \ar[ll]_-{(1 \,|\, \varepsilon^{2\mu_\alpha+1})}.
}
\end{array}
\end{equation}
commutative, i.e. $T_{(\alpha, 2m)}(f) = g_\alpha + \varepsilon^{2\mu_\alpha+1} h_\alpha$ for some $g_\alpha, h_\alpha \in \widetilde{L}$.
This condition is equivalent to the vanishing $f^{(2l-1)}_{\alpha} = 0$ for any $\alpha \in \Lambda_n$ and $1 \le l \le \mu_\alpha$. Hence, $(R, V, \theta)$ is indeed
isomorphic to $\FF(A)$, as asserted (compare with Corollary \ref{C:DescriptQuasiInv}).

Now, in virtue of (\ref{E:CanModule}), we have an isomorphism:
$
\Omega \cong \Hom_{\Tri(C)}(U_\bu, U_\circ).
$
A polynomial $f \in R$ belongs to the vector space $\Hom_{\Tri(C)}(U_\bu, U_\circ)$ if and only if for every $\alpha \in \Lambda_n$ there exist
elements $p_\alpha, q_\alpha \in \widetilde{K}$, making the diagram
\begin{equation}\label{E:constraints2}
\begin{array}{c}
\xymatrix{
\widetilde{L} \ar[d]_{T_{(\alpha, 2m)}(f)} & & \widetilde{L} \oplus \bigl(\widetilde{L}/(\varepsilon^{2\nu_\alpha})\bigr) \ar[ll]_-{(1 \,|\, \varepsilon^{2\mu_\alpha+1})} \ar[d]^-{\left(\begin{smallmatrix} p_\alpha | \varepsilon^{2 \mu_\alpha} q_\alpha \end{smallmatrix}\right)} \\
\widetilde{L}  & &  \widetilde{L}  \ar[ll]_-{1} \\
}
\end{array}
\end{equation}
commutative. In other words, for any $\alpha \in \Lambda_n$ there exist
 $p_\alpha, q_\alpha \in \widetilde{K}$ such that
$$
\bigl(T_{(\alpha, 2m)}(f) \,|\, \varepsilon^{2\mu_\alpha+1} T_{(\alpha, 2m)}(f)\bigr) = \bigl(p_\alpha \,|\, \varepsilon^{2\mu_\alpha} q_\alpha\bigr).
$$
The first condition $T_{(\alpha, 2m)}(f) = p_\alpha$ just means that $T_{(\alpha, 2m)}(f) \in \widetilde{K}$ for any $\alpha \in \Lambda_n$, i.e.~$f$ is an element of the algebra $C$. The second condition $\varepsilon^{2\mu_\alpha+1} T_{(\alpha, 2m)}(f) = \varepsilon^{2\mu_\alpha} q_\alpha$
is equivalent to the vanishing
$
f^{(0)}_\alpha = f^{(2)}_\alpha = \dots = f^{(2(\nu_\alpha-1))}_{\alpha} = 0$ for all $\alpha \in \Lambda_n$.
\end{proof}

\noindent
As a further refinement  of  Theorem \ref{T:canonicalmodule}, we have the following result.
\begin{lemma}\label{L:mixedweight}
For any $k \in \NN_0$, consider the  multiplicity  function $\Lambda_n \stackrel{\underline{\mu}}\lar \NN_0$ given by the rule: $\mu_0 = k+1$  and
$\mu_\alpha = 1$ for any  $\alpha \in \Lambda_n^\circ := \Lambda_n \setminus \{0\}$. Let $\Omega$ be a canonical   module of the algebra $A = A\bigl(\Lambda_n, \underline{\mu}\bigr)$ given by (\ref{E:canmodule}). Then we have:
$$
\Omega  \cong \left\{
g \in R \left|
\begin{array}{l}
g_0' = g'''_0 = \dots = g^{(2k+1)}_0 = 0 \\
\left(\dfrac{g}{l_0^{2k}}\right)'_{\alpha} = 0 \quad \mbox{\rm for all}\quad \alpha \in \Lambda^\circ_n.
\end{array}
\right.
\right\},
$$
where $l_0 = x_1 = \rho \cos(\varphi)$.
\end{lemma}
\begin{proof} In the notation of Theorem \ref{T:canonicalmodule} we have: $m = k+1$, $\nu_\alpha = k$ for each $\alpha \in \Lambda^\circ_n$. Let $C = A\bigl(\Lambda_n, \underline{\kappa}\bigr)$, then
$$
\Omega :=  \bigl\{f \in C \,\big|\, f^{(0)}_\alpha = f^{(2)}_\alpha = \dots = f^{(2k-2)}_{\alpha} = 0 \; \mbox{\rm for all} \; \alpha \in \Lambda^\circ_n\bigr\}
$$
is a canonical  module  of $A$.
More explicitly, for any $\alpha \in \Lambda^\circ_n$ and $f \in \Omega$ we have:
$$
f^{(l)}_\alpha = 0 \quad \mbox{for any} \quad 0 \le l \le 2k-1 \quad \mbox{and} \quad f^{(2k+1)}_\alpha = 0.
$$
The first condition is equivalent to the statement that $l_{\alpha}^{2k} \,|\, f$ for every $\alpha \in \Lambda^\circ_n$. Let
$$\delta_\circ := \prod\limits_{\alpha \in \Lambda_n^\circ} l_\alpha^{2 k} = \rho^{2k(n-1)} \Bigl(\prod_{l=1}^{n-1} \sin\Bigl(\varphi - \frac{l}{n}\pi\Bigr)\Bigr)^{2k}.
$$
Then there exists a uniquely determined $g \in R$ such that $f = g \cdot \delta_\circ$. Note that  $\delta_\circ(\rho, -\varphi) = \delta_\circ(\rho, \varphi)$, hence
$(\delta_\circ)^{(2p+1)}_{0} = 0$ for all $p \in \NN_0$. It is not difficult to show by induction that the condition
$
f^{(1)}_{0} = f^{(3)}_{0} = \dots = f^{(2k+1)}_{0} = 0
$
(recall  that $f \in C = A(\Lambda_n, \underline{\kappa})$)
is equivalent to
$
g^{(1)}_{0} = g^{(3)}_{0}
= \dots = g^{(2k+1)}_{0} = 0.
$

\smallskip
\noindent
Finally, it remains to interpret the constraint  $f^{(2k+1)}_{\alpha} = 0$ for $\alpha \in \Lambda_n^\circ$ in terms of the polynomial $g$.  Since
$(\delta_\circ)^{(l)}_{\alpha} = 0$ for all $0 \le l \le 2k-1$, we have:
\begin{equation}\label{E:constraint1}
f^{(2k+1)}_{\alpha} = (2k+1) (\delta_\circ)^{(2k)}_{\alpha} g'_\alpha + (\delta_\circ)^{(2k+1)}_{\alpha} g_\alpha = 0.
\end{equation}
As usual, we put
$
\delta^2:= \prod\limits_{\alpha \in \Lambda_n} l_\alpha^{2 k} = l_0^{2k}  \cdot \delta_\circ.
$
Then we have: $\delta^2(\rho, -\varphi) = \delta^2(\rho, \varphi)$, hence $\delta^{(2p+1)}_{0} = 0$ for all $p \in \NN_0$. Moreover,
$\delta^2(\rho, \varphi + \alpha) = \delta^2(\rho, \varphi)$ for any $\alpha \in \Lambda_n$ (here, we essentially use the fact that the image of the set
$\Lambda_n$ in $\RR/\pi\ZZ$ is a subgroup). Therefore, $\delta^{(2p+1)}_{\alpha} = 0$ for any
$\alpha \in \Lambda_n^\circ$ and  $p \in \NN_0$. In particular, we get:
\begin{equation}\label{E:constraint2}
\delta^{(2k+1)}_{\alpha} = (2k+1) (\delta_\circ)^{(2k)}_{\alpha} (l_0^{2k})'_\alpha + (\delta_\circ)^{(2k+1)}_{\alpha} (l_0^{2k})_\alpha = 0.
\end{equation}
Comparing the equations (\ref{E:constraint1}) and (\ref{E:constraint2}), we see  that the condition  $f^{(2k+1)}_{\alpha} = 0$ is equivalent to
$\left(\dfrac{g}{l_0^{2k}}\right)'_{\alpha} = 0$. Summing up, we obtain:
$$
\Omega  =  \delta_\circ \cdot \left\{
 g \in R \left|
\begin{array}{l}
g_0' = g'''_0 = \dots = g^{(2k+1)}_0 = 0 \\
\left(\dfrac{g}{l_0^{2k}}\right)'_{\alpha} = 0 \quad \mbox{\rm for all}\quad \alpha \in \Lambda^\circ_n
\end{array}
\right.
\right\},
$$
what implies the result.
\end{proof}

\begin{remark} According to a result of Feigin and Johnston \cite[Theorem 7.14]{FeiginJohnston}, the algebra of quasi--invariants from Lemma \ref{L:mixedweight} is Gorenstein if and only if $k = 0$, what matches with our description of a canonical   module.
\end{remark}

\subsection{Picard group of the algebra of quasi--invariants $A(\Pi, \underline{\mu})$} In this subsection, we describe the Picard group
$\Pic(A)$ of the algebra of planar quasi--invariants $A = A(\Pi, \underline{\mu})$.

\smallskip
\noindent
We begin with the following elementary observation.
\begin{lemma}\label{L:ArmAberSexy}
Let $\kk$ be any ring, $m \in \NN$,  $K = \kk[\varepsilon^2]/(\varepsilon^{2m})$ and $L = \kk[\varepsilon]/(\varepsilon^{2m})$. Then for any
$g \in L$, there exists a unique element $\gamma(g) \in K$ such that $g \cdot (1+ \varepsilon \gamma(g)) \in K$.
\end{lemma}

\smallskip
\noindent
Let $\widehat{R} := \CC\llbracket z_1, z_2\rrbracket = \CC\llbracket \rho\cos(\varphi), \rho \sin(\varphi)\rrbracket$. For any  $h \in \widehat{R}$ and a pair $(\alpha, k)  \in \CC \times \NN$, we define
the power series $\widehat{T}_{(\alpha, k)}(h) \in \CC\llbracket \rho\rrbracket[\varepsilon]/(\varepsilon^{2\mu_\alpha})$ by the rule
\begin{equation}\label{L:TaylorSeriesVariation}
T_{(\alpha, k)}\bigl(\exp(h)\bigr) =  \widehat{T}_{(\alpha, k)}(h) \cdot \exp\bigl(h_\alpha^{(0)}\bigr).
\end{equation}
It follows from the definition, that for any $h_1, h_2 \in \widehat{R}$ we have:
\begin{equation}
\widehat{T}_{(\alpha, k)}(h_1 + h_2) = \widehat{T}_{(\alpha, k)}(h_1) \cdot \widehat{T}_{(\alpha, k)}(h_2).
\end{equation}
\begin{lemma}\label{L:prepCompPicard}
For $(\alpha, m) \in \CC \times \NN$, let $\widehat{R} \xrightarrow{\Upsilon_{(\alpha, m)}} \CC\llbracket \rho \rrbracket[\sigma]/(\sigma^m)$
be given by the composition
\begin{equation}
\widehat{R} \xrightarrow{\widehat{T}_{(\alpha, 2m)}}  \CC\llbracket \rho\rrbracket[\varepsilon]/(\varepsilon^{2m}) \stackrel{\gamma}\lar \CC\llbracket \rho \rrbracket[\sigma]/(\sigma^m),
\end{equation}
where $\gamma$ is the map from Lemma \ref{L:ArmAberSexy}. Then we have: $(\widehat{R}, +) \xrightarrow{\Upsilon_{(\alpha, m)}} \bigl(\CC\llbracket \rho \rrbracket[\sigma]/(\sigma^m), \circ\bigr)$ is a group homomorphism.
\end{lemma}
\begin{proof}
Let $h_k \in \widehat{R}$ and $\gamma_k := \Upsilon_{(\alpha, m)}(h_k)$ for $k = 1, 2$. By definition, we have:
$$
(1 + \varepsilon \gamma_k) \cdot \widehat{T}_{(\alpha, 2m)}(h_k) \in \CC\llbracket
\rho\rrbracket[\varepsilon^2]/(\varepsilon^{2m}) \quad \mbox{\rm for} \quad k = 1,2.
$$
Note that we have the following identity in $\CC\llbracket
\rho\rrbracket[\varepsilon^2]/(\varepsilon^{2m})$:
$$
(1 + \varepsilon \gamma_1) \cdot (1 + \varepsilon \gamma_2) \cdot \widehat{T}_{(\alpha, 2m)}(h_1) \cdot  \widehat{T}_{(\alpha, 2m)}(h_2) =
\bigl((1 + \varepsilon^2 \gamma_1 \gamma_2) + \varepsilon (\gamma_1 + \gamma_2)\bigr)\cdot  \widehat{T}_{(\alpha, 2m)}(h_1 + h_2).
$$
Then we have:
$
\widehat{T}_{(\alpha, 2m)}(h_1 + h_2) \cdot\bigl(1 + \varepsilon (\gamma_1 + \gamma_2) \cdot (1 + \varepsilon^2 \gamma_1 \gamma_2)^{-1}\bigr)$
belongs to  $\CC\llbracket
\rho\rrbracket[\varepsilon^2]/(\varepsilon^{2m})$. It follows from the definition of the operation $\circ$ that
$$
\Upsilon_{(\alpha, m)}(h_1 + h_2) = \Upsilon_{(\alpha, m)}(h_1) \circ \Upsilon_{(\alpha, m)}(h_2),
$$
proving the statement.
\end{proof}

\begin{lemma}\label{L:TriplesSpecialCase} Let $B = \CC\llbracket u, v^2, v^{2m+1}\rrbracket$ for some $m \in \NN$.  Then its normalization
is $\widehat{R} = \CC\llbracket u, v\rrbracket$ and the diagram (\ref{E:keydiagram}) has   the form
\begin{equation}
\begin{array}{c}
\xymatrix{
B \ar[r] \ar@{_{(}->}[d]& \widehat{K} = \CC\llbrace u\rrbrace[v^2]/(v^{2m}) \ar@{^{(}->}[d] \\\
\widehat{R}  \ar[r] &  \widehat{L} = \CC\llbrace u\rrbrace[v]/(v^{2m})
}
\end{array}
\end{equation}
for some $m \in \NN$.
Let $U = \bigl(\widehat{R}, \widehat{K}, (1+ \bar{v}\gamma)\bigr)$ be an object of $\Tri(B)$ for some $\gamma \in
\CC\llbrace u\rrbrace[v^2]/(v^{2m})$. Then we have: $\FF(B) \cong U$ if and only if $\gamma \in
\CC\llbracket u\rrbracket[v^2]/(v^{2m})$.
\end{lemma}
\begin{proof}
Recall that $\FF(B) = U_\circ := (\widehat{R}, \widehat{K}, 1)$.
Let $U = \bigl(\widehat{R}, \widehat{K}, (1+ \bar{v}\gamma)\bigr)$ be such that  $\gamma \in
\CC\llbracket u\rrbracket[v^2]/(v^{2m})$. Then we have an expansion
$
\gamma = \sum\limits_{j = 0}^{m-1} \gamma_j(u) \bar{v}^{2j}$, where $\gamma_j \in \CC\llbracket u\rrbracket$ for all
$0 \le j \le m-1$. Let $g := 1 + v \cdot \bigl(\sum\limits_{j = 0}^{m-1} \gamma_j v^{2j}\bigr) \in \widehat{R}$. Then $g$ is a unit and
$\bar{g}^{-1} \cdot (1 + \bar{v} \gamma) = 1$ in $\widehat{L}$. Hence $U_\circ \xrightarrow{(g, 1)} U$ is an isomorphism in the category $\Tri(B)$.

On the other hand, if $\gamma_j \in \CC\llbrace u\rrbrace$ has a non--trivial Laurent part  for some $0 \le j \le m-1$ then
$U \not\cong U_\circ$ (since we can not eliminate  the Laurent part of $\gamma$ by multiplying it with the image of a unit from  $\widehat{R}$).
\end{proof}

\begin{theorem}\label{T:descriptionPicard} Let $\bigl(\Pi, \underline{\mu})$ be any datum and $A = A\bigl(\Pi, \underline{\mu})$ be the corresponding algebra
of quasi--invariants. Consider the following homomorphism of abelian groups:
\begin{equation}\label{E:homomUpsilon}
\CC\llbracket z_1, z_2\rrbracket \xrightarrow{\Upsilon} \prod\limits_{\alpha \in \Pi} \bigl(\CC\llbracket \rho\rrbracket[\sigma]/(\sigma^{\mu_\alpha}),  \circ\bigr), \quad h \mapsto \bigl(\Upsilon_{(\alpha, \mu_\alpha)}(h)\bigr)_{\alpha \in \Pi}.
\end{equation}
Then we have:
\begin{equation}\label{E:descriptionPicard}
 \Pic(A) \cong \mathrm{Im}(\Upsilon) \cap K^\diamond\bigl(\Pi, \underline{\mu}\bigr),
 \end{equation}
 where $K^\diamond\bigl(\Pi, \underline{\mu}\bigr) := \prod\limits_{\alpha \in \Pi} \bigl(\CC[\rho][\sigma]/(\sigma^{\mu_\alpha}), \circ\bigr)$.
More explicitly, let
$$
\Gamma\bigl(\Pi, \underline{\mu}\bigr) := \bigl\{h \in \CC\llbracket z_1, z_2\rrbracket  \big|  \Upsilon(h) \in K^\diamond\bigl(\Pi, \underline{\mu}\bigr)\bigr\}.$$
 Then for any $h \in \Gamma\bigl(\Pi, \underline{\mu}\bigr)$, the corresponding  projective $A$--module of rank one is given by
\begin{equation}\label{E:projectiveQuasiInv}
P(h) := \bigl\{f \in  R \,\big|\, \exp(h) f \;\mbox{\rm is} \; (\Pi, \underline{\mu})-\mbox{\rm quasi--invariant}\bigr\} = B\bigl(\Upsilon(h)\bigr).
\end{equation}
Conversely,  for any $P \in \Pic(A)$, there exists $h \in \Gamma\bigl(\Pi, \underline{\mu}\bigr)$ such that $P \cong P(h)$. Moreover,
\begin{itemize}
\item $P(h_1) \cong P(h_2)$ if and only if $\Upsilon(h_1) = \Upsilon(h_2)$.
\item The multiplication map $P(h_1) \otimes_A P(h_2) \lar P(h_1 + h_2), \; f_1 \otimes f_2 \mapsto f_1 f_2$ is an isomorphism of $A$--modules.
\end{itemize}
\end{theorem}
\begin{proof}
Let $P \in \CM_1^{\mathsf{lf}}(A)$, then we have:
$
\FF(P) := U \cong (R, V, \theta),
$
where $V \cong \oplus_{\alpha \in \Pi} K_\alpha$ and $\theta = (\theta_\alpha)_{\alpha \in \Pi}$, where  $\theta_\alpha =
1 + \varepsilon  \gamma_\alpha \in L_\alpha$ for some  $\gamma \in K_\alpha$.

Note that the  module $P$ is projective if and only if $\widehat{P}_\idm \cong \widehat{A}_\idm$ for any
maximal ideal  $\idm \in \mathsf{Max}(A)$. Let $p \in X$ be the point corresponding to $\idm$ and  $q := \nu^{-1}(p) \in \AA^2$.
Assume that $q \notin \cup_{\alpha \in \Pi} V(l_\alpha)$. Then $\widehat{P}_\idm \cong \widehat{A}_\idm$  is automatically true.

\smallskip
\noindent
Now, let $q \in \cup_{\alpha \in \Pi} V(l_\alpha)$. According to Theorem \ref{T:BurbanDrozd}, $\widehat{P}_\idm \cong \widehat{A}_\idm$ if and only if
the triples $\LL_\idm(U)$ and $\FF_\idm(\widehat{A}_\idm)$ are isomorphic in the category $\Tri(\widehat{A}_\idm)$.

\smallskip
\noindent
\underline{Case 1}. Assume that $q \ne (0,0)$.
Then there  exists uniquely determined $\alpha \in \Pi$ and $\rho_0 \in \CC^\ast$ such that $q = \bigl(\rho_0 \cos(\alpha), \rho_0 \sin(\alpha)\bigr) \in  V(l_\alpha)$.  Denote $u = \rho\cos(\varphi-\alpha) -\rho_0$ and $v = \rho \sin(\varphi-\alpha)$. Obviously, we have:
$R = \CC[u, v]$, $\widehat{R}:= \widehat{R}_\idm \cong \CC\llbracket u,v\rrbracket$ and $\widehat{I} := \widehat{I}_\idm = (v^{2\mu_\alpha})$.
Moreover, the following diagram is commutative:
$$
\xymatrix{
I \ar@{^{(}->}[r] \ar@{_{(}->}[d]_-{\mathsf{can}} & R \ar[r]^-{T_{(\alpha, 2\mu_\alpha)}} \ar@{_{(}->}[d]^-{\mathsf{can}} & \CC(\rho)[\varepsilon]/(\varepsilon^{2\mu_\alpha}) \ar[d] \ar@/^40pt/[dd]^-{\lambda_{\rho_0}} \\
\widehat{I} \ar@{^{(}->}[r] \ar@{_{(}->}[d]_-\cong & \widehat{R} \ar[r] \ar@{_{(}->}[d]^-\cong & Q(\widehat{R}/\widehat{I}) \ar[d]^-\cong\\
(v^{2\mu_\alpha}) \ar@{^{(}->}[r] & \CC\llbracket u, v\rrbracket \ar[r] & \CC\llbrace u\rrbrace[v]/(v^{2\mu_\alpha})
}
$$
where $\lambda_{\rho_0}(g) \in \CC\llbrace u\rrbrace[\sigma]/(\sigma^{\mu_\alpha})$ is the Laurent expansion of $g \in \CC(\rho)[\sigma]/(\sigma^{\mu_\alpha})$ at the point $\rho_0 \in \AA^1\cong V(l_\alpha)$.
Note that  we are in the setting
of Lemma \ref{L:TriplesSpecialCase}: $\widehat{A}_\idm \cong \CC\llbracket u, v^2, v^{2\mu_\alpha+1}\rrbracket$.
The key point is that we have the following formula for the localized and completed  triple:
$$
\LL_\idm(U) \cong \bigl(\CC\llbracket u, v\rrbracket, \CC\llbrace u\rrbrace[v^2]/(v^{2\mu_\alpha}), 1 + \bar{v} \lambda_{\rho_0}(\gamma_\alpha)\bigr).
$$

\smallskip
\noindent
According to Lemma \ref{L:TriplesSpecialCase}, the module $P$ is locally free at the point  $p \in X$ if and only if $\gamma_\alpha \in \CC(\rho)[\sigma]/(\sigma^{\mu_\alpha})$ has no pole at $\rho_0 \in \AA^1$.

\smallskip
\noindent
\underline{Case 2}. For  the point $q = (0,0)$ we have:
$$
\LL_\idm(U) \cong \bigl(\CC\llbracket z_1, z_2\rrbracket, \oplus_{\alpha \in \Pi} \CC\llbrace \rho\rrbrace[\sigma]/(\sigma^{\mu_\alpha}),
(1 + \varepsilon \gamma_{\alpha})_{\alpha \in \Pi}\bigr),
$$
where the element  $\gamma_\alpha \in \CC(\rho)[\sigma]/(\sigma^{\mu_\alpha})$ is viewed as an element of
$\CC\llbrace \rho\rrbrace[\sigma]/(\sigma^{\mu_\alpha})$ for each $\alpha \in \Pi$.
Similarly to the previous case we conclude  that $P$ can be projective only if $\gamma_\alpha \in K_\alpha$ is regular at $0$ for each $\alpha \in \Pi$. Hence, $\vec\gamma = (\gamma_\alpha)_{\alpha \in \Pi} \in K^\diamond\bigl(\Pi, \underline{\mu}\bigr)$.

It follows from the definition
of the category $\Tri(\widehat{A}_\idm)$ that $\LL_\idm(U) \cong \FF_\idm(\widehat{A}_\idm)$ if and only if there exists a \emph{unit}
$f \in \widehat{R}$ such that
\begin{equation}\label{E:condit}
T_{(\alpha, 2\mu_\alpha)}(f) \cdot (1 + \varepsilon \gamma_\alpha) \in \CC\llbracket \rho\rrbracket [\varepsilon^2]/(\varepsilon^{2\mu_\alpha})
\end{equation}
for every $\alpha \in \Pi$.
Every unit in the algebra $\widehat{R}$ can be written as the exponential of some power series, hence
$f = \exp(h)$ for some $h \in \widehat{R}$. In the notation of formula (\ref{L:TaylorSeriesVariation}),
the  condition (\ref{E:condit}) can be rewritten as:
$
\widehat{T}_{(\alpha, 2\mu_\alpha)}(h) \cdot (1 + \varepsilon \gamma_\alpha) \in \CC\llbracket \rho\rrbracket [\varepsilon^2]/(\varepsilon^{2\mu_\alpha}),
$
i.e. $\Upsilon(h) = \vec\gamma$.

\smallskip
\noindent
Note that the constraints on a polynomial $f \in R$  from the formula (\ref{E:cosnstraints}) defining the module $B(\vec{\gamma})$ and the ones from (\ref{E:condit}) are in fact the same.  It implies that $P(h) = B(\vec{\gamma}) = B\bigl(\Upsilon(h)\bigr)$. Moreover,
Theorem \ref{T:listCMlf} implies that $P(h_1) \cong P(h_2)$ if and only if  $\Upsilon(h_1) = \Upsilon(h_2)$. Finally, the diagram
\begin{equation}
\begin{array}{c}
\xymatrix{
P(h_1) \otimes_A P(h_2) \ar[rr]^-{\mathsf{mult}} \ar[d]_= & & P(h_1+h_2) \ar[d]^=\\
B\bigl(\Upsilon(h_1)\bigr) \boxtimes_A B\bigl(\Upsilon(h_2)\bigr)
 \ar[rr]^-\cong \ar@{_{(}->}[d]
  && B\bigl(\Upsilon(h_1) \circ \Upsilon(h_2)\bigr) \ar@{^{(}->}[d] \\
  R \otimes_R R \ar[rr]^-{\mathsf{mult}}  & & R
}
\end{array}
\end{equation}
is commutative. It implies that the multiplication map
$P(h_1) \otimes_A P(h_2) \xrightarrow{\mathsf{mult}} P(h_1 + h_2)$ is indeed an isomorphism of $A$--modules. Theorem is proven.
\end{proof}

\smallskip
\noindent
The following special cases of Theorem \ref{T:descriptionPicard} are perhaps of independent interest.

\begin{example}\label{EX:PicardSingSurfaces}
Let $\Pi = \left\{0, \dfrac{\pi}{2}\right\}$. Note that  for any $f \in \CC\llbracket x, y\rrbracket$, we have the following formulae for the directional derivatives (\ref{E:directderivat}):
$$
f'_{0}(\rho) = \rho \frac{\partial f}{\partial y}(\rho, 0) \quad \mbox{and} \quad f'_{\frac{\pi}{2}}(\rho) = -\rho \frac{\partial f}{\partial x}(0, \rho).
$$
\begin{enumerate}

\item Let $\Pi \stackrel{\underline{\mu}}\lar \NN_0$ be given by the rule: $\underline{\mu}(0) = \underline{\mu}\left(\dfrac{\pi}{2}\right) = 1$. Then we have: $$A = \CC[x^2, x^3, y^2, y^3] \cong \CC[x^2, x^3] \otimes_\CC \CC[y^2, y^3].$$
    Then the description of $\Pic(A)$ from Theorem \ref{T:descriptionPicard} gets the following form:
    \begin{equation}\label{E:PicExamp}
    \begin{tabular}{c}
    \xymatrix{
    \CM_{1}^{\mathsf{lf}}(A) \ar[rr]^-\Theta & & \bigl(\CC(\rho), +\bigr) \oplus \bigl(\CC(\rho), +\bigr)\\
    \Pic(A) \ar[rr]^-\cong  \ar@{^{(}->}[u] & & \bigl\{(\rho f, \rho g) \in \rho \CC[\rho] \oplus \rho \CC[\rho] \; \big| \; f'(0) + g'(0) = 0\bigr\}. \ar@{_{(}->}[u]
    }
    \end{tabular}
    \end{equation}
    Indeed, for any $h \in \CC\llbracket z_1, z_2\rrbracket$ we have: 
    $$
 \bigl(\widehat{T}_{(0, 2)}(h),  \widehat{T}_{(\frac{\pi}{2}, 2)}(h)\bigr) = 
 \Bigl(1 + \varepsilon  \rho \frac{\partial h}{\partial y}(\rho, 0), 
 1 - \varepsilon \rho \frac{\partial h}{\partial x}(0, \rho)
 \Bigr) \in \CC\llbracket \rho \rrbracket[\varepsilon]/(\varepsilon^2)  \times \CC\llbracket \rho \rrbracket[\varepsilon]/(\varepsilon^2).
    $$
    It follows that $\Upsilon(h) = \bigl(- \rho \dfrac{\partial h}{\partial y}(\rho, 0),  \rho \dfrac{\partial h}{\partial x}(0, \rho)\bigr) \in 
    \CC\llbracket \rho\rrbracket \times  \CC\llbracket \rho\rrbracket$. Applying  formula
(\ref{E:descriptionPicard}), we get    
      the    description of the Picard group of $A$ given by (\ref{E:PicExamp}).
Note that  $\Upsilon(h) \in  \CC[\rho] \times  \CC[\rho]$   if and only if
$h(z_1, z_2) = a(z_1) + c(z_2) + z_1^2 z_2^2 e(z_1, z_2) + z_2 b(z_1) + z_1 d(z_2)$, where 
$a(z_1) \in \CC\llbracket z_1\rrbracket$, $c(z_2) \in \CC\llbracket z_2\rrbracket$,
$e(z_1, z_2) \in \CC\llbracket z_1, z_2\rrbracket$, $b(z_1) \in \CC[z_1]$ and $d(z_2) \in \CC[z_2]$. 
      \item Let $\Pi \stackrel{\underline{\mu}}\lar \NN_0$ be given by the rule: $\underline{\mu}(0) = 1$ and $\underline{\mu}\left(\dfrac{\pi}{2}\right) = 0$. Then we have: $A = \CC[x^2, x^3, y]$. Proceeding as above one can show that the  Picard group $\Pic(A)$ has the following description:
    $$
    \xymatrix{
    \CM_{1}^{\mathsf{lf}}(A) \ar[rr]^-\Theta & & \bigl(\CC(\rho), +\bigr)\\
    \Pic(A) \ar[rr]^-\cong  \ar@{^{(}->}[u] & & \bigl(\rho \CC[\rho], +\bigr), \ar@{_{(}->}[u]
    }
    $$
where $\Theta$ is the isomorphism from Theorem \ref{T:listCMlf}. It is interesting to note that $\Pic\bigl(D[t]\bigr) \cong \Pic(D)$ for any Noetherian normal domain $D$; see for instance \cite{Swan}.
\end{enumerate}
\end{example}
\section{Spectral module of a rational Calogero--Moser system of dihedral type}\label{S:SpectralModule}
In this section, we shall discuss a link between results on Cohen--Macaulay modules over an algebra planar quasi--invariants with the theory  Calogero--Moser systems.

\begin{definition}\label{D:BakerAkhieserWL} For any $\alpha \in \CC$ we put: $e(\alpha) = \exp(2 i\alpha)$.
A weighted line arrangement $\bigl(\Pi, \underline{\mu}\bigr)$ is called  \emph{Baker--Akhieser} if for any $\alpha \in \Pi$ and $1 \le k \le \mu_\alpha$ we have:
\begin{equation}\label{E:BAarrangement}
\sum\limits_{\substack{\beta \in \Pi \\ \beta \ne \alpha}} \frac{\mu_\beta \bigl(e(\beta)+e(\alpha)\bigr)^{2k-1}}{\bigl(e(\beta)-e(\alpha)\bigr)^{2k-1}} = 0 \quad \mbox{\rm and}\quad
\sum\limits_{\substack{\beta \in \Pi \\ \beta \ne \alpha}} \frac{\mu_\beta (\mu_\beta +1)e(\beta) \bigl(e(\beta)+e(\alpha)\bigr)^{2k-1}}{\bigl(e(\beta)-e(\alpha)\bigr)^{2k+1}} = 0;
\end{equation}
see \cite{ChalykhStyrkasVeselov, ChalykhFeiginVeselov2} and \cite[Lemma 2.1]{FeiginJohnston}.
\end{definition}

\begin{example}\label{E:CoxeterLA}
A so--called  \emph{Coxeter} weighted line arrangement $\bigl(\Lambda_n, \underline{\mu}\bigr)$ defined below is Baker--Akhieser; see  \cite{ChalykhStyrkasVeselov, ChalykhFeiginVeselov2,FeiginJohnston}.
\begin{itemize}
\item $ \Lambda_n:= \bigl\{0, \frac{1}{n} \pi, \dots, \frac{n-1}{n} \pi \bigr\} \subset \RR$ for some $n \in \NN$.
\item $\mu_\alpha = m$ for all $\alpha \in \Lambda_n$ and some $m \in \NN$ (constant multiplicity function), or
\item $n = 2 \bar{n}$ and $\underline{\mu}\bigl(\frac{2k}{n}\pi\bigr) = m_1$ and $\underline{\mu}\bigl(\frac{2k+1}{n}\pi\bigr) = m_2$ for all
$0 \le k \le \bar{n}-1$ and some $m_1, m_2 \in \NN$.
\end{itemize}
\end{example}

\noindent
We put $\mu = \mu\bigl(\Pi, \underline{\mu}\bigr) := \sum_{\alpha \in \Pi} \mu_\alpha$ and
$\delta(z_1, z_2) := \prod_{\alpha\in \Lambda_n} l^{\mu_\alpha}_\alpha(z_1, z_2)$. Note that $\delta(z_1, z_2)$ is a homogeneous polynomial of degree $\mu$.

\subsection{Some results on two--dimensional  rational Calogero--Moser system} Let $\bigl(\Pi, \underline{\mu}\bigr)$ be a   \emph{Baker--Akhieser}
weighted line arrangement.
The \emph{rational Calogero--Moser operator} $H = H(\Pi, \underline{\mu})$ is defined by the formula
\begin{equation}\label{E:CMHamiltonianNonPerturb}
H  :=
\left(\frac{\partial^2}{\partial x_1^2} + \frac{\partial^2}{\partial x_2^2}\right)
- \sum\limits_{\alpha \in \Pi} \dfrac{\mu_\alpha(\mu_\alpha+1)}{l_\alpha^2(x_1, x_2)}.
\end{equation}
According to a result of Chalykh and Veselov \cite{ChalykhVeselov90} (elaborated in their later joint work  with Styrkas \cite{ChalykhStyrkasVeselov})  extending  earlier results   of Heckman and Opdam \cite{HeckmanOpdam} and Heckman \cite{Heckman}, the operator $H$ can be included into a large family of pairwise commuting partial differential operators. In order to make this statement more precise, we recall the following formula of Berest \cite[Theorem 2.8]{Berest} (see also \cite[Theorem 3.1]{ChalykhFeiginVeselov2}) for   the so--called  \emph{multivariate Baker--Akhieser function}, corresponding to the datum $\bigl(\Pi, \underline{\mu}\bigr)$:
\begin{equation}\label{E:Berest}
\Phi(x_1, x_2; z_1, z_2) := \frac{1}{2^\mu \cdot \mu!} \Bigl(H_{(x_1, x_2)} -z_1^2 - z_2^2\Bigr)^\mu \diamond \bigl(\delta(x_1, x_2) \cdot \exp(x_1 z_1 + x_2 z_2)\bigr).
\end{equation}
\begin{theorem}\label{T:BAfunction} The following results are true.
\begin{enumerate}
\item The Baker--Akhieser function $\Phi(x_1, x_2; z_1, z_2)$ is an eigenfunction of the Calogero--Moser operator (\ref{E:CMHamiltonianNonPerturb}) in the following sense:
\begin{equation}\label{E:BAEigenfunction}
H_{(x_1, x_2)} \diamond \Phi(x_1, x_2; z_1, z_2) = (z_1^2 + z_2^2) \cdot \Phi(x_1, x_2; z_1, z_2).
\end{equation}
\item Moreover, there exists an injective algebra homomorphism
\begin{equation}\label{E:QuantumIntegrabCM}
A\bigl(\Pi, \underline{\mu}\bigr) \stackrel{\Xi}\lar  \CC(x_1, x_2)[\partial_1, \partial_2].
\end{equation}
such that the highest symbol of $\Xi(f)$ is equal to the highest symbol of $f(\partial_1, \partial_2)$ and
\begin{equation}\label{E:QuantumIntegrabCM2}
\bigl(\Xi(f)\bigr)_{(x_1, x_2)}  \diamond \Phi(x_1, x_2; z_1, z_2) = f(z_1, z_2) \cdot \Phi(x_1, x_2; z_1, z_2).
\end{equation}
for any $f \in A\bigl(\Pi, \underline{\mu}\bigr)$.  In particular, $H = \Xi(\omega)$ for $\omega = z_1^2 + z_2^2$.
\item The Baker--Akhieser function $\Phi$ has the following expansion:
\begin{equation}\label{E:expBAf}
\Phi(x_1, x_2; z_1, z_2) = \bigl(\delta(z_1, z_2) + \sum\limits_{i_1 + i_2 < \mu} c_{i_1, i_2}(x_1, x_2) z_1^{i_1} z_2^{i_2}\bigr) \cdot \exp(x_1 z_1 + x_2 z_2),
\end{equation}
where $c_{i_1, i_2}(x_1, x_2) \in \CC(x_1, x_2)$ for all $(i_1, i_2)$. Moreover,
$$
c_{0, 0}(x_1, x_2) = c \prod\limits_{\alpha \in \Pi} \frac{1}{l_\alpha(x_1 - \xi_1, x_2 - \xi_2)^{\mu_\alpha}},
$$
where $c \in \CC$ is a certain explicit constant, whose value can be found in  \cite{FeiginNew}.
\item Let $z_1 = \rho \cos(\varphi)$ and $z_2 = \rho \sin\varphi$. Then we have:
\begin{equation}\label{E:BAquasiinv}
\Phi(x_1, x_2; \rho)^{(2l-1)}_\alpha := \left.\frac{\partial^{2l-1} \Phi}{\partial \varphi^{2l-1}} \right|_{\varphi = \alpha} = 0 \quad \mbox{\rm for all}\;  \alpha \in \Pi
\;  \mbox{\rm and} \;
1 \le l \le \mu_\alpha.
\end{equation}
\end{enumerate}
\end{theorem}

\noindent
\emph{Comment to the proof}. From the historical perspective, the development of results collected in Theorem \ref{T:BAfunction} was slightly different.
The notion of a  multivariate Baker--Akhieser function $\Phi(x_1, x_2; z_1, z_2)$  corresponding to a Baker--Akhieser datum $\bigl(\Pi, \underline{\mu}\bigr)$  was axiomatized (in arbitrary dimension) by Chalykh, Styrkas and Veselov in \cite{ChalykhVeselov90,ChalykhStyrkasVeselov}.
The properties (\ref{E:expBAf}) and (\ref{E:BAquasiinv}) were stated as defining axioms, whereas the eigenfunction properties (\ref{E:QuantumIntegrabCM}) and (\ref{E:QuantumIntegrabCM2})  were shown to be formal consequences of the proposed axiomatic.

In \cite[Theorem 2.8]{Berest}, Berest discovered an explicit formula (\ref{E:Berest}) for a multivariate Baker--Akhieser function; see also \cite[Theorem 3.1]{ChalykhFeiginVeselov2}. There is a closed expression  for the homomorphism
$\Xi$; see \cite[Section 2.3]{Berest}, \cite[Theorem 1.3]{ChalykhFeiginVeselov2} or \cite[Corollary 3.3]{EtingofStrickland}.

In our exposition, we start with Berest's formula (\ref{E:Berest}) for the Baker--Akhieser function $\Phi(x_1, x_2; z_1, z_2)$. The formula  (\ref{E:expBAf}) can be deduced from (\ref{E:Berest})  by induction on $\mu$. We refer to a paper of Chalykh, Feigin and Veselov \cite{ChalykhFeiginVeselov2} for further details. \qed

\medskip
Our next goal is to study the rational Calogero--Moser operator  (\ref{E:CMHamiltonianNonPerturb}) using methods of the higher--dimensional  Krichever correspondence
developed in \cite{KOZ, Zheglov, KurkeZheglov}. To do this, we have to introduce the following  minor modification of the operator
(\ref{E:CMHamiltonianNonPerturb}). Let $\vec\xi = (\xi_1, \xi_2) \in \CC^2$ be such that
\begin{itemize}
\item $l_\alpha(\vec{\xi}) = -\sin(\alpha) \xi_1 + \cos(\alpha) \xi_2 \ne 0$ for all $\alpha \in \Pi$.
\item $\xi_1^2 + \xi_2^2 \ne 0$ (for example, one can simply take  $\vec\xi \in \RR^2 \setminus \{\vec{0}\}$).
\end{itemize}
The second condition on $\vec\xi$ implies that  one can find  $(\rho_0, \varphi_0) \in \CC^\ast \times \CC$  such that $$\vec{\xi} =  \bigl(\rho_0 \cos(\varphi_0), \rho_0 \cos(\varphi_0)\bigr).$$
For any such vector $\vec\xi$, we have an automorphism $$\CC(x_1, x_2) \stackrel{t(\vec{\xi})}\lar \CC(x_1, x_2), \quad \mbox{\rm where} \;
\bigl(t(\vec\xi)(f)\bigr)(\vec{x}) = f(\vec{x} - \vec{\xi}),$$
which can obviously be extended  to an automorphism $t(\vec{\xi})$ of the algebra of partial differential operators $\CC(x_1, x_2)[\partial_1, \partial_2]$.

\smallskip
\noindent
Summing up, we have an injective algebra homomorphism $\Xi(\vec{\xi})$ given as the composition
\begin{equation}\label{E:mapXiTwist}
A\bigl(\Pi, \underline{\mu}\bigr) \stackrel{\Xi}\lar  \CC(x_1, x_2)[\partial_1, \partial_2] \stackrel{t(\vec\xi)}\lar  \CC(x_1, x_2)[\partial_1, \partial_2].
\end{equation}
Then the perturbed Calogero--Moser operator $H = H\bigl((\Pi, \underline{\mu}), \vec\xi\bigr) := \bigl(\Xi(\vec\xi)\bigr)(\omega)$ is given by the formula
\begin{equation}\label{E:CMHamiltonianN}
H  :=
\left(\frac{\partial^2}{\partial x_1^2} + \frac{\partial^2}{\partial x_2^2}\right)
- \sum\limits_{\alpha \in \Pi} \dfrac{\mu_\alpha(\mu_\alpha+1)}{l_\alpha^2(\vec{x}-\vec{\xi})},
\end{equation}
whereas the conventional Calogero--Moser operator (\ref{E:CMHamiltonianNonPerturb}) is $H\bigl((\Pi, \underline{\mu}), \vec{0}\bigr)$.
Note that the potential of (\ref{E:CMHamiltonianN}) is regular at the point $(0, 0)$. Moreover,
\begin{equation}\label{E:algebraB}
\gB := \mathsf{Im}\bigl(\Xi(\vec{\xi})\bigr) \subseteq \gD:= \CC\llbracket x_1, x_2\rrbracket[\partial_1, \partial_2],
\end{equation}
 hence we get the  embedding (\ref{E:IntegrabilityCM}).

\begin{definition}
The $\gB$--module $F:= \gD/(x_1, x_2)\gD \cong \CC[\partial_1, \partial_2]$ is called \emph{spectral module} of the algebra $\gB$.
\end{definition}

\noindent
Note that $F$ is  actually a \emph{right} $\gD$--module. However, since  the algebra $\gB$ is commutative, we shall view $F$ as a \emph{left} $\gB$--module,  having the natural right action $\circ$ in mind.

\begin{theorem}\label{T:spectralmodule}
The following results are true.
\begin{enumerate}
\item $F$ is a finitely generated Cohen--Macaulay $\gB$--module of rank one.
\item  For any character $\gB \stackrel{\chi}\lar \CC$ (i.e.~an algebra homomorphism), consider the vector space
\begin{equation}\label{E:solspace}
\mathsf{Sol}\bigl(\gB, \chi\bigr):= \Bigl\{f\in \CC\llbracket x_1, x_2\rrbracket \big| P\diamond f = \chi(P) f \; \mbox{\rm for all}\; P \in \gB\Bigr\}.
\end{equation}
Then there exists  a canonical isomorphism of vector spaces
\begin{equation}\label{E:solspaceisom}
F\big|_{\chi} := F \otimes_{\gB} \bigl(\gB/\mathrm{Ker}(\chi)\bigr) \cong \mathsf{Sol}\bigl(\gB, \chi\bigr)^\ast.
\end{equation}
assigning to a class $\overline{\partial_1^{p_1} \partial_2^{p_2}} \in F\big|_{\chi}$ the linear functional $f  \mapsto \left.\dfrac{1}{p_1! p_2!} \dfrac{\partial^{p_1+p_2}f}{\partial x_1^{p_1} \partial x_2^{p_2}}\right|_{(0, 0)}$ on the vector space $\mathsf{Sol}\bigl(\gB, \chi\bigr)$. In particular, $\dim_{\CC}\Bigl(\mathsf{Sol}\bigl(\gB, \chi\bigr)\Bigr)<\infty$  for any  $\chi$.
\end{enumerate}
\end{theorem}

\begin{proof}
We get the first statement,  combining   \cite[Theorem 3.1]{KurkeZheglov} with \cite[Corollary 3.1]{Chalykh}; see also \cite[Theorem 2.1]{KOZ}.

\smallskip
\noindent
In the one--dimensional case, the isomorphism (\ref{E:solspaceisom}) is due to Mumford \cite[Section 2]{Mumford}. For partial differential operators, we  follow the exposition in \cite[Theorem 1.14]{BurbanZheglov}. The key point is the following  isomorphism of left $\gD$--modules:
\begin{equation}
\Hom_{\CC}\bigl(F, \CC\bigr) \stackrel{\mathfrak{t}}\lar \CC\llbracket x_1, x_2 \rrbracket, \quad l \mapsto \sum\limits_{p_1, p_2 = 0}^\infty \frac{1}{p_1! p_2!} l\bigl(\partial_1^{p_1} \partial_2^{p_2}\bigr) x_1^{p_1} x_2^{p_2},
\end{equation}
where we take the right action $\circ$ on $F$ and the usual right action $\diamond$ on $\CC\llbracket x_1, x_2\rrbracket$ of the algebra $\gD$.
Let $\gB \stackrel{\chi}\lar \CC$ be a character, then $\CC = \CC_\chi:= \gB/\ker(\chi)$ is a left $\gB$--module. Next, we have a $\gB$--linear map
\begin{equation}
\mathfrak{b}: \Hom_{\gB}(F, \CC_\chi) \stackrel{\mathfrak{i}}\lar \Hom_{\CC}(F, \CC) \stackrel{\mathfrak{t}}\lar \CC\llbracket x_1, x_2 \rrbracket,
\end{equation}
where $\mathfrak{i}$ is the forgetful map. The image of $\mathfrak{i}$ consists of those $\CC$--linear functionals, which are also $\gB$--linear, i.e.
\begin{equation*}
\mathsf{Im}(\mathfrak{i}) = \bigl\{l \in \Hom_{\CC}(F, \CC) \; \big| \;  l(P \diamond  \,-\,) = \chi(P)\cdot l(\,-\,)\; \mbox{for all} \; P \in \gB\bigr\}.
\end{equation*}
This implies that $\mathsf{Im}(\mathfrak{b}) = \mathsf{Sol}(\gB, \chi)$. By adjunction, we have  a canonical  isomorphism of $\gB$--modules:
$
\Hom_{\gB}(F, \CC_\chi) \cong \Hom_{\CC}\bigl(F \otimes_{\gB}  (\gB/\ker(\chi)), \CC\bigr).
$
Taking duals,  we get an isomorphism of vector spaces
$
 \mathsf{Sol}(\gB, \chi)^\ast \lar \bigl(F \otimes_{\gB}  (\gB/\ker(\chi))\bigr)^{\ast\ast} \cong F\big|_{\chi}.
$
It remains to observe that $F\big|_{\chi} \stackrel{(\mathfrak{b}^\ast)^{-1}}\lar  \mathsf{Sol}(\gB, \chi)^\ast$ is  the map from the statement of the theorem.
\end{proof}

\begin{remark} By Hilbert's Nullstellensatz, the characters $\gB \lar \CC$ stand in bijection with the points of the spectral surface $X$ of the Calogero--Moser system $\gB$ (i.e.~an affine surface, whose coordinate ring is isomorphic to $A\bigl(\Pi, \underline{\mu}\bigr) \cong \gB$). The finitely generated $\gB$--module $F$ determines a coherent sheaf of $X$, so the $\CC$--vector space
$F\big|_\chi$ is the fiber of $F$ over the point of $X$ corresponding to the character $\gB \stackrel{\chi}\lar \CC$.
\end{remark}

\subsection{Spectral module of a two--dimensional rational Calogero--Moser system}
In their recent paper  \cite[Section 8]{FeiginJohnston}, Feigin and Johnston raised a question about an explicit description of the spectral module $F$ of the algebra $\gB$ given by (\ref{E:algebraB}). In this subsection, we give  a solution of  this problem. To do this, we need a more concrete description  of the algebra homomorphisms $\Xi$ and $\Xi(\vec{\xi})$; see (\ref{E:QuantumIntegrabCM}) and (\ref{E:mapXiTwist}).

\begin{lemma}\label{L:BAshifted} Consider the following variation of the  function (\ref{E:Berest}):
\begin{equation}\label{E:BerestShifted}
\Psi(\vec{x}; \vec{z}; \vec{\xi}) := \frac{1}{2^\mu \cdot \mu!} \Bigl(H_{(x_1, x_2)} -z_1^2 - z_2^2\Bigr)^\mu \diamond \bigl(\delta(x_1 -\xi_1, x_2-\xi_2) \cdot \exp(x_1 z_1 + x_2 z_2)\bigr).
\end{equation}
Then the following results are true.
\begin{enumerate}
\item For any quasi--invariant polynomial $f \in A$ we have:
\begin{equation}\label{E:BAEigenfunctionModified2}
\bigl(\Xi(\vec{\xi})(f)\bigr)_{(x_1, x_2)} \diamond \Psi(\vec{x}; \vec{z}; \vec{\xi})  = f(z_1, z_2) \cdot \Psi(\vec{x}; \vec{z}; \vec{\xi}).
\end{equation}
\item For any $p_1, p_2 \in \NN_0$ we have:
\begin{equation}\label{E:elementsOfBasisW}
w_{(p_1, p_2)}(z_1, z_2)  := \left.\frac{\partial^{p_1+p_2} \Psi}{\partial x_1^{p_1} \partial x_2^{p_2}} \right|_{(x_1, x_2)= (0, 0)} \in \, \CC[z_1, z_2].
\end{equation}
Here, we view $w_{(p_1, p_2)}$ as a function of $\vec{z}$ depending on the parameter $\vec\xi \in \CC^2$.
Moreover, we have the following expansion:
\begin{equation}\label{E:basisvectorsW}
w_{(p_1, p_2)}(z_1, z_2) = z_1^{p_1} z_2^{p_2} \cdot \delta(z_1, z_2) + \; \mbox{\rm lower order terms}.\;
\end{equation}
\item For any $p_1, p_2 \in \NN_0$, the function $\exp\bigl(-\rho\rho_0 \cos(\varphi - \varphi_0)\bigr) \cdot w_{(p_1, p_2)}(\rho, \varphi)$ is quasi--invariant with respect to the datum $(\Pi, \underline{\mu})$, i.e.
\begin{equation}\label{E:BasisSpaceW}
\Bigl(\exp\bigl(-\rho\rho_0 \cos(\varphi - \varphi_0)\bigr) \cdot w_{(p_1, p_2)}(\rho, \varphi)\Bigr)^{(2l-1)}_\alpha = 0
\; \mbox{\rm for all}\;  \alpha \in \Pi
\;  \mbox{\rm and} \;
1 \le l \le \mu_\alpha,
\end{equation}
where as usual, $(z_1, z_2) = \bigl(\rho \cos(\varphi), \rho \sin(\varphi)\bigr)$ and $(\xi_1, \xi_2) = \bigl(\rho_0 \cos(\varphi_0), \rho_0 \sin(\varphi_0)\bigr)$.
\end{enumerate}
\end{lemma}
\begin{proof} Observe that we have the following equality:
$$
\Phi(\vec{x} - \vec{\xi}, \vec{z}) = \Psi(\vec{x}; \vec{z}; \vec{\xi}) \cdot \exp(-\xi_1 z_1 -  \xi_2 z_2) = \Psi(\vec{x}; \vec{z}; \vec{\xi}) \cdot
\exp\bigl(-\rho\rho_0 \cos(\varphi - \varphi_0)\bigr).
$$
Hence, all statements of Lemma \ref{L:BAshifted} are straightforward  consequences of the corresponding results from  Theorem \ref{T:BAfunction}.
\end{proof}

Let $\gE := \CC\llbracket x_1, x_2\rrbracket\llbrace \partial_1^{-1}\rrbrace \llbrace \partial_2^{-1}\rrbrace$ be the algebra of partial \emph{pseudo--differential operators}; see for instance \cite{Parshin} for a  precise definition and main ring--theoretic properties. This algebra admits the following convenient characterization.

\begin{lemma}\label{E:DefPseudoDiffOp}
Let $\mathfrak{M} := \CC\llbracket x_1, x_2\rrbracket \llbrace z_1^{-1}\rrbrace \llbrace z_2^{-1}\rrbrace \cdot \exp(x_1 z_1 + x_2 z_2)$
be the so--called \emph{Baker--Akhieser module}. Then we have an injective algebra homomorphism
\begin{equation}\label{E:pseudo1}
\gE \stackrel{\mathfrak{e}}\lar \End_{\CC\llbrace z_1^{-1}\rrbrace \llbrace z_2^{-1}\rrbrace}(\mathfrak{M})
\end{equation}
mapping $\partial_j^\pm \in \gE$ to $z_j^\pm \in \End_{\CC\llbrace z_1^{-1}\rrbrace \llbrace z_2^{-1}\rrbrace}(\mathfrak{M})$ for $j = 1, 2$.
Moreover, for any element $ Q \in \mathfrak{M}$, there exists a uniquely determined element $S \in \gE$ such that $$
Q = S \diamond \exp(x_1 z_1 + x_2 z_2) :=  \bigl(\mathfrak{e}(S)\bigr)\bigl(\exp(x_1 z_1 + x_2 z_2)\bigr).$$
\end{lemma}

\begin{remark}\label{R:SatoOperator} The recipe to construct the operator $S \in \gE$ corresponding to an element $Q \in \mathfrak{M}$ is as follows. Let
$Q(x_1, x_2; z_1, z_2) = T(x_1, x_2; z_1, z_2) \cdot \exp(x_1 z_1 + x_2 z_2)$, where
\begin{equation}\label{E:operatorT}
T(x_1, x_2; z_1, z_2) = \sum\limits_{p_1, p_2} a_{p_1, p_2}(x_1, x_2) z^{p_1} z_2^{p_2} \in \CC\llbracket x_1, x_2\rrbracket \llbrace z_1^{-1}\rrbrace \llbrace z_2^{-1}\rrbrace.
\end{equation}
Then we have:
\begin{equation}\label{E:pseudo2}
 S = \sum\limits_{p_1, p_2} a_{p_1, p_2}(x_1, x_2) \partial_1^{p_1} \partial_2^{p_2} \in \CC\llbracket x_1, x_2\rrbracket\llbrace \partial_1^{-1}\rrbrace \llbrace \partial_2^{-1}\rrbrace.
\end{equation}
Here, both sums (\ref{E:pseudo1}) and (\ref{E:pseudo2}) are taken in the appropriate sense.
\end{remark}

\begin{definition} Let $\Psi(x_1, x_2; z_1, z_2; \vec{\xi}) \in \mathfrak{M}$ be the Baker--Akhieser function of $\gB$ given by (\ref{E:BerestShifted}).
 Then the corresponding pseudo--differential operator  $S \in \gE$, defined by the recipe (\ref{E:pseudo2}), is called \emph{Sato operator} of the algebra $\gB$.
\end{definition}

\begin{lemma}\label{L:Sato}  For any  quasi--invariant polynomial $f \in A$ we have:
\begin{equation}\label{E:conjSato}
\bigl(\Xi(\vec{\xi})\bigr)(f) = S \cdot f(\partial_1, \partial_2) \cdot S^{-1},
\end{equation}
where both sides of (\ref{E:conjSato}) are viewed as elements of the algebra $\gE$.
\end{lemma}
\begin{proof}
Let $\Theta = \Theta(x_1, x_2; z_1, z_2; \vec{\xi}) := \Psi(x_1, x_2; z_1, z_2; \vec{\xi}) \cdot \exp(-x_1 z_1 - x_2 z_2)$. Then we have an expansion
$
\Theta = \delta(z_1, z_2) + \sum\limits_{i_1 + i_2 < \mu} b_{i_1, i_2}(x_1, x_2) z_1^{i_1} z_2^{i_2}
$
for some coefficients $b_{i_1, i_2} \in \CC\llbracket x_1, x_2\rrbracket$. In particular, the Sato operator of the algebra  $\gB$ belongs to $\gD$:
\begin{equation}
S = \delta(\partial_1, \partial_2) + \sum\limits_{i_1 + i_2 < \mu} b_{i_1, i_2}(x_1, x_2) \partial_1^{i_1} \partial_2^{i_2}.
\end{equation}
Since the  highest symbol $\delta(\partial_1, \partial_2)$ of $S$ is a partial differential operator with constant coefficients, $S$ is a unit in the algebra $\gE$; see for instance \cite[Proposition 1]{Parshin}.

\smallskip
\noindent
By definition, we have: $\Psi(x_1, x_2; z_1, z_2; \vec{\xi}) = S \diamond \exp(x_1 z_1 + x_2 z_2)$. Hence, the equality (\ref{E:BAEigenfunctionModified2}) can be rewritten in the form:
$$
\Bigl(\bigl(\bigl(\Xi(\vec{\xi})\bigr)(f)\bigr) \cdot S\Bigr) \diamond \exp(x_1 z_1 + x_2 z_2) = f(z_1, z_2) \cdot \bigl(S \diamond \exp(x_1 z_1 + x_2 z_2)\bigr).
$$
Since $\mathfrak{e}(S)$ is a $\CC[z_1, z_2]$--linear endomorphism of the Baker--Akhieser module $\mathfrak{M}$, we have:
$$
f(z_1, z_2) \cdot \bigl(S \diamond \exp(x_1 z_1 + x_2 z_2)\bigr) = S \diamond \bigl(f(z_1, z_2) \cdot \exp(x_1 z_1 + x_2 z_2)\bigr) =
$$
$$
S \diamond \bigl(f(\partial_1, \partial_2) \diamond \exp(x_1 z_1 + x_2 z_2)\bigr) = \bigl(S \cdot f(\partial_1, \partial_2)\bigr) \diamond \exp(x_1 z_1 + x_2 z_2).
$$
Summing up, $\Bigl(\bigl(\bigl(\Xi(\vec{\xi})\bigr)(f)\bigr) \cdot S - S \cdot f(\partial_1, \partial_2)\Bigr) \diamond \exp(x_1 z_1 + x_2 z_2) = 0$, implying the result.
\end{proof}

\begin{remark}
Using the identification $z_j = \partial_j$ for $j = 1,2$, we can view the algebra of quasi--invariants $A \subset R = \CC[z_1, z_2]$ as a subalgebra
of the algebra of partial differential operators with constant coefficients  $\CC[\partial_1, \partial_2]$. If $S$ the Sato operator of $\gB$ given by
(\ref{E:pseudo2}), then we have:
$
\gB = S \cdot A \cdot S^{-1}.
$
\end{remark}

\begin{proposition}\label{P:spectralmodule}
Consider the vector space
\begin{equation}\label{E:SpaceW}
W := \bigl\langle w_{p_1, p_2} \big| (p_1, p_2) \in \NN_0 \times \NN_0\bigr\rangle \subset R = \CC[z_1, z_2],
\end{equation}
where $w_{p_1, p_2}$ are the elements given by (\ref{E:elementsOfBasisW}). Then  $W$ is an $A$--module and  the map
\begin{equation}\label{E:isospecmodule}
F := \CC[\partial_1, \partial_2] \xrightarrow{-\circ S} W,  \quad f(\partial_1, \partial_2) \mapsto f(\partial_1, \partial_2) \circ S
\end{equation}
is a $(\gB-A)$--equivariant isomorphism, i.e.~the following diagram
$$
\xymatrix{
F \ar[rr]^-{-\circ S} \ar[d]_{-\,\circ \Xi(\vec\xi)(f)} & & W \ar[d]^-{-\,\cdot f} \\
F \ar[rr]^-{-\circ S} & & W
}
$$
is commutative for any $f \in A$.

\smallskip
\noindent
For any $j \in \NN_0$ put $W_j := \bigl\{w\in W \, \big|\, \deg(w) \le j\bigr\}$. Then  we have the following formula for the Hilbert function $H_W$ of the filtered $A$--module $W$:
\begin{equation}\label{E:HilbertFunction}
H_W(\mu+k):= \dim_{\CC}\bigl(W_{\mu+k}\bigr) = \frac{(k+1)(k+2)}{2} \quad \mbox{\rm for} \; k \in \NN_0.
\end{equation}
\end{proposition}
\begin{proof} We have the following commutative diagram:
$$
\xymatrix{
\gD/(x_1, x_2) \gD \ar[r]^-{\cong} \ar@{_{(}->}[d] & \CC[\partial_1, \partial_2] \ar@{_{(}->}[d] \ar[r]^-{=}& \CC[z_1, z_2] \ar@{^{(}->}[d] \\
\gE/(x_1, x_2) \gE \ar[r]^-{\cong}  & \CC\llbrace\partial_1^{-1}\rrbrace \llbrace\partial_2^{-1}\rrbrace \ar[r]^-{=} & \CC\llbrace z_1^{-1}\rrbrace \llbrace z_2^{-1}\rrbrace.
}
$$
Let $W_0 = \CC[z_1, z_2]$. Then $W := W_0 \circ S \subset \CC\llbrace z_1^{-1}\rrbrace \llbrace z_2^{-1}\rrbrace$. Indeed,
it follows from the definition of the action $\circ$ of the algebra $\gE$ on the vector space $\CC\llbrace\partial_1^{-1}\rrbrace \llbrace\partial_2^{-1}\rrbrace$ that
$
\bigl(\partial_1^{p_1} \partial_2^{p_2}\bigr) \circ S = \left.\dfrac{\partial^{p_1+p_2} \Psi}{\partial x_1^{p_1} \partial x_2^{p_2}}\right|_{(x_1, x_2) = (0,0)}
$
for any $(p_1, p_2) \in \NN_0 \times \NN_0$, implying the claim.
Since the operator $S$ is a unit in the algebra $\gE$, the $\CC$--linear map (\ref{E:isospecmodule}) is an
isomorphism. Next, by Lemma \ref{L:Sato}, for any
$w \in F = W_0$ and $f \in A$ we have:
$$
w \circ \bigl(\bigl(\Xi(\vec{\xi})(f)\bigr) \cdot S\bigr) = (w \circ S) \cdot f
$$
Therefore, $W \cdot f \subseteq W$ and the map $F  \xrightarrow{-\circ S} W$ is $(\gB-A)$--equivariant, as claimed.

Next,
the highest  degree homogeneous term of $w_{p_1, p_2}(z_1, z_2)$ is $z_1^{p_1} z_2^{p_2} \cdot \delta(z_1, z_2)$; see formula  (\ref{E:basisvectorsW}).  This shows that
$$
W_{\mu +k}/W_{\mu+k-1} \cong \CC[z_1, z_2]_{k} := \bigl\{w \in \CC[z_1, z_2] \, \big|\,  w \; \mbox{\rm is homogeneous of degree} \; k\bigr\}.
$$
Hence, $\dim_{\CC}\bigl(W_{\mu +k}/W_{\mu+k-1}\bigr) = k+1$, implying that
$$
\dim_{\CC}\bigl(W_{\mu +k}\bigr) = 1 + 2 +\dots +(k+1) = \frac{(k+1)(k+2)}{2}.
$$
Proposition is proven.
\end{proof}
\begin{theorem}\label{T:DescriptSpectModule} Let $\vec\xi = \bigl(\rho_0 \cos(\varphi_0), \rho_0 \sin(\varphi_0)\bigr) \in \CC^2 \setminus\{\vec{0}\}$ be such that $\sin\bigl(2(\alpha -\varphi_0)\bigr) \ne 0$ for all $\alpha \in \Pi$.
Let $W$ be the $A$--module defined by (\ref{E:SpaceW}) (i.e.~$W$ is the spectral module of the algebra $\gB$). Then we have:
\begin{equation}
W \cong P\bigl(-\xi_1 z_1 - \xi_2 z_2\bigr) := \bigl\{f \in  R \,\big|\, \exp(-\xi_1 z_1 - \xi_2 z_2) f \;\mbox{\rm is} \; (\Pi, \underline{\mu})-\mbox{\rm quasi--invariant}\bigr\}.
\end{equation}
In particular, the spectral module $W$ is projective; see Theorem \ref{T:descriptionPicard}.
\end{theorem}
\begin{proof}
Let $u = \xi_1 z_1 + \xi_2 z_2  = \rho\rho_0 \cos(\varphi - \varphi_0)$.
According to (\ref{E:BasisSpaceW}), we have an inclusion $W \subseteq P(-u)$. Our goal is to show that in fact $W =  P(-u)$. By
(\ref{E:HilbertFunction}), $\dim_{\CC}\bigl(W_{\mu+k}\bigr) = \dfrac{(k+1)(k+2)}{2}$ for any $k \in \ZZ$. Hence, it is sufficient to prove that
\begin{equation}\label{E:HilbertFunctionEstimate}
\dim_{\CC}\bigl(P(-u)_{\mu+k}\bigr) \le  \frac{(k+1)(k+2)}{2}
\end{equation}
for any $k \in \ZZ$.
Let $g(\rho, \varphi) := \exp(-u) = \exp\bigl(-\rho\rho_0 \cos(\varphi - \varphi_0)\bigr)$ and $v = \rho\rho_0 \sin(\varphi - \varphi_0)$.
Consider polynomials $t_n \in \CC[u,v]$ defined by the rule:
$
\dfrac{\partial^n g}{\partial \varphi^n} = t_n  \cdot g.
$
We have:  $t_0 = 1$ and
\begin{equation}\label{E:BalabuevPolynomials}
t_{n+1}(u, v) := -v \frac{\partial t_n}{\partial u}(u, v) + u \frac{\partial t_n}{\partial v}(u, v) + v t_n(u, v)
\end{equation}
for any $n \in \NN$. Note that the highest order term of $t_n(u, v)$ is $v^{n}$. For any $\alpha \in \Pi$ and $j \in \NN_0$ we put:
$
t_{(\alpha, j)} := t_n(u, v)\Big|_{\varphi = \alpha} \in \CC[\rho].
$
Then in notation (\ref{L:TaylorSeriesVariation}) we have:
$$
\widehat{T}_{(\alpha, 2\mu_\alpha)}(-u) = \sum\limits_{j = 0}^{2\mu_\alpha -1} \frac{t_{(\alpha, j)}(\rho)}{j!} \varepsilon^j \; \in \;  \CC[\rho, \varepsilon]/(\varepsilon^{2\mu_\alpha}).
$$
By definition, a polynomial  $f \in R$ belongs to the subspace $P(-u)$ if and only if
\begin{equation}\label{E:systemWi}
\left(\sum\limits_{k = 0}^{2\mu_\alpha -1} \frac{f_{\alpha}^{(k)}}{k!} \varepsilon^k \right) \cdot
\left(\sum\limits_{j = 0}^{2\mu_\alpha -1} \frac{t_{(\alpha, j)}}{j!} \varepsilon^j\right)  \in \CC[\rho, \varepsilon^2]/(\varepsilon^{2\mu_\alpha}).
\end{equation}
The constraint (\ref{E:systemWi}) is equivalent to the following system of polynomial identities:
\begin{equation}\label{E:BlueSystem}
\left\{
\begin{array}{lcc}
f'_\alpha + t_{(\alpha, 1)} f_\alpha & = &  0 \\
\dfrac{f'''_\alpha}{3!} + \dfrac{f''_\alpha}{2!} \cdot \dfrac{t_{(\alpha, 1)}}{1!} + \dfrac{f'_\alpha}{1!} \cdot \dfrac{t_{(\alpha, 2)}}{2!} + f_\alpha \cdot
\dfrac{t_{(\alpha, 3)}}{3!} & = & 0 \\
\quad   \vdots  & & \\
\sum\limits_{j = 0}^{2\mu_\alpha -1} \dfrac{f^{(2\mu_\alpha-1-j)}_\alpha}{(2\mu_\alpha-1-j)!} \cdot \dfrac{t_{(\alpha, j)}}{j!} & = &  0.
\end{array}
\right.
\end{equation}
Let $d = \deg(f)$ and $f = f_d + f_{d-1} + \dots + f_0$ be a decomposition of $f$ into a sum of its homogeneous components. We prove the following

\smallskip
\noindent
\underline{Claim}. Suppose  that $f \in R$ satisfies  the system (\ref{E:BlueSystem}) for $m = \mu_\alpha$. Then  $l_\alpha^{m}$ divides $f_d$.

\smallskip
\noindent
\emph{Proof of the claim}.
It is instructive to begin with the special cases, when $m = 1$ or $2$.

\smallskip
\noindent
For $m = 1$, the system (\ref{E:BlueSystem}) consists only of one equation: $f'_\alpha + t_{(\alpha, 1)} f_\alpha = 0$.
We put $\varrho:= \rho \rho_0 \sin(\alpha - \varphi_0)$ (note, that the conditions on  the vector $\vec\xi$ insure that $\varrho \ne 0$).  Taking the homogeneous part of the top degree of the left--hand side, we obtain: $\varrho \cdot \bigl(f_d\bigr)_\alpha = 0$ (recall that the top degree homogeneous part of $t_{(\alpha, j)}(u, v)$
is $v^j$). Hence, we have the vanishing $\bigl(f_d\bigr)_\alpha = 0$,  which implies  that $l_\alpha \, \big| \,  f_d$.

\smallskip
\noindent
Let  $m = 2$. In this case,  the system (\ref{E:BlueSystem}) consists  of two equations:
\begin{equation}\label{E:BlueSystemSpecial}
\left\{
\begin{array}{lcl}
f'_\alpha + t_{(\alpha, 1)} f_\alpha & = &  0 \\
\dfrac{f'''_\alpha}{3!} + \dfrac{f''_\alpha}{2!} \cdot \dfrac{t_{(\alpha, 1)}}{1!} + \dfrac{f'_\alpha}{1!} \cdot \dfrac{t_{(\alpha, 2)}}{2!} + f_\alpha \cdot
\dfrac{t_{(\alpha, 3)}}{3!} & = & 0.
\end{array}
\right.
\end{equation}
We have already seen in the previous step, that the first equation of (\ref{E:BlueSystemSpecial}) implies that $\bigl(f_d\bigr)_\alpha = 0$. Taking
the top degree (with respect to $\rho$) of both equations of (\ref{E:BlueSystemSpecial}), we get the following system:
$$
\left\{
\begin{array}{lcc}
\bigl(f_d\bigr)'_\alpha + \varrho \bigl(f_{d-1}\bigr)_\alpha & = &  0 \\
\dfrac{\varrho^2}{2!} \bigl(f_d\bigr)'_\alpha + \dfrac{\varrho^3}{3!} \bigl(f_{d-1}\bigr)_\alpha
& = & 0.
\end{array}
\right.
$$
Since
$
\det
\left(
\begin{array}{cc}
1 & \varrho \\
\dfrac{\varrho^2}{2!} & \dfrac{\varrho^3}{3!}
\end{array}
\right) = -\dfrac{1}{3} \varrho^3 \ne 0,
$
we conclude that $\bigl(f_d\bigr)'_\alpha = \bigl(f_{d-1}\bigr)_\alpha = 0$. The conditions $\bigl(f_d\bigr)_\alpha  = \bigl(f_d\bigr)'_\alpha = 0$
imply that $l_\alpha^2 \, \big| \, f_d$; see Lemma \ref{L:TaylorFractFields}.

\smallskip
\noindent
Now we proceed to the general case. We prove by induction on $m$ that
\begin{equation}\label{E:systemW}
\left\{
\begin{array}{cccccccccccc}
\bigl(f_d\bigr)_\alpha & = &  \bigl(f_d\bigr)'_\alpha & =  & \dots & = & \bigl(f_d\bigr)^{(m-2)}_\alpha & = & \bigl(f_d\bigr)^{(m-1)}_\alpha & = & 0 \\
\bigl(f_{d-1}\bigr)_\alpha & = &  \bigl(f_{d-1}\bigr)'_\alpha & = & \dots  & = & \bigl(f_{d-1}\bigr)^{(m-2)}_\alpha & = & 0 & & \\
\vdots & & & & & & & & & & & \\
\bigl(f_{d-m+1}\bigr)_\alpha & = & 0. & & & & & & & & &
\end{array}
\right.
\end{equation}
Consider the following infinite matrix
\begin{equation}\label{E:matrixreloaded}
\left(
\begin{array}{ccccccccc}
\vspace{1mm}
\binom{1}{1} & \binom{1}{1} \lambda & 0 & 0 & 0 & 0 & 0 & 0  & \dots \\
\vspace{1mm}
\binom{3}{0} & \binom{3}{1} \lambda & \binom{3}{2} \lambda^2 & \binom{3}{3} \lambda^3 & 0 & 0 & 0 & 0  & \dots \\
\vspace{1mm}
\binom{5}{0} & \binom{5}{1} \lambda & \binom{5}{2} \lambda^2 & \binom{5}{3} \lambda^3 & \binom{5}{4} \lambda^4 & \binom{5}{5} \lambda^5 & 0 & 0 & \dots \\
\vspace{1mm}
\binom{7}{0} & \binom{7}{1} \lambda & \binom{7}{2} \lambda^2 & \binom{7}{3} \lambda^3 & \binom{7}{4} \lambda^4 & \binom{7}{5} \lambda^5 & \binom{7}{6} \lambda^5 & \binom{7}{7} \lambda^7 & \dots\\
\vspace{1mm}
\binom{9}{0} & \binom{9}{1} \lambda & \binom{9}{2} \lambda^2 & \binom{9}{3} \lambda^3 & \binom{9}{4} \lambda^4 & \binom{9}{5} \lambda^5 & \binom{9}{6} \lambda^5 & \binom{9}{7} \lambda^7 & \dots\\
\vspace{1mm}
\vdots & \vdots  & \vdots & \vdots  & \vdots & \vdots  & \vdots  & \vdots  & \ddots\\
\end{array}
\right).
\end{equation}
To prove the induction step, it is sufficient to show that the  first principal $(m \times m)$--minor  of the matrix (\ref{E:matrixreloaded}) in non--zero. Suppose it is not the case. Then the elements
$$
\overline{(1+\lambda)}, \overline{(1+\lambda)^3}, \dots, \overline{(1+\lambda)^{2m-1}} \in \CC[\lambda]/(\lambda^{m})
$$
are linearly dependent. Hence, there exist $c_0, c_1, \dots, c_{m-1} \in \CC$ such that
$\lambda^{m}$ divides the polynomial $c_0 + c_1 (1+\lambda)^2 + \dots + c_{m-1} (1+\lambda)^{2\cdot(m-1)}$. Let $1 \le \bar{m} \le m$ be such that
$c_{\bar{m}-1} \ne 0$, whereas $c_{\bar{m}} = c_{\bar{m}+1}  = \dots = 0$. Let $\zeta_1, \dots, \zeta_{\bar{m}-1} \in \CC$ be  such that
$$
\pi(t) := c_0 + c_1 t + \dots + c_{\bar{m}-1} t^{\bar{m}-1} = c_{\bar{m}-1} (t-\zeta_1)\cdot \dots \cdot (t-\zeta_{\bar{m}-1}) \in \CC[t].
$$
Taking the substitution $t = (1 + \lambda)^2$, we get:
\begin{equation}\label{E:polynomials}
\lambda^m \; \big|\; c_{\bar{m}-1} \prod\limits_{j=1}^{\bar{m}-1} \bigl((1+\lambda)^2 - \zeta_j\bigr).
\end{equation}
However, the order of vanishing at $0$ of the polynomial in the right hand side of (\ref{E:polynomials}) is at most $\bar{m}-1$, contradiction.
Hence, $c_0 = c_1 = \dots = c_{m-1} = 0$.

\smallskip
\noindent
Therefore, for any $f \in P(-u)$ we have: $\bigl(f_d\bigr)_\alpha =  \bigl(f_d\bigr)'_\alpha =  \dots = \bigl(f_d\bigr)^{(m-1)}_\alpha  = 0$.
By Lemma \ref{L:TaylorFractFields}, the polynomial $l_\alpha^m$ divides $f_d$, proving the claim.

\medskip
\noindent
Summing up, we have shown  that  the polynomial $\delta = \prod\limits_{\alpha \in \Pi} l_{\alpha}^{\mu_\alpha}$ divides the top homogeneous part of any element of the module $P(-u)$. Since $\deg(\delta) = \mu$, it  implies that
$$
\dim_\CC\bigl(P(-u)_{\mu}\bigr) \le 1 \quad \mbox{\rm and} \quad  \dim_\CC\bigl(P(-u)_{\mu+k}/P(-u)_{\mu+k-1}\bigr) \le k+1
$$
for any $k \in \NN$. Therefore,
$$
 \dim_\CC\bigl(P(-u)_{\mu+k}\bigr) \le 1 +2 + \dots + (k+1) = \frac{(k+1)(k+2)}{2}.
$$
Theorem is proven.
\end{proof}

\begin{corollary} Since the normalization map $\AA^2 \stackrel{\nu}\lar X$ is bijective, any character $A \stackrel{\chi}\lar \CC$ is given as the composition
$
\xymatrix{
A \ar@{^{(}->}[r] & R \ar[r]^-{\widetilde\chi} & \CC,
}
$
where $\widetilde{\chi}(P) = P(\zeta_1, \zeta_2)$ for any $P \in \CC[z_1, z_2]$ and some uniquely determined $(\zeta_1, \zeta_2) \in \AA^2$. For any $\vec\xi \in \CC^2$ satisfying the
conditions of Theorem \ref{T:DescriptSpectModule}, the power series
$\Psi(x_1, x_2; \zeta_1, \zeta_2; \xi_1, \xi_2)$, given by formula (\ref{E:BerestShifted}), is non--zero and regular at $(x_1, x_2) = (0, 0)$. Combining Theorem \ref{T:spectralmodule} with Theorem \ref{T:listCM}, we end up with  the following result:
$$
\mathsf{Sol}(\gB, \chi) = \Bigl\{f\in \CC\llbracket x_1, x_2\rrbracket \Big| P\diamond f = P(\zeta_1, \zeta_2)f \; \mbox{\rm for all}\; P \in \gB\Bigr\}   =  \bigl\langle\Psi(x_1, x_2; \zeta_1, \zeta_2; \xi_1, \xi_2)\bigr\rangle_\CC.
$$
\end{corollary}

\begin{remark}
The $A$--module $P(-u)$ from Theorem \ref{T:DescriptSpectModule} appeared also in  \cite[Proposition 7.6]{BerestEtingofGinzburg}, where another proof of its projectivity was given.
\end{remark}

\section{Elements of the higher--dimensional Sato theory}\label{S:SatoTheory}
For any $n \in \NN$, let $B = \CC[x_1, \dots, x_n]$ and  $\widehat{B} = \CC\llbracket x_1, \dots, x_n\rrbracket$. For any $d \in \ZZ$, we denote by  $B_d$ be the vector space of homogeneous elements of $B$ of degree $d$ (in particular, $B_d = 0$ for $d \le 0$). Next, let $\idm = (x_1, \dots, x_n)$ be the unique maximal ideal of $\widehat{B}$ and $\widehat{B} \stackrel{\upsilon}\lar \NN_0 \cup \{\infty\}$ be the corresponding valuation. To simplify the notation, we denote
$\Sigma := \NN_0^n$ and for any $\underline{k} = (k_1, \dots, k_n) \in \Sigma$ we  write:
\begin{itemize}
\item $\underline{x}^{\underline{k}} := x_1^{k_1} \dots x_n^{k_n}$ and $\underline{\partial}^{\underline{k}} := \partial_1^{k_1} \dots \partial_n^{k_n}$.
\item $\underline{k}! = k_1! \dots k_n!$ and $|\underline{k}| = k_1 + \dots + k_n$.
\end{itemize}
We denote $\underline{0} := (0, \dots, 0) \in \Sigma$ and for $\underline{k}, \underline{l} \in \Sigma$ say that $\underline{k} \ge \underline{l}$ if any only if $k_i \ge l_i$ for all $1 \le i \le n$. Next, consider the following $\CC$--vector space:
\begin{equation}
\gO := \CC\llbracket x_1, \dots, x_n\rrbracket\llbracket \partial_1, \dots, \partial_n\rrbracket = \left\{
\sum\limits_{\underline{k} \ge \underline{0}} a_{\underline{k}} \underline{\partial}^{\underline{k}} \; \left|\;  a_{\underline{k}} \in \widehat{B} \right. \;\mbox{for all}\;  \underline{k} \in \Sigma
\right\}.
\end{equation}
Note that $\gO$ has no natural $\CC$--algebra  structure since the natural product $\cdot$ is not defined on the whole vector space $\gO$.
\begin{definition}\label{D:AlgebraPi} For any element
$
P := \sum\limits_{\underline{k} \ge \underline{0}} a_{\underline{k}} \underline{\partial}^{\underline{k}} \in \gO
$
we define its \emph{order} $o(P) \in \mathbbm{Z} \cup \{\pm \infty\}$ to be
\begin{equation}\label{E:LastOrder}
o(P) := \left\{
\begin{array}{cl}
\sup\bigl\{|\underline{k}| - \upsilon(a_{\underline{k}}) \; \big|\; \underline{k} \in \Sigma \bigr\} & P \ne 0\\
-\infty & P = 0.
\end{array}
\right.
\end{equation}
In particular, if $d = o(P) < \infty$ then we have:
$$
\upsilon(a_{\underline{k}}) \ge |\underline{k}| - d  = (k_1 +\dots + k_n) - d \quad \mbox{for any} \; \underline{k} \in \Sigma.
$$
The key role in this section is plaid by the following subspace of the vector space $\gO$:
\begin{equation}
\gP := \bigl\{Q \in \gO \,\big|\, o(Q) < \infty \bigr\}.
\end{equation}
Note that for a partial differential  operator $$
P = \sum_{|\underline{k}| = m} a_{\underline{k}} \underline{\partial}^{\underline{k}} + \sum_{|\underline{i}| < m}
b_{\underline{i}}\underline{\partial}^{\underline{i}}
\in \CC\llbracket x_1, \dots, x_n\rrbracket[\partial_1, \dots, \partial_n],$$
with constant highest symbol $0 \ne \sigma(P) = \sum_{|\underline{k}| = m} \alpha_{\underline{k}} \underline{\partial}^{\underline{k}}
 \in
\CC[\partial_1, \dots, \partial_n]$, the order of $P$ taken in the sense (\ref{E:LastOrder}) is equal to $m$ and coincides with the usual definition of the order of a differential operator.

\smallskip
\noindent
Let $P \in \gP$. Then for any $\underline{k}, \underline{i} \in \Sigma$, we have a uniquely determined  $\alpha_{\underline{k}, \underline{i}} \in \CC$ such that
\begin{equation}\label{E:expOperatorP}
P = \sum\limits_{\underline{k}, \underline{i} \, \ge \, \underline{0}} \alpha_{\underline{k}, \underline{i}} \,  \underline{x}^{\underline{i}} \underline{\partial}^{\underline{k}}.
\end{equation}
For any $m \ge -d$ we put:
$$
P_m:= \sum\limits_{\substack{\underline{k}, \underline{i} \in \Sigma \\ |\underline{i}| - |\underline{k}| = m}} \alpha_{\underline{k}, \underline{i}} \,  \underline{x}^{\underline{i}} \underline{\partial}^{\underline{k}}
$$
to be the $m$-th \emph{homogeneous component} of $P$. Note that $o(P_m) = -m$ and we have a decomposition
$
P = \sum\limits_{m=-d}^\infty P_m.
$
Finally,  $\sigma(P) := P_{-d}$ is  the \emph{symbol} of $P$ (i.e.~the sum of all components of $P$ of maximal order). We say that $P \in \gP$ is \emph{homogeneous} if $P = \sigma(P)$.
\end{definition}

\begin{example}\label{E:elementsofPi} Let $n = 1$. Then we have:
\begin{itemize}
\item The operator
$
\exp(x \astar \partial) := \sum\limits_{k= 0}^\infty \dfrac{x^k}{k!} \partial^k
$
belongs to $\gP$. Moreover, $\exp(x \astar \partial)$ is homogeneous of order zero.
\item The element
$
\sum\limits_{k= 0}^\infty \dfrac{x^k}{k!} \partial^{2k}
$
of $\gO$
does not belong to $\gP$.
\end{itemize}
\end{example}

\begin{theorem} The following results are true.
\begin{enumerate}
\item The vector space $\gP$ is a $\CC$--algebra with respect to the natural operations $\cdot$ and $+$. In particular, $\gP$ contains the subalgebra
$\gD:= \CC\llbracket x_1, \dots, x_n\rrbracket[\partial_1, \dots, \partial_n]$ of partial differential operators.
\item We have a natural isomorphism of $\CC$--vector spaces
$
F:= \gP/\idm \gP \lar \CC[\partial_1, \dots, \partial_n].
$
\item We have a natural injective algebra homomorphism
$
\gP \lar \End_{\CC}^{\mathrm{c}}(\widehat{B}),
$
where $\End_{\CC}^{\mathrm{c}}(\widehat{B})$ is the algebra of $\CC$--linear endomorphisms of $\widehat{B}$, which are \emph{continuous} in the
$\idm$--adic topology. In particular, $\widehat{B}$ has a natural structure of a left $\gP$--module, which extends  its  natural structure of a left $\gD$--module.
\end{enumerate}
\end{theorem}

\begin{proof} (1) The main point is to show  that the natural product $\cdot$ is well--defined for any pair of elements $P, Q \in \gP$. Let $d = o(P)$ and
$e = o(Q)$. Assume first that $P$ and $Q$ are homogeneous. Then we have presentations
$
P = \sum\limits_{\underline{k} \in \Sigma} a_{\underline{k}} \underline{\partial}^{\underline{k}} \quad \mbox{\rm and} \quad
Q = \sum\limits_{\underline{l} \in \Sigma} b_{\underline{l}} \underline{\partial}^{\underline{l}},
$
where $a_{\underline{k}} \in B_{|\underline{k}|-d}$ and $b_{\underline{l}} \in B_{|\underline{l}|-e}$ for any $\underline{k}, \underline{l} \in \Sigma$.

\smallskip
\noindent
Having  the Leibniz formula in mind, we \emph{define}:
\begin{equation}\label{E:DefProduct}
P\cdot Q :=
\sum_{\underline{k} \in \Sigma} \sum_{\underline{l} \in \Sigma} \sum_{\underline{0} \, \le \, \underline{i} \, \le  \, \underline{k}}
\binom{k_1}{i_1} \dots
\binom{k_n}{i_n} a_{\underline{k}}
\frac{\partial^{|\underline{i}|} \, b_{\underline{l}}}{\partial x_1^{i_1} \dots \partial x_n^{i_n}} \partial^{\underline{k} + \underline{l} - \underline{i}}.
\end{equation}
Since for any $\underline{j} \in \Sigma$,  there exist only finitely many $\underline{k}, \underline{l}, \underline{i} \in \Sigma$ such that
$$
\underline{j} = \underline{k} + \underline{l} - \underline{i}, \quad \underline{k} \ge \underline{i} \; \; \mbox{\rm and} \; \;  \underline{l} \ge \underline{i},
$$
the right--hand side of (\ref{E:DefProduct}) is a well--defined \emph{homogeneous} element of $\gP$. Moreover, we have: 
$o(P\cdot Q) = o(P) + o(Q)$ provided $P\cdot Q \ne 0$.

Now, let $P, Q \in \gP$ be arbitrary elements and $P = \sum\limits_{m=-d}^\infty P_m$ respectively  $Q = \sum\limits_{l=-e}^\infty Q_l$ be the corresponding homogeneous decompositions. Then we put:
\begin{equation}
P\cdot Q :=
\sum\limits_{p=-(d+e)}^\infty\left(\sum\limits_{\substack{m+l = p \\ m \ge -d \\ l \ge -e}}  P_m \cdot Q_l\right).
\end{equation}
It is a tedious but straightforward computation to verify  that $\gP$ is indeed a $\CC$--algebra with respect to the introduced operations $\cdot$ and $+$.
Note that $\sigma(P \cdot Q) = \sigma(P) \cdot \sigma(Q)$, provided $\sigma(P) \cdot \sigma(Q) \ne 0$.

\smallskip
\noindent
(2) Note that we have a well--defined injective $\CC$--linear map
$$
\gP/\idm \gP \lar \CC\llbracket \partial_1, \dots, \partial_n\rrbracket, \quad P = \sum\limits_{\underline{k} \ge \underline{0}} a_{\underline{k}} \underline{\partial}^{\underline{k}} \mapsto P\big|_{\underline{0}}  :=  \sum\limits_{\underline{k} \ge \underline{0}}   a_{\underline{k}}(0) \underline{\partial}^{\underline{k}},
$$
whose image contains the subspace $\CC[\partial_1, \dots, \partial_n]$. Let $d = o(P)$, then by the definition we have: $\upsilon(a_{\underline{k}}) \ge |\underline{k}|-d$ for any $\underline{k} \in \Sigma$. In particular, $a_{\underline{k}}(0) = 0$ for any $\underline{k} \in \Sigma$ such that $|\underline{k}| \ge d+1$, hence
$P\big|_{\underline{0}} \in \CC[\partial_1, \dots, \partial_n]$ as claimed.

\smallskip
\noindent
(3) In order to define the natural left action of the $\CC$--algebra $\gP$ on $\widehat{B}$, take first $P \in \gP$ homogeneous of order $d \in \ZZ$ and
$f \in B_e$ for some $e \in \NN_0$. Then we have an expansion $P = \sum\limits_{\underline{k} \ge \underline{0}} a_{\underline{k}} \underline{\partial}^{\underline{k}}$ with $a_{\underline{k}} \in B_{|\underline{k}|-d}$ for any $\underline{k} \in \Sigma$. Since
$\underline{\partial}^{\underline{k}} \circ f = 0$ for any $\underline{k} \in \Sigma$ such that $|\underline{k}| \ge e+1$, we have a well--defined element
$P\circ f \in B_{e-d}$.

Now, let $P \in \gP$ and $f \in \widehat{B}$ be arbitrary elements and $d = o(P)$. Since we have homogeneous decompositions $P = \sum\limits_{m=-d}^\infty P_m$ and $f = \sum\limits_{e=0 }^\infty f_e$, we can  define:
$$
P\circ f :=
\sum\limits_{k = 0}^\infty \bigl(\sum\limits_{\substack{ m \ge -d,\;  e \ge 0 \\ m+e = k}} P_m \circ f_e\bigr).
$$
It follows from the definition that $P \circ f \in \idm^k$ provided $f \in \idm^{k+d}$. This shows   that the action of $\gP$ on $\widehat{B}$ is indeed continuous in the $\idm$--adic topology.

\smallskip
\noindent
It remains to prove  that the algebra homomorphism $
\gP \lar \End_{\CC}^{\mathrm{c}}(\widehat{B})
$
is injective. For this, it  is sufficient to show that for any homogeneous operator
$$
P = \sum\limits_{\substack{\underline{k}, \underline{i} \, \ge \, \underline{0}\\ |\underline{k}|- |\underline{i}| = d}} \alpha_{\underline{k}, \underline{i}} \,  \underline{x}^{\underline{i}} \underline{\partial}^{\underline{k}} \in \gP
$$
of order $d$, there exists $f \in \widehat{B}$ such that $P \circ f \ne 0$. Let $\underline{l}$ be an element
of the set
$$
\left\{\underline{k} \in \Sigma \big| \; \mbox{\rm there exists} \;  \underline{i} \in \Sigma \; \mbox{\rm such that} \; \alpha_{\underline{k}, \underline{i}} \ne 0\right\}
$$
 with $|\underline{l}|$  smallest possible. Then $P \circ \underline{x}^{\underline{l}} = \underline{l}! \, \alpha_{\underline{l}, \underline{i}}\, \underline{x}^{\underline{i}} \ne 0$, implying the statement.
\end{proof}

\noindent
One evidence that the algebra $\gP$ deserves a further study is due to the fact that it contains several  operators, which do not belong to the subalgebra $\gD$ but act ``naturally'' on $\widehat{B}$.

\begin{example}\label{Ex:DeltaFunctionSeries} Let $n = 1$. For any $u  \in x\CC\llbracket x\rrbracket$, consider the following operator:
\begin{equation}
\exp(u \astar \partial) := \sum\limits_{k= 0}^\infty \dfrac{u^k}{k!} \partial^k.
\end{equation}
Obviously, $\exp(u \astar \partial)$ is an operator of the algebra $\gP$ of non--positive order. Moreover, for any $f \in \widehat{B}$ we have:
\begin{equation}
\exp(u \astar \partial) \circ f(x) = f\bigl(u + x\bigr),
\end{equation}
i.e.~the operator $\exp(u \astar \partial)$ can realize an arbitrary $\CC$--linear endomorphism of $\CC\llbracket x\rrbracket$. Indeed,
$$
 \exp(u \astar \partial)\circ x^m = \left(\sum\limits_{k = 0}^m \dfrac{u^k}{k!} \partial^k\right) \circ x^m = \sum\limits_{k = 0}^m \binom{m}{k} u^k x^{m-k} = \bigl(u+ x\bigr)^m,
$$
implying the statement.

\smallskip
\noindent
In particular, let $E:= \exp((-x) \astar \partial)$. Then $E  \circ f(x) = f(0)$, i.e.~the operator $E$ is  Dirac's delta--function.

\smallskip
\noindent
Now, let $n \in \NN$ be arbitrary. For any $1 \le i \le n$ and $u \in \idm$, consider the operator
\begin{equation}\label{E:operatorE}
\exp(u \astar \partial_i) := \sum\limits_{k= 0}^\infty \dfrac{u^k}{k!} \partial_i^k.
\end{equation}
Then  for any $f \in \widehat{B}$, we have the formula:
\begin{equation*}
\exp(u \astar \partial_i)
 \circ f(x_1, \dots, x_{i-1}, x_i, x_{i+1}, \dots, x_n) = f\bigl(x_1, \dots, x_{i-1}, x_i + u, x_{i+1}, \dots, x_n\bigr).
\end{equation*}
As a special case,  for $E_i:= \exp((-x_i) \astar \partial_i)$ we have the following formula:
\begin{equation}\label{E:DeltaFunctSeries}
E_i \circ f(x_1, \dots, x_{i-1}, x_i, x_{i+1}, \dots, x_n) = f\bigl(x_1, \dots, x_{i-1}, 0, x_{i+1}, \dots, x_n\bigr).
\end{equation}

\smallskip
\noindent
Finally, note that  the formula (\ref{E:DeltaFunctSeries}) implies that $\partial_i \cdot E_i = 0$ in $\gP$. As a consequence,
$\bigl(E_i \cdot \partial_i\bigr)^2 = 0$ for any $1 \le i \le n$.
\end{example}

\begin{example}\label{Ex:Integration} Again, let us first assume that $n = 1$. Consider the operator
\begin{equation}
G:= \bigl(1- \exp((-x) \astar \partial)\bigr) \cdot \partial^{-1} =
\sum\limits_{k = 0}^\infty \frac{x^{k+1}}{(k+1)!} (-\partial)^k.
\end{equation}
Then for any $m \in \NN_0$ we have:
$$
G \circ x^m = \sum\limits_{k= 0}^m (-1)^k \frac{x^{k+1}}{(k+1)!} \bigl(\partial^k \circ x^m) = \frac{x^{m+1}}{m+1}\sum\limits_{k = 0}^m (-1)^{k+1} \binom{m+1}{k+1} =  \frac{x^{m+1}}{m+1}.
$$
Hence,
\begin{equation}
G \circ \sum\limits_{m = 0}^\infty a_m  x^m = \sum\limits_{m = 0}^\infty   \frac{a_m}{m+1} x^{m+1},
\end{equation}
i.e.~$G$ acts on $\widehat{B}$ as  the integration operator. In particular, we have: $\partial \cdot G = 1$ in $\gP$.

\smallskip
\noindent
Similarly, for $n \in \NN$ and any $1 \le i \le n$, the operator $G_i := \bigl(1- \exp((-x_i) \astar \partial_i)\bigr) \cdot \partial_i^{-1}$ is the operator of indefinite integration in the $i$-th variable.
\end{example}

\smallskip
\noindent
In what follows, we shall study more precisely the right action of the algebra $\gP$ on the $\CC$--vector space $F = \gP/\idm \gP = \CC[\partial_1, \dots, \partial_n]$.

\begin{definition}
Let $P \in \gP$ be an operator of order $d$ given by the expansion (\ref{E:expOperatorP}). Then we have another form of the formal power series expansion of $P$ called \emph{slice decomposition}:
\begin{equation}
P = \sum\limits_{\underline{i} \ge \underline{0}} \frac{\underline{x}^{\underline{i}}}{\underline{i}!} \, P_{(\underline{i})}, \quad \mbox{\rm where} \quad
P_{(\underline{i})} = \underline{i}! \sum_{\substack{\underline{k} \ge \underline{0} \\ |\underline{k}| - |\underline{i}| \le d }} \alpha_{\underline{k}, \underline{i}} \, \underline{\partial}^{\underline{k}}.
\end{equation}
For any $\underline{i} \in \Sigma$, the partial differential operator with constant coefficients  $P_{(\underline{i})} \in \CC[\partial_1, \dots, \partial_n]$ is called $\underline{i}$-th slice of $P$.
\end{definition}

\begin{remark}\label{R:slicedecomp} Note that for any $\underline{i} \in \Sigma$ we have the following identity
$\underline{\partial}^{\underline{i}} \circ P = P_{(\underline{i})}$, where $P_{(\underline{i})}$ is viewed as an element of
 the module $F$. In particular, for any  $P, Q \in \gP$, the following statement is true:
$$
P = Q \quad \mbox{if and only if} \quad \underline{\partial}^{\underline{i}} \circ P = \underline{\partial}^{\underline{i}} \circ Q
\quad \mbox{for any} \quad \underline{i} \in \Sigma.
$$
In other words, the algebra homomorphism $\gP \lar \End_{\CC}(F)$ is injective.
\end{remark}

\begin{definition}\label{D:regular}
An element $P \in \gP$ is called \emph{regular} if the $\CC$--linear map $F \xrightarrow{-\circ \sigma(P)} F$ is injective. In particular, $P$ is regular if and only if its symbol $\sigma(P)$ is regular.
\end{definition}

\begin{lemma}\label{L:regularOp}
Let $P \in \gP$. Then the following results are true.
\begin{enumerate}
\item The operator $P$ is regular if and only if for any $m \in \NN_0$, the elements of the set
$\bigl\{\underline{\partial}^{\underline{k}} \circ \sigma(P)\,  \big| \, \underline{k} \in \Sigma: |\underline{k}| = m\bigr\} \subset F$
are linearly independent.
\item Assume that $P$ is regular. Then $P$ is not a right zero divisor in $\gP$, i.e.~the equation $Q \cdot P = 0$ in $\gP$ implies that $Q = 0$.
\end{enumerate}
\end{lemma}

\begin{proof} (1) Let $d := o(P) = o\bigl(\sigma(P)\bigr)$. Then for any $\underline{k} \in \Sigma$ we have:
$o\bigl(\underline{\partial}^{\underline{k}} \circ \sigma(P)\bigr) = |\underline{k}| + d$ provided $\underline{\partial}^{\underline{k}} \circ \sigma(P) \ne 0$. Therefore, the linear map $F \xrightarrow{-\circ \sigma(P)} F$ splits into a direct sum of its graded components
$F_m \xrightarrow{-\circ \sigma(P)} F_{m+d}$, implying the first statement.

\smallskip
\noindent
(2) Let $Q \ne 0$ be such that $Q \cdot P = 0$. Then $\sigma(Q) \ne 0$ as well,  whereas  $\sigma(Q) \cdot \sigma(P) = 0$. Next, there exists
$\underline{k} \in \Sigma$ such that $\overline{\underline{\partial}^{\underline{k}} \cdot \sigma(Q)} \ne 0$ in $F$. On the other hand,
$$
\overline{\underline{\partial}^{\underline{k}} \cdot \sigma(Q)} \circ \sigma(P) = \overline{\underline{\partial}^{\underline{k}} \cdot \sigma(Q) \cdot \sigma(P)} = 0,
$$
hence $P$ is not regular, contradiction.
\end{proof}

\begin{definition}
Let $\gP_{-} := \bigl\{P \in \gP \, \big| \, o(P) \le 0 \bigr\}$ and $\gP_{-}^\ast$ be the group of units of $\gP_{-}$.
\end{definition}

\begin{lemma} The following results are true.
\begin{enumerate}
\item Let $P \in \gP_{-}$. Then we have: $P \in \gP_{-}^\ast$ if and only if $\sigma(P) \in \gP_{-}^\ast$.
\item Let $Q \in \gP_{-}^\ast$. Then $o(Q) = 0$ and
$\overline{\sigma(Q)} = Q_{(\underline{0})} \in \CC^\ast$.
\end{enumerate}
\end{lemma}

\begin{remark} It is not true that any  unit in the algebra $\gP$ belongs to its subalgebra $\gP_-$. Indeed,
let $P = \exp\bigl((-x) \astar \partial\bigr)\cdot \partial  \in \gP$. Then $P^2 = 0$, hence  $1 + P$ is a unit in $\gP$, which is not an element of  $\gP_-$.
\end{remark}

\begin{definition} Let $W \subseteq F = \CC[\partial_1, \dots, \partial_n]$ be a $\CC$--linear subspace.
\begin{enumerate}
\item For any $k \in \ZZ$, we put:
$
W_k := \bigl\{w \in W \, \big| \, o(w) \le k \bigr\}.
$
\item $H_W(k):= \dim_{\CC}\bigl(W_k\bigr)$ is the \emph{Hilbert function} of $W$.
\end{enumerate}
\end{definition}

\begin{definition}\label{D:SatoGrassmannian} Let $\mu \in \NN_0$.
\begin{enumerate}
\item We put:
$
\mathsf{Gr}_\mu(F) := \Bigl\{W \subseteq F \, \Big| \, H_W(\mu +k) = \binom{n+k}{n} \quad \mbox{\rm for any} \quad k \in \NN_0 \Bigr\}
$ (recall that $H_F(k) = \binom{n+k}{n}$).
\item Let $W \in \mathsf{Gr}_\mu(F)$. Then $S \in \gP$ is a  \emph{Sato operator} of $W$ if the following conditions are fulfilled:
\begin{itemize}
\item $S$ is regular and $o(S) = \mu$.
\item We have: $W = F \circ S$.
\end{itemize}
\end{enumerate}
\end{definition}

\begin{proposition}\label{P:SatoOperators}
Let $\mu \in \NN_0$, $W \in \mathsf{Gr}_\mu(F)$ and $T$ be a Sato operator of $W$. Then the following results are true.
\begin{enumerate}
\item $U \cdot T$ is a Sato operator of $W$ for any $U \in \gP_-^{\ast}$.
\item For any $m \in \NN_0$, the elements $\bigl\{T_{(\underline{k})} \, \big|\, \underline{k} \in \Sigma \; \mbox{\rm such that} \; |\underline{k}| \le m \bigr\}$ form a basis of the vector space
$W_{\mu +m}$, where $T_{(\underline{k})} \in F$ is the $\underline{k}$-th slice of the operator $T$.
\item Moreover, for any $m \in \NN_0$, the elements $\bigl\{\overline{T}_{(\underline{k})} \, \big|\, \underline{k} \in \Sigma \; \mbox{\rm such that}\; |\underline{k}| =  m \bigr\}$ form a basis of the vector space
$W_{\mu + m}/W_{\mu + m -1}$, where $\overline{T}_{(\underline{k})}$ denotes the class of $T_{(\underline{k})}$.
\item The linear map $F \xrightarrow{-\circ T} W$ is a bijection.
\end{enumerate}
\end{proposition}
\begin{proof} (1) If $U \in \gP_-^\ast$ then $U \cdot T \ne 0$ and $o(U \cdot T) = o(U) + o(T) = \mu$. Next,
$F \circ (U \cdot T) = (F \circ U) \circ T = F \circ T = W$. Finally, $\sigma(U \cdot T) = \sigma(U) \cdot \sigma(T)$ and $\sigma(U) \in \CC^*$.
Hence, the linear map $F \xrightarrow{- \circ \sigma(U\cdot T)} F$ is injective, implying that the operator $U \cdot T$ is regular. Hence, $U \cdot T$ is indeed
a Sato operator of $W$.

\smallskip
\noindent
(2)  Recall that we have a slice expansion:
$
T = \sum\limits_{\underline{k} \ge \underline{0}} \dfrac{\underline{x}^{\underline{k}}}{\underline{k}!} T_{(\underline{k})}.
$
Since $o(T) = \mu$, we have: $o\bigl(T_{(\underline{k})}\bigr) \le \mu + |\underline{k}|$, i.e.
$T_{(\underline{k})} \in W_{|\underline{k}| + \mu}$ for any $\underline{k} \in \Sigma$.  Moreover,  there exists
$\underline{l} \in \Sigma$ such that $o\bigl(T_{(\underline{l})}) = \mu + |\underline{l}|$.

\smallskip
\noindent
Next, for any $\underline{k} \in \Sigma$ we put:
$
\bar{\sigma}\bigl(T_{(\underline{k})}\bigr) =
\left\{
\begin{array}{ccc}
\sigma\bigl(T_{(\underline{k})}\bigr) &\mbox{\rm if} & o\bigl(T_{(\underline{k})}) = \mu + |\underline{k}| \\
0 &\mbox{\rm if} & o\bigl(T_{(\underline{k})}) < \mu + |\underline{k}|. \\
\end{array}
\right.
$
Then we have the formula:
$
\sigma(T)  = \sum\limits_{\underline{k} \ge \underline{0}} \dfrac{\underline{x}^{\underline{k}}}{\underline{k}!} \bar{\sigma}\bigl(T_{(\underline{k})}\bigr),
$
which implies that
$\underline{\partial}^{\underline{k}} \circ \sigma(T) = \bar{\sigma}\bigl(T_{(\underline{k})}\bigr)$ for all $\underline{k} \in \Sigma$. Since $T$ is regular,
$\bar{\sigma}\bigl(T_{(\underline{k})}\bigr) \ne 0$  and $o\bigl(T_{(\underline{k})}) = \mu + |\underline{k}|$ for all $\underline{k} \in \Sigma$, hence
\begin{equation}\label{E:symbolT}
\sigma(T)  = \sum\limits_{\underline{k} \ge \underline{0}} \dfrac{\underline{x}^{\underline{k}}}{\underline{k}!} \sigma\bigl(T_{(\underline{k})}\bigr).
\end{equation}
By assumption, $\dim_{\CC}(W_{m + \mu}) = \binom{n+m}{n}$. Therefore,
to prove the second statement, it  is sufficient to show that for any $m \in \NN_0$, the elements $\bigl\{T_{(\underline{k})} \, \big|\, \underline{k} \in \Sigma \; \mbox{\rm such that} \; |\underline{k}| \le m \bigr\}$ are linearly independent. We prove it by induction on $m$. The case $m = 0$ is clear:
since $\dim_{\CC}(W_\mu) = 1$ and  $o\bigl(T_{(\underline{k})}\bigr) = |\underline{k}| + \mu$, we have $W_\mu = \bigl\langle T_{(\underline{0})}\bigr\rangle_\CC$. Next,
let $\bigl\{\beta_{\underline{k}} \in \CC \,\big|\, \underline{k} \in \Sigma:   |\underline{k}| \le m \bigr\}$ be such that
$
\sum\limits_{|\underline{k}| \le m} \beta_{\underline{k}} T_{(\underline{k})} = 0.
$
Then we have:  $
\sum\limits_{|\underline{k}| = m} \beta_{\underline{k}} \, \sigma\bigl(T_{(\underline{k})}\bigr) = 0
$
in the vector space $F_{m + \mu}$. From  formula (\ref{E:symbolT}) and Lemma \ref{L:regularOp} we deduce  that $\beta_{\underline{k}} = 0$ for any
$\underline{k} \in \Sigma$ such that $|\underline{k}| = m$. Proceeding by induction, we get the second claim.

\smallskip
\noindent
(3) Analogously, assume that $
\sum\limits_{|\underline{k}| = m} \gamma_{\underline{k}} \, \overline{T}_{(\underline{k})} = 0
$
in the quotient vector space $W_{m+\mu}/W_{m+\mu-1}$.
Then we get: $
\sum\limits_{|\underline{k}| = m} \gamma_{\underline{k}} \, \sigma\bigl({T}_{(\underline{k})}\bigr) = 0
$
in $F_{m + \mu}$, hence $\beta_{\underline{k}} = 0$ for any
$\underline{k} \in \Sigma$ such that $|\underline{k}| = m$. The third claim follows from the fact that $\dim_{\CC}\bigl(W_{m+\mu}/W_{m+\mu-1}\bigr) =
\binom{n+m-1}{m}.
$

\smallskip
\noindent
(4) For any $m \in \NN_0$, the  linear map $F_m \xrightarrow{-\circ T} W_{m+\mu}, \; \underline{\partial}^{\underline{k}} \mapsto T_{(\underline{k})}$
is an isomorphism by the dimension reasons. This implies the fourth statement.
\end{proof}

\begin{theorem}\label{T:SatoAction}
Let $\mu \in \NN_0$ and $W \in \mathsf{Gr}_\mu(F)$. Then the following statements are true.
\begin{enumerate}
\item The vector space $W$ possesses  a Sato operator $S$.
\item If $T$ is another Sato operator for $W$ then there exists a uniquely determined $U \in \gP_-^\ast$ such that $S = U\cdot T$.
\end{enumerate}
In other words, a Sato operator of $W$ exists and is unique up to a unit of the algebra $\gP_-$.
\end{theorem}

\begin{proof} (1) Our construction of a Sato operator $S$ is algorithmic and depends on the following choice. Namely, for any $\underline{k} \in \Sigma$, we choose
$w_{\underline{k}} \in W_{\mu + |\underline{k}|}$ such that for any $m \in \NN_0$, the set
$\bigl\{\bar{w}_{\underline{k}} \, \big| \, \underline{k} \in \Sigma: |\underline{k}| = m \bigr\}$ forms a basis of the vector space $W_{m+\mu}/W_{m+\mu-1}$ (at this place, we essentially use the assumption on the Hilbert function of $W$). Then the following statements are true:
\begin{itemize}
\item $o\bigl(w_{\underline{k}}\bigr) = \mu + |\underline{k}|$ for any $\underline{k} \in \Sigma$;
\item the  set $\bigl\{{w}_{\underline{k}} \, \big| \, \underline{k} \in \Sigma: |\underline{k}| \le  m \bigr\}$ is  a basis of the vector space $W_{m+\mu}$.
\end{itemize}
Consider the operator $S :=
 \sum\limits_{\underline{k} \ge \underline{0}} \dfrac{\underline{x}^{\underline{k}}}{\underline{k}!} S_{(\underline{k})} \in \gP
$
such that $S_{(\underline{k})} = w_{\underline{k}}$ for all $\underline{k} \in \Sigma$. By construction of $S$ we have:
\begin{itemize}
\item $\sigma(S) =
 \sum\limits_{\underline{k} \ge \underline{0}} \dfrac{\underline{x}^{\underline{k}}}{\underline{k}!} \sigma\bigl(S_{(\underline{k})}\bigr)$, hence
 $o(S) = \mu$.
 \item $\underline{\partial}^{\underline{k}} \circ S = w_{\underline{k}}$ for all $\underline{k} \in \Sigma$, hence $W = F \circ S$.
 \item $\underline{\partial}^{\underline{k}} \circ \sigma(S) = \sigma\bigl(w_{\underline{k}})$ for all $\underline{k} \in \Sigma$. Since  $\bigl\{\bar{w}_{\underline{k}} \, \big| \, \underline{k} \in \Sigma: |\underline{k}| = m \bigr\}$ form a basis of the vector space $W_{m+\mu}/W_{m+\mu-1}$, the vectors from the set $\bigl\{\sigma\bigl(S_{\underline{k}}\bigr) \, \big| \, \underline{k} \in \Sigma:   |\underline{k}| = m \bigr\}$ are linearly independent. According to Lemma \ref{L:regularOp}, the operator $S$ is regular.
\end{itemize}
Summing up, $S$ is a Sato operator of the vector space $W$.

\smallskip
\noindent
(2) Let $S$ and $T$ be two Sato operators of $W$ and
$$
S = \sum\limits_{\underline{k} \ge \underline{0}} \frac{\underline{x}^{\underline{k}}}{\underline{k}!} S_{(\underline{k})}
\quad \mbox{\rm respectively} \quad T = \sum\limits_{\underline{k} \ge \underline{0}} \frac{\underline{x}^{\underline{k}}}{\underline{k}!} T_{(\underline{k})}
$$
be the corresponding slice decompositions. According to Proposition \ref{P:SatoOperators}, the following statements are true:
\begin{itemize}
\item $o\bigl(S_{(\underline{k})}\bigr)  = o\bigl(T_{(\underline{k})}\bigr) =  \mu + |\underline{k}|$ for any $\underline{k} \in \Sigma$.
\item For any $m \in \NN_0$, the elements of the set $\bigl\{T_{(\underline{k})} \, \big|\, \underline{k} \in \Sigma \; \mbox{\rm such that} \; |\underline{k}| \le m \bigr\}$ form a basis of the vector space
$W_{\mu +m}$.
\end{itemize}
Therefore, for any $\underline{i} \in \Sigma$ we can find (uniquely determined) scalars $\bigl\{\gamma_{\underline{i}, \underline{k}} \in \CC \, \big| \,
\underline{k} \in \Sigma: |\underline{k}| \le |\underline{i}|\bigr\}$ such that $S_{(\underline{i})} = \sum\limits_{|\underline{k}| \le |\underline{i}|}
\gamma_{\underline{i}, \underline{k}} T_{(\underline{k})}$. Moreover, there exists at least one $\underline{l} \in \Sigma$ such that $|\underline{l}| =
|\underline{i}|$ and $\gamma_{\underline{i}, \underline{l}} \ne 0$. Now we put:
$
U:= \sum\limits_{\underline{i} \ge \underline{0}} \dfrac{\underline{x}^{\underline{i}}}{\underline{i}!} \cdot
\sum\limits_{|\underline{k}| \le |\underline{i}|} \gamma_{\underline{i}, \underline{k}} \underline{\partial}^{\underline{k}} \in \gP.
$
It is clear that $o(U) = 0$, hence $U \in \gP_-$. We claim that $S = U \cdot T$. According to Remark \ref{R:slicedecomp}, it is equivalent to the statement that
$
\underline{\partial}^{\underline{i}}\circ S = \underline{\partial}^{\underline{i}}\circ (U \cdot T) = \bigl(\underline{\partial}^{\underline{i}}\circ U\bigr)  \circ T$ for all $\underline{i} \in \Sigma.
$

\noindent
By the construction of $U$, we have:
$$
\bigl(\underline{\partial}^{\underline{i}}\circ U\bigr)  \circ T = \sum\limits_{|\underline{k}| \le |\underline{i}|} \gamma_{\underline{i}, \underline{k}} \, \underline{\partial}^{\underline{k}} \circ T = \sum\limits_{|\underline{k}| \le |\underline{i}|} \gamma_{\underline{i}, \underline{k}} T_{(\underline{k})} = S_{(\underline{i})},
$$
so $S = U\cdot T$ as asserted. In a similar way, we can find $V \in \gP_-$ such that $T = V \cdot S$. Therefore, we have: $(1- U\cdot V)\cdot S = 0$. Since
$S$ is regular, Lemma \ref{L:regularOp} implies that $U\cdot V = 1$. In a similar way, $V\cdot U = 1$, hence $U, V \in \gP_-^\ast$ as claimed.

\smallskip
\noindent
The uniqueness of the unit $U$ also follows from Lemma \ref{L:regularOp}. Theorem is proven.
\end{proof}

\begin{definition}\label{D:embSchurPair} Let $\mu \in \NN_0$ and $F = \CC[\partial_1, \dots, \partial_n]$ (viewed as the polynomial algebra). We say that $(W, A)$ is a \emph{Schur pair} of index $\mu$ if
$W \in \mathsf{Gr}_\mu(F)$ and $A \subseteq F$ is a subalgebra such that $W \cdot A = W$.
\end{definition}

\begin{theorem}\label{T:SatoSchur}
Let $(W, A)$ be a Schur pair of index $\mu \in \NN_0$ and $S$ be a Sato operator of $W$. Then the following statements are true.
\begin{enumerate}
\item For any polynomial $f \in A$, there exists a uniquely determined operator $\mathtt{L}_S(f) \in \gP$ such that
$
S \cdot f = \mathtt{L}_S(f) \cdot S.
$
Moreover, $o\bigl(\mathtt{L}_S(f)\bigr) = \deg(f)$ for any $f\in A$.
\item Next, for any $f_1, f_2 \in A$ and $\lambda_1, \lambda_2 \in C$ we have:
\begin{equation*}
\mathtt{L}_S(f_1 \cdot f_2) = \mathtt{L}_S(f_1) \cdot \mathtt{L}_S(f_2) \quad \mbox{\rm and} \quad \mathtt{L}_S(\lambda_1 f_1 + \lambda_2 f_2) = \lambda_1 \mathtt{L}_S(f_1) + \lambda_2 \mathtt{L}_S(f_2)
\end{equation*}
In other words, the map $A \lar \gP, f \mapsto \mathtt{L}_S(f)$ is a homomorphism of $\CC$--algebras, which is moreover
injective.
\item Finally, for any polynomial $f \in A$, the following diagram of $\CC$--linear maps
\begin{equation}\label{E:diagSpectModule}
\begin{array}{c}
\xymatrix{
F \ar[d]_-{-\circ \mathtt{L}_S(f)}\ar[rr]^-{-\circ S} & & W \ar[d]^-{\cdot f} \\
F  \ar[rr]^-{-\circ S} & & W
}
\end{array}
\end{equation}
is commutative.
\end{enumerate}
\end{theorem}

\begin{proof} (1) Let
$
S = \sum\limits_{\underline{k} \ge \underline{0}} \dfrac{\underline{x}^{\underline{k}}}{\underline{k}!} S_{(\underline{k})}$ be the slice decomposition
of $S$ and $f$ be an element of $A$ with  $d = \deg(f)$. Then for any $\underline{i} \in \Sigma$ we have:
$$
\underline{\partial}^{\underline{i}} \circ S = S_{(\underline{i})} \in W_{|\underline{i}|+\mu} \quad \mbox{\rm and} \quad
\underline{\partial}^{\underline{i}} \circ (S \cdot f) = S_{(\underline{i})} \cdot f \in W_{|\underline{i}|+d+\mu}.
$$
 According to Proposition \ref{P:SatoOperators},  there exists uniquely determined
 scalars $\bigl\{\gamma_{\underline{i}, \underline{k}} \in \CC \, \big| \,
\underline{k} \in \Sigma: |\underline{k}| \le |\underline{i}| +d\bigr\}$ such that
$S_{(\underline{i})} \cdot f = \sum\limits_{|\underline{k}| \le |\underline{i}| +d}
\gamma_{\underline{i}, \underline{k}} S_{(\underline{k})} \in W_{|\underline{i}|+d+\mu}.$ In these terms, we put:
\begin{equation}
\mathtt{L}_S(f) := \sum\limits_{\underline{i} \ge \underline{0}} \dfrac{\underline{x}^{\underline{i}}}{\underline{i}!} \cdot
\sum\limits_{|\underline{k}| \le |\underline{i}| + d} \gamma_{\underline{i}, \underline{k}} \, \underline{\partial}^{\underline{k}} \in \gP.
\end{equation}
Since for any $\underline{i} \in \Sigma$,
there exists at least one $\underline{l} \in \Sigma$ such that $|\underline{l}| =
|\underline{i}| +d $ and $\gamma_{\underline{i}, \underline{l}} \ne 0$, we have: $o\bigl(\mathtt{L}_S(f)\bigr) = d$.
The identity $S \cdot f = \mathtt{L}_S(f) \cdot S$ in the algebra $\gP$ follows from the fact that $\underline{\partial}^{\underline{i}} \circ (S \cdot f) =
\underline{\partial}^{\underline{i}} \circ \bigl(\mathtt{L}_S(f)  \cdot S\bigr)$ for any $\underline{i} \in \Sigma$ (which is true by the construction of $\mathtt{L}_S(f)$).
The uniqueness of $\mathtt{L}_S(f)$ follows from the regularity of $S$ and Lemma  \ref{L:regularOp}.

\smallskip
\noindent
(2) Let $f_1, f_2 \in A$. By construction, we have:
$$
S \cdot (f_1 \cdot f_2) = \bigl(\mathtt{L}_S(f_1) \cdot S\bigr) \cdot f_2 = \bigl(\mathtt{L}_S(f_1) \cdot \mathtt{L}_S(f_2)\bigr)\cdot S.
$$
Since the operator $L_S(f)$ is uniquely determined by $f \in A$, we get:  $\mathtt{L}_S(f_1) \cdot \mathtt{L}_S(f_2) = \mathtt{L}_S(f_1 \cdot f_2)$. The proof of the second statement is analogous, hence
$A \lar \gP, \, f\mapsto L_S(f)$ is indeed a homomorphism of $\CC$--algebras. For $f \ne 0$ we have: $\sigma(S \cdot f) = \sigma(S) \cdot \sigma(f) \ne 0$, hence $\mathtt{L}_S(f) \ne 0$, too.

\smallskip
\noindent
(3) The commutativity of diagram (\ref{E:diagSpectModule}) is a reformulation of the first assertion of this theorem. Note that by  Proposition \ref{P:SatoOperators}, the linear map $- \, \circ S$ is an isomorphism.
\end{proof}

\begin{remark}
Let $(W, A)$ be a Schur pair of index $\mu \in \NN_0$. According to Theorem \ref{T:SatoSchur}, any choice of Sato operator $S \in \gP$ of the vector space $W$ specifies
an injective algebra homomorphism $A \xrightarrow{\mathtt{L}_S} \gP$. However,  a Sato operator  is  determined only up to a unit of the algebra
$\gP_-$; see Theorem \ref{T:SatoAction}. If $V \in \gP_-^\ast$ and $T = V \cdot S$ is any other Sato operator of $W$, then we have: $\mathtt{L}_T = V \cdot
\mathtt{L}_S \cdot V^{-1}$. In other words,
any Schur pair defines an injective algebra homomorphism $A \stackrel{\mathtt{L}}\lar \gP$, which is unique up to an appropriate inner automorphism of the algebra $\gP$.
\end{remark}

\begin{remark} The modern algebro--geometric study  of commuting differential operators
was initiated  by Krichever \cite{Krichever77, Krichever} and investigated  by many authors in what follows. Our  work was especially  influenced by the approaches  of Mumford \cite{Mumford} and Mulase \cite{Mulase}. In particular, the higher--dimensional Sato theory  developed in this section was inspired  by  \cite{Mulase}.
The idea to enlarge the algebra of differential operators in the context of a generalized Krichever correspondence was suggested by the second--named author in \cite{Zheglov}, see also \cite{KOZ, KurkeZheglov}. The algebra $\gP$ introduced in Definition  \ref{D:AlgebraPi} deviates from its cousins   studied in \cite{Zheglov} (on the one--hand side it  is larger, on the other hand it is much more symmetric).  The construction
of commutative subalgebras of $\gP$ based on  Schur pairs $(W, A)$ can be thought as an attempt to generalize Wilson's theory of bispectral commutative subalgebras of ordinary differential operators of rank one \cite{WilsonCrelle}.
\end{remark}

\section{On the algebraic inverse scattering method in dimension two}\label{S:InverseScattering}
In this section, we are going to discuss some examples of the theory developed in the previous section in the special case $n = 2$. Let $R = \CC[z_1, z_2]$ and
$\gD = \CC\llbracket x_1, x_2\rrbracket[\partial_1, \partial_2]$, whereas
$$
\gP := \left\{\sum\limits_{k_1, k_2 \ge 0} a_{k_1, k_2}(x_1, x_2) \partial_1^{k_1} \partial_2^{k_2} \, \Big| \, \exists d \in \ZZ: \, k_1 + k_2 - \upsilon\bigl(a_{k_1, k_2}(x_1, x_2)\bigr) \le d \;\, \forall  k_1, k_2 \ge 0\right\}
$$
is   the algebra from Definition \ref{D:AlgebraPi} (here, $\upsilon\bigl(a(x_1, x_2)\bigr)$ is the valuation of the power series $a(x_1, x_2) \in \CC\llbracket x_1, x_2\rrbracket$).

\begin{lemma} Let $\bigl(\Pi, \underline{\mu}\bigr)$ be a Baker--Akhieser weighted line arrangement,
$A = A\bigl(\Pi, \underline{\mu}\bigr)$  the corresponding algebra of quasi--invariants  and
$W = P(-\xi_1 z_1 - \xi_2 z_2) \subset R$ the projective $A$--module from Theorem \ref{T:DescriptSpectModule} for an appropriate $(\xi_1, \xi_2) \in \CC^2$. Then the embedding $A \lar \gP$ determined by the  Schur pair $(W, A)$ (see Theorem \ref{T:SatoSchur}) coincides (up to an appropriate inner automorphism of $\gP$) with the embedding $\Xi(\vec{\xi})$  of Chalykh and Veselov defined by the formula  (\ref{E:mapXiTwist}).
\end{lemma}

\begin{proof} The fact that $W \in \mathsf{Gr}_\mu(R)$ (where $\mu := \sum_{\alpha \in \Pi} \mu_\alpha$) was established in Proposition \ref{P:spectralmodule}. The same Proposition implies
 that the differential operator $S$ of order $\mu$ from Lemma \ref{L:Sato}   is a Sato operator of the vector space $W$
 in the sense of Definition \ref{D:SatoGrassmannian}. Recall that for  any $f \in A$, we have the following equality
$S \cdot f(\partial_1, \partial_2)  = \bigl(\Xi(\vec{\xi})\bigr)(f) \cdot S$ in $\gD \subset \gP$;  see formula (\ref{E:conjSato}). This implies the result.
\end{proof}

\begin{remark} Let $A = A\bigl(\Pi, \underline{\mu}\bigr)$ be an  algebra of planar quasi--invariants and $W \in \mathsf{Gr}_\mu(R)$ be a finitely generated torsion free $A$--module of rank one (as usual, $\mu := \sum_{\alpha \in \Pi} \mu_\alpha$).

\smallskip
\noindent
1.~It is not true that $W$ is automatically a Cohen--Macaulay $A$--module. Indeed, let
$A := \CC[x^2, x^3, y^2, y^3]$ and $W := \bigl\langle x^2, x^3, y^3, y^4\bigr\rangle_A$. Obviously, $W$ is a finitely generated torsion free $A$--module of rank one and  $W \in \mathsf{Gr}_2(R)$.  Moreover, $A/W =  \langle 1, \bar{y}^2\rangle_\CC$ is non--zero and finite dimensional. Hence, the module $W$ is not Cohen--Macaulay (in fact, the regular module $A$ is the Macaulayfication of $W$).

\smallskip
\noindent
2.~Nevertheless, there are good reasons to focus on those Schur pairs $(W, A)$ of index $\mu$,  for which $W$ is a Cohen--Macaulay $A$--module of rank one. Assume that $(\Pi, \underline{\mu})$ is a Baker--Akhieser weighted line arrangement (see Definition \ref{D:BakerAkhieserWL}) and  $A = A(\Pi, \underline{\mu})$ is the corresponding algebra of quasi--invariants. Consider
the Rees algebra (respectively, the Rees module)
$$\widetilde{A} = \bigoplus\limits_{k = 0}^\infty A_k t^k \subset A[t] \quad \mbox{\rm respectively }\quad \widetilde{W} = \bigoplus\limits_{k = 0}^\infty W_{k+\mu} t^k \subset W[t].
$$
Let $\widetilde{X} := \mathsf{Proj}(\widetilde{A})$ (projective spectral surface), $C := V(t) \subset \widetilde{X}$ and $\kF := \mathsf{Proj}(\widetilde{W})$ (projective spectral sheaf). Then the following statements are true; see \cite[Lemma 3.3 and Lemma 3.8]{Zheglov} as well as \cite[Theorem 2.1]{KOZ}.
\begin{enumerate}
\item There exists an isomorphism of algebraic varieties $X \cong \widetilde{X} \setminus C$. In particular, we have an isomorphism of $\CC$--algebras
$A \cong \Gamma\bigl(\widetilde{X} \setminus C, \kO_{\widetilde{X}}\bigr)$.
\item Moreover, there exists a natural isomorphism of $A$--modules $W \cong \Gamma\bigl(\widetilde{X} \setminus C, \kF\bigr)$.
\item The variety $C$ is an integral projective curve. Moreover, there exists $d \in \NN$ such that $C'= dC$ is a Cartier divisor,  and $\kL = \kO_{\widetilde X}(C')$ is an ample  line bundle.
\item $\kF$ is a torsion free coherent sheaf of rank one on $X$. Moreover, we have: $$\chi\bigl(\widetilde{X}, \kF \otimes \kL^{\otimes{k}}\bigr) = \binom{kd+1}{2} \; \mbox{\rm for any}\, k \in \NN_0.$$
\end{enumerate}
Let $M_\chi$ be the moduli space of stable torsion free coherent sheaves on $X$ with Hilbert polynomial $\chi = \binom{kd+1}{2}$ with respect to the ample line bundle
$\kL$. Then $M_\chi$ is a projective variety (see e.g.~\cite[Theorem 4.3.4]{HuybrechtsLehn}) and $\kF \in M_\chi$. More precisely, let $\kU \in \Coh(X \times M_\chi)$
be a
universal family of the moduli functor $\underline{M}_\chi$. Then  any Schur pair $(W, A)$ as above defines a point $p = p(W) \in M_\chi$ such that $\kU_p:=  \kU\big|_{X \times \{p\}} \cong \kF := \mathsf{Proj}(\widetilde{W})$.

Now, let
$\gB := \mathsf{Im}\bigl(\Xi(\vec{\xi})\bigr) \subset \gD$ be the algebra defined by  (\ref{E:mapXiTwist}). Then the corresponding projective spectral sheaf $\kF$ is \emph{Cohen--Macaulay}; see \cite[Theorem 3.1]{KurkeZheglov}. Let $p \in M_\chi$ be the corresponding point. Then by
    \cite[Th\'eor\`eme 12.2.1]{EGAIV}, there exists an open neighbourhood $p \in U \subset M_\chi$ such that for any $q \in U$, the coherent sheaf $\kU_q$ is Cohen--Macaulay of rank one. As a consequence, the $A$--module $W_q:= \Gamma(X, \kU_q)$ is Cohen--Macaulay of rank one, too.
Therefore, in order to construct algebra embeddings $A \lar \gP$ arising from  Schur pairs $(W, A)$, which are  ``deformations'' of the standard Calogero--Moser system $A  \xrightarrow{\Xi(\vec{\xi})} \gD$ given by (\ref{E:mapXiTwist}), it is natural to take those  rank one torsion free $A$--modules $W \in \mathsf{Gr}_\mu(R)$, which are \emph{Cohen--Macaulay}. \qed
\end{remark}

In the remaining part of this section, we illustrate
the  ``algebraic inverse scattering method'' of Theorem \ref{T:SatoSchur} by constructing an ``isospectral deformation'' of the simplest dihedral Calogero--Moser system associated with the operator
\begin{equation}\label{E:LameSplit}
H =
\left(\frac{\partial^2}{\partial x_1^2} + \frac{\partial^2}{\partial x_2^2}\right)
- 2 \left(
\frac{1}{(x_1 - \xi_1)^2} + \frac{1}{(x_2 - \xi_2)^2}\right),
\end{equation}
where $(\xi_1, \xi_2) \in \CC^2$ is such that $\xi_1 \xi_2 \ne 0$. In this case we have:
\begin{itemize}
\item $A = A(\Pi, \underline{\mu}) = \CC[z_1^2, z_1^3, z_2^2, z_2^3]$, where $\Pi = \Lambda_2 = \left\{0, \dfrac{\pi}{2}\right\}$ and
$\underline{\mu}(0) = \underline{\mu}\Bigl(\dfrac{\pi}{2}\Bigr) = 1$.
\item It follows from Theorem \ref{T:DescriptSpectModule} that the spectral module $F$ of the corresponding Calogero--Moser system has the following description:
$$
F =  \left\{
 f \in R \left|
\begin{array}{l}
\dfrac{\partial f}{\partial z_1}(0, \rho) = \xi_1 \rho f(0, \rho) \vspace{1mm}\\
\dfrac{\partial f}{\partial z_2}(\rho, 0) = \xi_2 \rho f(\rho, 0)
\end{array}
\right.
\right\}.
$$
\item Let $K = \CC(\rho)$ and $L = K[\varepsilon]/(\varepsilon^2)$. Then the diagram (\ref{E:keydiag}) for the algebra $A$ has the following form:
\begin{equation}\label{E:Examplekeydiag}
\begin{array}{c}
\xymatrix{
A  \ar@{_{(}->}[d] \ar[rr] & & K \times K \ar@{^{(}->}[d] \\
R \ar[rr]^-{T}   & & L \times L,
}
\end{array}
\end{equation}
where $T(f) = \left(f(\rho, 0) + \varepsilon \rho \dfrac{\partial f}{\partial z_2}(\rho, 0), f(0, \rho) - \varepsilon \rho \dfrac{\partial f}{\partial z_1}(0, \rho)\right)$.
\item In terms of Theorem \ref{T:listCMlf} we have: $F \cong \kB(\vec\gamma)$, where $\vec{\gamma} = (\xi_2 \rho, -\xi_1 \rho)$.
\end{itemize}

\begin{lemma} For $(\xi_1, \xi_2) \in \CC^2$ such that $\xi_1 \xi_2 \ne 0$ and $\beta \ne \xi_1 \xi_2 \in \CC$, consider the Cohen--Macaulay $A$--module
$W_\beta := B\bigl(\vec{\gamma}_\beta\bigr) \in \CM_1^{\mathsf{lf}}(A)$, where
$$
\vec{\gamma}_\beta = \left(\frac{\xi_2^2 \rho}{\xi_2 + \beta \rho},  -\frac{\xi_1^2 \rho}{\xi_1 + \beta \rho}\right) \in \CC(\rho) \oplus \CC(\rho).
$$
Then $W_\beta \in \mathsf{Gr}_2(R)$. Moreover, $W_\beta$ is not projective over $A$ for $\beta \ne 0$.
\end{lemma}

\begin{proof}
By definition, we have:
\begin{equation}\label{E:SpaceWbeta}
W_\beta =  \left\{
 f \in R \left|
\begin{array}{l}
\dfrac{\partial f}{\partial z_1}(0, \rho) =  \dfrac{\xi_1^2 \rho}{\xi_1 + \beta \rho}  f(0, \rho)\\
\dfrac{\partial f}{\partial z_2}(\rho, 0) =  \dfrac{\xi_2^2 \rho}{\xi_2 + \beta \rho} f(\rho, 0)
\end{array}
\right.
\right\}.
\end{equation}
A lengthy but straightforward computation shows that
\begin{equation}\label{E:DeformedSchurPair}
W_\beta = \CC \cdot w + (\xi_2 + \xi_2^2 z_2 + \beta z_1) z_1^2 \CC[z_1] + (\xi_1 + \xi_1^2 z_1 + \beta z_2) z_2^2 \CC[z_2] + z_1^2 z_2^2 \CC[z_1, z_2],
\end{equation}
where
$
w = 1 + \xi_1 z_1 + \xi_2 z_2 + (\xi_1 \xi_2 - \beta) \cdot \left(z_1 z_2 + \left(\dfrac{z_1^2}{\xi_2^2} + \dfrac{z_2^2}{\xi_1^2}\right)\right).
$
 The description (\ref{E:DeformedSchurPair}) implies that $W_\beta \in \mathsf{Gr}_2(R)$. Since the rational functions $\gamma_1(\rho)$ and $\gamma_2(\rho)$ have a pole provided  $\beta  \ne 0$, Lemma \ref{L:TriplesSpecialCase} (see also the proof of Theorem \ref{T:descriptionPicard}) implies that the $A$--module  $W_\beta$ is not projective.
\end{proof}

\begin{remark} It is interesting to note that the $A$--module $W_\beta$ is actually \emph{locally free} at the point $p:= \nu(0,0) \in X$ (the ``most singular'' point of $X$),
where $\AA^2 \stackrel{\nu}\lar X$ is the normalization of the spectral surface $X$ of the algebra $A$.  Indeed, the Cohen--Macaulay $A$--module $W_\beta$ corresponds to the following object
$$
\left(R, K \oplus K,
\left(
1 + \varepsilon \dfrac{\xi_2^2 \rho}{\xi_2 + \beta \rho}, 1- \varepsilon \dfrac{\xi_1^2 \rho}{\xi_1 + \beta \rho}
\right)\right).
$$
of the category of triples $\Tri(A)$; see Theorem \ref{T:listCMlf}. As in the course of the proof of Theorem  \ref{T:descriptionPicard} one can show, that
$W_\beta$ is locally free at the point $p$ if and only if there exist $h \in \CC\llbracket z_1, z_2\rrbracket$ and $f, g \in \CC\llbrace \rho\rrbrace$ making the following two diagrams
\begin{equation*}
\begin{array}{ccc}
\xymatrix{
L \ar[d]_-{f}\ar[rr]^-{1} & & L \ar[d]^-{(1 + \varepsilon \rho \frac{\partial h}{\partial z_2}(\rho, 0)) \cdot \exp(h(\rho, 0))} \\
L  \ar[rr]_-{1 + \varepsilon \frac{\xi_2^2 \rho}{\xi_2 + \beta \rho}} & & L
}
& \;
\;
&
\xymatrix{
L \ar[d]_-{g}\ar[rr]^-{1} & & L \ar[d]^-{(1 - \varepsilon \rho \frac{\partial h}{\partial z_1} (0, \rho)) \cdot \exp(h(0, \rho))} \\
L  \ar[rr]_-{1 - \varepsilon \frac{\xi_1^2 \rho}{\xi_1 + \beta \rho}} & & L
}
\end{array}
\end{equation*}
commutative in the category of $L$--modules. To achieve this,  we have to  put $f := \exp(h(\rho, 0))$ and  $g:= \exp(h(0, \rho))$, whereas the power series $h$ has to fulfil the constraint
$$
\left(\dfrac{\partial h}{\partial z_1}(0, \rho), \dfrac{\partial h}{\partial z_2}(\rho, 0)\right) = \left(\frac{\xi_1^2}{\xi_1 + \beta \rho}, \frac{\xi_2^2}{\xi_2 + \beta \rho}\right),
$$
which  (as one can easily see) is consistent. \qed
\end{remark}

\begin{theorem}\label{L:ExampleSato} The operator $S_\beta := S_0 + \beta T\in \gP$, where
$$
S_0:= \partial_1 \partial_2 + \frac{1}{\xi_2 - x_2} \partial_1 + \frac{1}{\xi_1 - x_1} \partial_2 + \frac{1}{(\xi_1 - x_1)(\xi_2 - x_2)}
$$
and
\begin{equation*}
\begin{split}
T =  &   \frac{1}{(\xi_1 - x_1)(\xi_2 - x_2)}\left(\frac{1}{\xi_2}  \left(E_2 \partial_1 + (\xi_1-x_1) E_2\partial_1^2\right)  +
         \frac{1}{\xi_1}\left(E_1 \partial_2 + (\xi_2-x_2) E_1 \partial_2^2\right)\right)   + \\
           &  \frac{1}{(\xi_1 \xi_2 - \beta)(\xi_1 - x_1) (\xi_2 - x_2)} E_1 E_2\left(1 + \beta \left(\frac{\partial_1}{\xi_2} + \frac{\partial_2}{\xi_1}\right)\right),
\end{split}
\end{equation*}
with  $E_1, E_2 \in \gP$ defined in Example \ref{Ex:DeltaFunctionSeries}, is a Sato operator of the space $W_\beta$, given by formula (\ref{E:DeformedSchurPair}).
\end{theorem}
\begin{proof}
It is easy to see that $o(S_{\beta})=2$ and
$
\sigma (S_{\beta})=\partial_1\partial_2+ \dfrac{\beta}{\xi_2^2}\,  E_2 \partial_1^2  + \dfrac{\beta}{\xi_1^2} \, E_1\partial_2^2.
$
Let $i, j \in \NN_0$ be such that $i + j = m$. Then we have:
$$
\CC[\partial_1, \partial_2] \ni \partial_1^i \partial_2^j \xrightarrow{-\circ \, \sigma(S_\beta)}
\left\{
\begin{array}{lcl}
\partial_1^{i+1} \partial_2^{j+1} & \mbox{\rm if} &  i \cdot j \ne 0 \\
(-1)^m \dfrac{\beta}{\xi_2^2} \partial_1^{m+2} + \partial_1^{m+1} \partial_2 &  \mbox{\rm if} & (i, j) = (m, 0) \vspace{1mm}\\
(-1)^m \dfrac{\beta}{\xi_1^2} \partial_2^{m+2} + \partial_2^{m+1} \partial_1 &  \mbox{\rm if} & (i, j) = (0, m).
\end{array}
\right.
$$
By Lemma \ref{L:regularOp}, the operator $S_\beta$  is regular.

Let $\Theta_0(x_1, x_2; z_1, z_2; \xi_1, \xi_2) = z_1 z_2 + \dfrac{z_1}{\xi_2 - x_2}  + \dfrac{z_2}{\xi_1 - x_1}  + \dfrac{1}{(\xi_1 - x_1)(\xi_2 - x_2)}$. It follows from Berest's formula \cite{Berest} that
$$\Psi_0(x_1, x_2; z_1, z_2; \xi_1, \xi_2) := \Theta_0(x_1, x_2; z_1, z_2; \xi_1, \xi_2) \cdot \exp(x_1 z_1 + x_2 z_2)$$ is a Baker--Akhieser
function of the Calogero--Moser operator $H$ given by (\ref{E:LameSplit}). This implies that $S_0$ is a Sato operator of the unperturbed space $W_0$.
For general $\beta$, consider the formal power series $\Psi_\beta = \Psi_0 + \beta \overline{\Psi}$, where $\overline{\Psi}(x_1, x_2; z_1, z_2; \xi_1, \xi_2) = $
\begin{equation*}
\begin{split}
\frac{1 + \beta \left(\dfrac{z_1}{\xi_2} + \dfrac{z_2}{\xi_1}\right)}{(\xi_1 \xi_2 - \beta)(\xi_1 - x_1) (\xi_2 - x_2)} + &             \frac{1}{(\xi_1 - x_1)(\xi_2 - x_2)\xi_2}\left(\exp(x_1 z_1) z_1 + (\xi_1-x_1) \exp(x_1 z_1) z_1^2\right)+   \\
        &       \frac{1}{(\xi_1 - x_1)(\xi_2 - x_2)\xi_2}\left(\exp(x_2 z_2) z_2 + (\xi_2-x_2) \exp(x_2 z_2) z_2^2\right). \\
\end{split}
\end{equation*}

\smallskip
\noindent
A straightforward computation shows  that $\Psi_\beta$ satisfies  the same equations (\ref{E:SpaceWbeta})
$$
\left\{
\begin{array}{l}
\left.\dfrac{\partial \Psi_{\beta}}{\partial z_1}\right|_{(z_1, z_2) = (0, \rho)} = \left.
\dfrac{\xi_1^2\rho}{\xi_1+\beta \rho} \Psi_\beta\right|_{(z_1, z_2) = (0, \rho)} \\
\left.\dfrac{\partial \Psi_{\beta}}{\partial z_2}\right|_{(z_1, z_2) = (\rho, 0)} = \left.
\dfrac{\xi_2^2\rho}{\xi_2+\beta \rho} \Psi_\beta\right|_{(z_1, z_2) = (\rho, 0)}
\end{array}
\right.
$$
which define the vector space $W_{\beta}$.
Thus,
$$
W_\beta \supseteq W_{\beta}':= \CC[\partial_1, \partial_2] \circ S_\beta =
\left\langle \left. \left.
\dfrac{\partial^{p_1+p_2} \Psi_\beta}{\partial x_1^{p_1} \partial x_2^{p_2}}
\right|_{(x_1, x_2) = (0, 0)} \; \right| \; (p_1, p_2)
\in \NN_0 \times \NN_0 \right\rangle_\CC.
$$
Since both vector spaces $W_\beta$ and
$W_{\beta}'$ belong to  $\mathsf{Gr}_2(R)$, we conclude that $W_{\beta}' = W_\beta$.
\end{proof}

\begin{example} Let $A \stackrel{\mathtt{L}}\lar  \gP$ be the algebra homomorphism corresponding to the Sato operator $S_\beta$ from Theorem \ref{L:ExampleSato} for the special value
$(\xi_1, \xi_2) = (1, 1)$ and $H_\beta := \mathtt{L}(\omega)$, where  $\omega = z_1^2 + z_2^2 \in A$. Proceeding along the lines of the proof of Theorem \ref{T:SatoSchur}, we get the formula $H_\beta = H_0 + \beta (Z + Z')$, where $H_0 = H$ is the unperturbed operator (\ref{E:LameSplit}),
\begin{equation*}
\begin{split}
Z = & \left(\frac{1}{1-x_1} E_1  + \frac{1}{1-x_2} E_2\right)
\partial_1 \partial_2 + \frac{\beta}{(1-\beta)(1-x_1)(1-x_2)}E_1 E_2 (\partial_1 + \partial_2) \\
 & - \frac{1}{1-x_1} E_1 \partial_2 - \frac{1}{1-x_2} E_2 \partial_1  +\frac{2\beta}{\beta-1} E_1 E_2 -\left(
 G_2 E_1 \partial_1 + G_1 E_2 \partial_2\right)
\end{split}
\end{equation*}
and $Z'$ is some operator of negative order (it would be interesting to find a compact form of this operator).

\smallskip
\noindent
On the other hand, for any $f \in A$  of the form $f = z_1^2 z_2^2 \cdot g$ for some $g\in R$ an explicit compact form of corresponding operators can be given: it can be shown that
$$
\mathtt{L}(f) = S_\beta \cdot g(\partial_1, \partial_2) \cdot
\left(\partial_1 + \frac{1}{x_1-1}\right) \cdot  \left(\partial_2 + \frac{1}{x_2-1}\right).
$$
Finally, the formal power series $\Psi_\beta(x_1, x_2; z_1, z_2) = \Psi_\beta(x_1, x_2; z_1, z_2; 1,1)$ introduced in the course of the proof of Theorem \ref{L:ExampleSato}, is a Baker--Akhieser function of the \emph{deformed} Calogero--Moser system $\mathtt{L}(A) \subset \gP$:
for any $f \in A$ we have:
$$
\mathtt{L}(f)_{(x_1, x_2)} \diamond \Psi_\beta(x_1, x_2; z_1, z_2) = f(z_1, z_2) \cdot \Psi_\beta(x_1, x_2; z_1, z_2).
$$
\end{example}

\begin{remark} Let $A = A\bigl(\Pi, \underline{\mu}\bigr)$ be the algebra of planar quasi--invariants of Baker--Akhieser type. A description
of those $\vec\gamma \in K\bigl(\Pi, \underline{\mu}\bigr)$, for which the corresponding Cohen--Macaualy module $B(\vec{\gamma}) \in \CM_1^{\mathsf{lf}}(A)$
belong to $\mathsf{Gr}_\mu(R)$ (see Theorem \ref{T:listCMlf} for the used notations), is a non--trivial problem, which  will be studied in a future work. Theorem \ref{T:SatoAction} on existence and uniqueness of the corresponding Sato operator can be understood as an analogue of the axiomatic description (in terms of the works \cite{ChalykhFeiginVeselov1,ChalykhStyrkasVeselov}) of a Baker--Akhieser function for the corresponding deformed Calogero--Moser system $A \stackrel{\mathtt{L}}\lar  \gP$.
\end{remark}

\section{Appendix: the compactified Picard variety of  an affine  cuspidal curve}

\noindent
For any $m \in \NN$, let  $A_m := \CC[t^2, t^{2m+1}]$.  Although the (compactified) Picard variety  of an algebraic curve is a well--studied object, we were not able to find a precise reference in the literature for an explicit  description of the Picard group $\Pic(A_m)$, mentioned in Introduction. For  the reader's convenience we give  its proof below.
\begin{theorem}\label{T:Cusp} There is  an isomorphism of algebraic groups:
\begin{equation}\label{E:PicardCusp2}
\Pic(A_m) \cong K_m := \bigl(\CC[\sigma]/(\sigma^m), \circ\bigr),
\end{equation} where
$\gamma_1 \circ \gamma_2 := (\gamma_1+ \gamma_2)\cdot (1 + \sigma \gamma_1 \gamma_2)^{-1}$ for any $\gamma_1, \gamma_2 \in K_m$.
Next, let $Q$ be a torsion free $A_m$--module of rank one. Then either $Q$ is projective or there exists $m' < m$ and a projective module of rank one $Q'$ over $A_{m'}$ such that $Q$ is isomorphic to the restriction of $Q'$ on $A_m \subset A_{m'}$.
\end{theorem}
\begin{proof} To simplify the notation, we denote $A = A_m$. Then the normalization of the algebra $A$ is $R = \CC[t]$  and we have:
$$
A = \bigl\{f\in R \, \big|\, f'(0) = f'''(0) = \dots = f^{(2m-1)}(0) = 0\bigr\}.
$$
Let
$I := \mathsf{Ann}_A(R/A)$ be the conductor ideal, then we have: $I = (t^{2m})_R$ and  $K := \CC[\sigma]/(\sigma^{m}) \cong A/I$, whereas
$L:= \CC[\varepsilon]/(\varepsilon^{2m}) \cong R/I$. The canonical inclusion $A/I \subset R/I$ realizes $K$  as a subalgebra of $L$ via the identification
$\sigma = \varepsilon^{2}$.  Note that $L = K \dotplus \varepsilon K$.

Let $\TF(A)$ be the category of finitely generated torsion free $A$--modules. Similarly to Theorem \ref{T:BurbanDrozd}, we have an equivalence of categories
$$
\TF(A) \stackrel{\FF}\lar \Tri(A), \quad M \mapsto \Bigl(\bigl(R\otimes_A M\bigr)/\mathsf{tor}, K \otimes_A M, \theta_M\Bigr),
$$
where $\mathsf{tor}$ denotes the torsion part of the $R$--module $R\otimes_A M$.
Here, the category of triples $\mathsf{Tri}(A)$ is the one--dimensional prototype of its two--dimensional descendent from Definition \ref{D:triples}.
Namely, it is the full subcategory of the comma category associated with the pair of functors
$$
 \TF(R) \xrightarrow{\; \; L \otimes_{R} \,-\,\;\;  }  L-\mathsf{mod} \xleftarrow{L \otimes_{K} \,-\,} K-\mathsf{mod}.
$$
consisting of those triples $(\widetilde{M}, V, \theta)$, for which the morphism of $L$--modules (gluing map)
$$L \otimes_{K} V \stackrel{\theta}\lar L \otimes_{R} \widetilde{M}$$ is surjective and its adjoint morphism of $K$--modules
$V \lar L \otimes_{K} V \stackrel{\theta}\lar L \otimes_{R} \widetilde{M}$ is injective. The functor $\FF$ restricts
to an equivalence of categories $\Pro(A) \stackrel{\FF}\lar \Tri^{\mathsf{lf}}(A)$, where $\Pro(A)$ is the category of finitely generated projective modules
and $\Tri^{\mathsf{lf}}(A)$ consists of those triples $(\widetilde{M}, V, \theta)$, for which the gluing map  $\theta$ is an isomorphism (this condition actually implies that $V$ is a free $K$--module); see \cite[Theorem 1.3 and Theorem 3.2]{Thesis}, \cite[Theorem 16]{Survey},  \cite[Theorem 2.5]{BDNonIsol} as well as the beginning of \cite[Chapter 2]{BDNonIsol} for a survey of similar constructions.

\smallskip
\noindent
It follows from the definition of the functor $\FF$ that $\FF(P_1 \otimes_A P_2) \cong \FF(P_1) \otimes \FF(P_2)$ for any  finitely generated
projective $A$--modules $P_1$ and $P_2$, where the definition of the monoidal structure on the category $\Tri^{\mathsf{lf}}(A)$ is straightforward.

Let $P \in \Pic(A)$. Then $\FF(P) \cong \bigl(R, K, \theta\bigr)$, where the gluing map $L \stackrel{\theta}\lar  L$ can be written in
the normal form: $\theta \sim \vartheta_\gamma = 1 + \varepsilon \cdot {\gamma}$ for a uniquely determined $\gamma \in K$. Conversely, we put
$P_\gamma := \FF^{-1}\bigl(R, K, \vartheta_\gamma)$ for any $\gamma \in K$. Then we have:
$$
\FF\bigl(P_{\gamma_1} \otimes_A P_{\gamma_2}\bigr) \cong \Bigl(R, K,
\bigl((1 + \varepsilon^2 \gamma_1 \cdot \gamma_2) + \varepsilon\cdot(\gamma_1 + \gamma_2)\bigr)\Bigr) \cong
\bigl(R, K, 1 + \varepsilon (\gamma_1+ \gamma_2)\cdot (1 + \varepsilon^2 \gamma_1 \gamma_2)^{-1}
\bigr).
 $$
 This implies that $P_{\gamma_1} \otimes_A P_{\gamma_2} \cong P_{\gamma_1 \circ \gamma_2}$ for any $\gamma_1, \gamma_2 \in K$,
 proving the isomorphism (\ref{E:PicardCusp2}).  It is clear that the regular module $A = P_0$ corresponds to the triple
 $(A, K, 1)$. Hence, for any $\gamma \in K$, we have the following realization of the module $P_\gamma$:
$$
 P_\gamma \cong \Hom_A(A, P_\gamma) \cong \Hom_{\Tri(A)}\bigl((A, K, 1), (A, K, 1 + \varepsilon \gamma)\bigr) =
 \bigl\{
f \in R \; \big| \; T_{2m}^-(f) = \gamma \cdot T_{2m}^+(f)\bigr\},
 $$
 where the elements $T_{2m}^\pm(f)  \in K$ are defined by the rules: $$T_{2m}^+(f) = \sum\limits_{j=0}^{m-1} \dfrac{f^{(2j)}(0)}{(2j)!} \sigma^j\quad
  \mbox{\rm and} \quad T_{2m}^-(f) = \sum\limits_{j=0}^{m-1} \dfrac{f^{(2j+1)}(0)}{(2j+1)!} \sigma^j.
  $$
Let $\overline{\Pic}(A)$ be the set  of the isomorphism classes of finitely generated torsion free $A$--modules of rank one (compactified Picard variety) and
$Q \in \overline{\Pic}(A)$. Then $\FF(Q) \cong \bigl(R, V, \theta)$, where $V \cong K \oplus K/(\sigma^i)$ for some uniquely determined $1 \le i \le m$; see Lemma \ref{L:structureMP}. By   Proposition \ref{P:normalform}, we can transform the gluing map $\theta$ into a uniquely determined normal form
$
\theta \sim \vartheta_\gamma:= (1 + \varepsilon \gamma \,\mid \, \varepsilon^{2 (m-j)+1}),
$
where  $\gamma = \alpha_0 + \alpha_1 \varepsilon^2 + \dots + \alpha_{m-j-1} \varepsilon^{2 (m-j-1)} \in L$. This implies that
$$
Q \cong \Hom_{\Tri(A)}\bigl((A, K, 1), (A, K \oplus K/(\sigma^j), \vartheta_\gamma)\bigr) = \bigl\{
f \in R \; \big| \; T_{2(m-j)}^-(f) = \gamma \cdot T_{2(m-j)}^+(f)\bigr\}.
 $$
Hence, the vector space $Q \subset R$ is stable under the multiplication with the elements of the algebra $A_{m-j} = \CC[t^2, t^{2(m-j)+1}]$. Moreover, from the above description of the Picard group $\Pic(A)$ it follows that $Q$ is a projective module over the algebra $A_{m-j} \supset A_m = A$.
\end{proof}

\end{document}